\documentclass[12pt]{amsart}
\usepackage{a4wide}
\usepackage[latin9]{inputenc}
\usepackage{amsmath}
\usepackage{amsfonts}
\usepackage{amssymb}
\usepackage{amsthm}
\usepackage{subfigure}
\usepackage[usenames,dvipsnames]{color}
\usepackage{pstricks}
\usepackage{graphicx}
\usepackage[all]{xy}
\newtheorem{theo}{Theorem}[section]
\newtheorem{prop}[theo]{Proposition}

\newtheorem{lemm}[theo]{Lemma}

\theoremstyle{definition}
\newtheorem{def1}[theo]{Definition}
\theoremstyle{remark}
\newtheorem{rema}[theo]{Remark}
\newcommand{\Op}{\operatorname{Op}}

\newcommand{\nwc}{\newcommand}
\nwc{\eps}{\epsilon}
\nwc{\ep}{\epsilon}
\nwc{\vareps}{\varepsilon}
\nwc{\Oph}{\operatorname{Op}_\hbar}
\nwc{\la}{\langle}
\nwc{\ra}{\rangle}

\nwc{\mf}{\mathbf} 
\nwc{\blds}{\boldsymbol} 
\nwc{\ml}{\mathcal} 

\nwc{\defeq}{\stackrel{\rm{def}}{=}}

\nwc{\cE}{\ml{E}}
\nwc{\cN}{\ml{N}}
\nwc{\cO}{\ml{O}}
\nwc{\cP}{\ml{P}}
\nwc{\cU}{\ml{U}}
\nwc{\cV}{\ml{V}}
\nwc{\cW}{\ml{W}}
\nwc{\tU}{\widetilde{U}}
\nwc{\IN}{\mathbb{N}}
\nwc{\IR}{\mathbb{R}}
\nwc{\IZ}{\mathbb{Z}}
\nwc{\IC}{\mathbb{C}}
\nwc{\IT}{\mathbb{T}}
\nwc{\IS}{\mathbb{S}}
\nwc{\tP}{\widetilde{P}}
\nwc{\tPi}{\widetilde{\Pi}}
\nwc{\tV}{\widetilde{V}}
\nwc{\supp}{\operatorname{supp}}
\nwc{\rest}{\restriction}
\nwc{\x}{\mathbf{x}}
\nwc{\y}{\mathbf{y}}
\nwc{\z}{\mathbf{z}}
\nwc{\w}{\mathbf{w}}

\begin{document}

\title[Pollicott-Ruelle spectrum and Witten Laplacians]{Pollicott-Ruelle spectrum and Witten Laplacians}

\author[Nguyen Viet Dang]{Nguyen Viet Dang}

\address{Institut Camille Jordan (U.M.R. CNRS 5208), Universit\'e Claude Bernard Lyon 1, B\^atiment Braconnier, 43, boulevard du 11 novembre 1918, 
69622 Villeurbanne Cedex }

\email{dang@math.univ-lyon1.fr}

\author[Gabriel Rivi\`ere]{Gabriel Rivi\`ere}

\address{Laboratoire Paul Painlev\'e (U.M.R. CNRS 8524), U.F.R. de Math\'ematiques, Universit\'e Lille 1, 59655 Villeneuve d'Ascq Cedex, France}

\email{gabriel.riviere@math.univ-lille1.fr}

\begin{abstract} We study the asymptotic behaviour of eigenvalues and eigenmodes of the Witten Laplacian on a smooth compact Riemannian 
manifold without boundary. We show that they converge to the Pollicott-Ruelle spectrum of the corresponding 
gradient flow acting on appropriate anisotropic Sobolev spaces. In particular, our results relate the approach of Laudenbach and Harvey--Lawson to Morse theory using currents, 
which was discussed in previous work of the authors, and 
Witten's point of view based on semiclassical analysis and tunneling. As an application of our methods, 
we also construct a natural family of quasimodes satisfying the 
Witten-Helffer-Sj\"ostrand tunneling formulas and the Fukaya conjecture on Witten deformation of the wedge product.

\end{abstract}

\maketitle

\section{Introduction}\label{s:intro}

Let $M$ be a smooth ($\ml{C}^{\infty}$), compact, oriented, boundaryless manifold of dimension $n\geq 1$. Let $f:M\rightarrow \IR$ 
be a smooth Morse function whose set of critical points is denoted by $\text{Crit}(f)$. In~\cite{Wi82}, Witten introduced the following semiclassical deformation of the de Rham coboundary operator:
$$\forall \hbar>0,\quad d_{f,\hbar}:=e^{-\frac{f}{\hbar}}de^{\frac{f}{\hbar}}=d+\frac{df}{\hbar}\wedge:\Omega^{\bullet}(M)\rightarrow\Omega^{\bullet+1}(M)$$
where $\Omega^\bullet(M)$ denotes smooth differential forms on $M$.
Then, fixing a smooth Riemannian metric $g$ on $M$, he considered the adjoint of this operator with respect to the induced scalar product on 
the space of $L^2$ forms
$L^2(M,\Lambda(T^*M))$~:
$$\forall \hbar>0,\quad d_{f,\hbar}^*=
d^*+\frac{\iota_{V_f}}{\hbar}:\Omega^{\bullet}(M)\rightarrow\Omega^{\bullet-1}(M),$$
where $V_f$ is the gradient vector field associated with the pair $(f,g)$, i.e. the unique vector field satisfying
$$\forall x\in M,\quad df(x)=g_x(V_f(x),.).$$
The operator $D_\hbar=\left( d_{f,\hbar}+d_{f,\hbar}^* \right)$ is the analog of a Dirac operator and its square $D^2_\hbar$ is usually 
defined to be the Witten Laplacian~\cite{ZhangWitten}. In the present paper, we take a different convention and we choose
to rescale $D^2_\hbar$ by a factor $\frac{\hbar}{2}$. Hence, following~\cite{Bi86}, the \textbf{Witten Laplacian} is defined as
$$\boxed{W_{f,\hbar}:=\frac{\hbar}{2}\left(d_{f,\hbar}d_{f,\hbar}^*+d_{f,\hbar}^*d_{f,\hbar}\right).}$$ 
This defines a selfadjoint, elliptic operator whose principal symbol coincides with the principal 
symbol of the Hodge--De Rham Laplace operator acting on forms. It has 
a discrete spectrum on $L^2(M,\Lambda^k(T^*M))$ that we denote, for every $0\leq k\leq n$, by
$$0\leq \lambda_1^{(k)}(\hbar)\leq\lambda_2^{(k)}(\hbar)\leq\ldots\leq\lambda_j^{(k)}(\hbar)\rightarrow+\infty\ \text{as}\ j\rightarrow+\infty.$$
It follows from the works of Witten~\cite{Wi82} and Helffer-Sj\"ostrand~\cite{HeSj85} that there exists a constant $\eps_0>0$ such that, 
for every $0\leq k\leq n$ and for 
$\hbar>0$ small enough, there are exactly $c_k(f)$ eigenvalues inside the interval $[0,\eps_0]$, where $c_k(f)$ is the number of critical 
points of index $k$ -- see e.g. the recent proof of Michel and Zworski in~\cite[Prop.~1]{MiZw17}.

The purpose of the present work is to describe the convergence of the spectrum (meaning both eigenvalues and eigenmodes) of the Witten Laplacian. This
will be achieved by using microlocal techniques that were developped in the context of dynamical systems 
questions~\cite{dang2016spectral, dangrivieremorsesmale1}. Note that \emph{part of these results} 
could probably be obtained by more 
classical methods in the spirit of the works of Helffer-Sj\"ostrand~\cite{HeSj85} and Simon~\cite{Si83} on harmonic oscillators. 
We refer the reader to the book of Helffer and 
Nier~\cite{HeNi05} for a detailed account of the state of the art on these aspects.  
For instance, Frenkel, Losev and Nekrasov~\cite{frenkel2011instantons} did very explicit
computations of the Witten's spectrum for the case of the height function on the sphere, and they implicitely
connect the spectrum of the Witten Laplacian to a dynamical spectrum as we shall do it here.
They also give a strategy to derive asymptotic expansions
for dynamical correlators of holomorphic gradient flows acting on compact K\"ahler manifolds.
Yet, unlike~\cite{frenkel2011instantons}, we attack the problem from the dynamical viewpoint rather than from 
the semiclassical perspective. Also, we work in the $C^\infty$ case instead of the
compact K\"ahler case and we make use of tools from microlocal analysis 
to replace tools from complex geometry.

In fact, the main purpose of the present work is to propose an approach to these problems having a 
more dynamical flavour than these references. In some sense, 
this point of view shares some similarities with Bismut's 
approach to the study of the Witten Laplacian~\cite{Bi86}. 
However, we stress that our study of the limit operator is more inspired by the study of the so-called transfer operators in dynamical 
systems~\cite{Go15, Ba16, Zw17}. This dynamical perspective allows us to make some explicit connection between the spectrum 
of the Witten Laplacian and the dynamical results from~\cite{Lau92, HaLa01, dang2016spectral, DaRi17b}.

\section{Main results}

\subsection{Semiclassical versus dynamical convergence and a question by Harvey--Lawson.}

In order to illustrate our results, we let $\varphi_f^t$ be the flow induced by 
the gradient vector field $V_f$. Then, among other results, we shall prove the following Theorem:
\\
\\
\fbox{
\begin{minipage}{0.94\textwidth} 
\begin{theo}[Semiclassical versus dynamical convergence]\label{t:maintheo-harveylawsonwitten} Let $f$ be a smooth Morse function and $g$ be a smooth Riemannian metric such that $V_f$ is $\ml{C}^1$-linearizable near every critical point
and satisfies the Smale transversality assumption. 

Then, there exists $\eps_0>0$ small enough such that, for every $0\leq k\leq n$, and for every $(\psi_1,\psi_2)\in\Omega^k(M)\times\Omega^{n-k}(M)$,
\begin{equation}\label{e:semiclassicaleqlargetime}
\forall 0<\eps\leq\eps_0,\quad\lim_{\hbar\rightarrow 0^+}\int_M\mathbf{1}_{[0,\eps]}\left(W_{f,\hbar}^{(k)}\right)
 \left(e^{-\frac{f}{\hbar}}\psi_1\right)\wedge\left(e^{\frac{f}{\hbar}}\psi_2\right)=\lim_{t\rightarrow+\infty}\int_M\varphi_f^{-t*}(\psi_1)\wedge\psi_2
\end{equation} 
 where $\mathbf{1}_{[0,\eps]}\left(W_{f,\hbar}^{(k)}\right)$ is the spectral projector
 on $[0,\varepsilon]$
for the self--adjoint elliptic operator $W_{f,\hbar}^{(k)}$ and both sides have well--defined limits.
\end{theo}
\end{minipage}
}
\\
\begin{rema}
The Smale transversality assumption means that the stable 
and unstable manifolds satisfy some transversality conditions~\cite{Sm60} -- see appendix~\ref{a:order-function} for a brief reminder. 
Recall that, given a Morse function $f$, this property is satisfied 
by a dense open set of Riemannian metrics thanks to the Kupka-Smale Theorem~\cite{Ku63, Sm63}. The hypothesis of being $\ml{C}^1$-linearizable 
near every critical point
means that, near every $a$ in $\text{Crit}(f)$, one can find a $\ml{C}^1$-chart such that the vector field can be written locally as 
$V_f(x)=L_f(a)x\partial_x,$ where $L_f(a)$ is the unique (symmetric) matrix satisfying $d^2f(a)=g_a(L_f(a).,.).$ 
By fixing a finite number of nonresonance conditions on the eigenvalues of $L_f(a)$, 
the Sternberg-Chen Theorem~\cite{Ne69} ensures that, for a given $f$, one can find an open and dense subset of Riemannian metrics 
satisfying this property. 
\end{rema}

As we shall see in our proof, the rank of the operator $\mathbf{1}_{[0,\eps_0]}\left(-W_{f,\hbar}^{(k)}\right)$ is equal to $c_k(f)$. Hence, 
on the one hand, the spectral projector $\mathbf{1}_{[0,\eps_0]}\left(-W_{f,\hbar}^{(k)}\right)$ appearing on the left-hand
side of equation (\ref{e:semiclassicaleqlargetime}) projects on the above mentionned
sum of eigenspaces with small eigenvalues of the Witten Laplacian. 
On the other hand, the term on the 
right-hand side of (\ref{e:semiclassicaleqlargetime}) 
only involves the classical flow generated by the gradient flow 
and the limit of this quantity was described by Harvey 
and Lawson~\cite{HaLa01} in the case of certain metrics 
adapted to the Morse function -- see also~\cite{Mi15, dang2016spectral} for extension 
of this result. This dynamical limit can in fact be expressed in terms of certain currents carried by the stable and 
unstable manifolds of the gradient flow which were first constructed by Laudenbach~\cite{Lau92} -- see Theorem~\ref{t:maintheo-current} below.
In a nutshell, our Theorem identifies a certain semiclassical limit of scalar product of quasimodes
for the Witten Laplacian with
a large time limit of some dynamical correlation for the gradient flow which converges 
to equilibrium~:
\begin{eqnarray*}
\boxed{\lim_{\hbar\rightarrow 0^+}\underset{\text{Quantum object}}{\underbrace{\left\langle\mathbf{1}_{[0,\eps]}\left(W_{f,\hbar}^{(k)}\right)
 \left(e^{-\frac{f}{\hbar}}\psi_1\right),e^{\frac{f}{\hbar}}\psi_2\right\rangle_{L^2}}}=\lim_{t\rightarrow+\infty}
 \underset{\text{Dynamical object}}{\underbrace{\left\la \varphi_f^{-t*}(\psi_1),\psi_2\right\ra_{L^2}}} ,}
\end{eqnarray*}
for every $(\psi_1, \psi_2)\in\Omega^k(M)^2$. From this point of view, this Theorem gives some insights on a question raised by Harvey and Lawson in~\cite[Intro.]{HaLa00} who asked about the 
connection between their approach to Morse theory and Witten's one -- see also~\cite[Chap.~9]{CdV} for related questions.

\subsection{Asymptotics of Witten eigenvalues and zeros of the Ruelle zeta function}

Theorem~\ref{t:maintheo-harveylawsonwitten} is the Corollary of much more general results. In order to state these generalized statements, 
we define the \textbf{dynamical Ruelle determinant}~\cite[p.~65-68]{Ba16}, for every $0\leq k\leq n$,
$$\boxed{\zeta_{R}^{(k)}(z):=\exp\left(-\sum_{l=1}^{+\infty}\frac{e^{-lz}}{l}
\sum_{a\in\text{Crit}(f)}\frac{\text{Tr}\left(\Lambda^k\left(d\varphi_f^{-l}(a)\right)\right)}{\left|\text{det}\left(\text{Id}-d\varphi_f^{-l}(a)\right)\right|}\right).}$$
This quantity is related to the notion of distributional determinants~\cite[p.~313]{GuSt90}. This function is well defined 
for $\text{Re}(z)$ large enough, and, from appendix~\ref{a:holomorphic}, it has an holomorphic extension 
to $\IC$. The zeros of this holomorphic extension can be explicitely described in terms of the Lyapunov exponents
of the flow $\varphi_f^t$ at the critical points of $f$~:
$$\forall a\in\text{Crit}(f),\quad\chi_1(a)\leq\ldots\leq\chi_r(a)<0<\chi_{r+1}(a)\leq\ldots\leq \chi_n(a)$$
where the numbers $(\chi_j(a))_{j=1}^n$
are the eigenvalues of $L_f(a)$ which is the unique 
(symmetric) matrix satisfying $d^2f(a)=g_a(L_f(a).,.)$. 
Our first result of spectral theoretic nature reads as follows~:
\\
\\
\fbox{
\begin{minipage}{0.94\textwidth}
\begin{theo}[Convergence of Witten eigenvalues]\label{t:maintheo-eigenvalues} Suppose that the assumptions of Theorem~\ref{t:maintheo-harveylawsonwitten} are satisfied. 
Then, the following holds, for every $0\leq k\leq n$:
\begin{enumerate}
 \item for every $j\geq 1$, $-\lambda_j^{(k)}(\hbar)$ converges as 
$\hbar\rightarrow 0^+$ to a zero of the dynamical Ruelle determinant $\zeta_{R}^{(k)}(z)$, 
 \item conversely, any zero of the dynamical Ruelle determinant $\zeta_{R}^{(k)}(z)$ is the limit of a 
 sequence $(-\lambda_j^{(k)}(\hbar))_{\hbar\rightarrow 0^+}.$
\end{enumerate}
\end{theo}
\end{minipage}
}
\\
\\
Following Appendix~\ref{a:holomorphic}, this Theorem shows that the Witten eigenvalues converge, as $\hbar\rightarrow 0$, to 
integer combinations of the Lyapunov exponents. The result is in fact more precise than what we stated here for the sake of simplicity. 
In fact, we will also prove the convergence of the spectral projectors of the Witten Laplacian -- see Theorem~\ref{t:maintheo-proj} below. 
Hence, we will verify that, for every zero of degree $N\geq 1$ of the dynamical Ruelle determinant, one can find exactly $N$ eigenvalues of the 
Witten Laplacian converging to this zero. For every $0\leq k\leq n$, $\zeta_{R}^{(k)}(z)$ has a zero 
of multiplicity $c_k(f)$ according to Appendix~\ref{a:holomorphic}. Hence, we recover the above results on the bottom of the 
spectrum of the Witten Laplacian. These small eigenvalues are known to be exponentially small in terms of $\hbar$~\cite{HeSj85, HeNi05, MiZw17} 
but our proof does not say a priori anything on this aspect of the Witten-Helffer-Sj\"ostrand result. Contrary to the upcoming results, we 
emphasize that this Theorem could be recovered from the techniques of~\cite{HeSj85, Si83} but the dynamical interpretation will come out more clearly 
from our analysis.

\subsection{Correlation function of gradient vector fields}

Before stating our results on Witten Laplacian in their full generality, we would like to describe a dynamical question which 
was studied in great details in~\cite{dang2016spectral} in the case of Morse-Smale gradient flows -- see also~\cite{BaTs08, dyatlov2016pollicott} 
for earlier related results. A classical question in dynamical systems is to study the asymptotic behaviour of the correlation function
$$\boxed{\forall 0\leq k\leq n,\ \forall (\psi_1,\psi_2)\in\Omega^k(M)\times\Omega^{n-k}(M),\ C_{\psi_1,\psi_2}(t):=\int_M\varphi_f^{-t*}(\psi_1)\wedge\psi_2,}$$
which already appeared in the statement of Theorem~\ref{t:maintheo-harveylawsonwitten}. In the case of certain Riemannian metrics adapted 
to the Morse function, the limit of this quantity was identified by Harvey 
and Lawson in~\cite{HaLa00, HaLa01}. In~\cite{dang2016spectral}, we extended this result to more general Riemannian metrics and we also gave the 
full asymptotic expansion of $C_{\psi_1,\psi_2}(t)$ -- see also~\cite{dangrivieremorsesmale1, DaRi17b} for extension of 
these results to Morse--Smale flows admitting periodic orbits. Reference~\cite{HaLa01} was entirely based on the theory of 
currents \`a la Federer which is essentially geometric measure theoretic. On the contrary, our approach
was inspired by the recent developments in hyperbolic dynamical systems related to the study of the 
so-called Pollicott--Ruelle spectrum~\cite{Po85, Ru87a} 
following the pioneering works of Blank, Keller and Liverani~\cite{BlKeLi02}. We refer for instance 
to the book of Baladi~\cite{Ba16} or to the survey article of Gou\"ezel~\cite{Go15} for detailed accounts and 
references related to these dynamical questions. More specifically, we used a microlocal approach to this spectral problems 
on currents. Closely related to our microlocal approach to this problem, we also refer to the survey 
of Zworski for the relation of these questions with scattering theory~\cite{Zw17}. In the sequel, we will always denote by 
$\mathcal{D}^{\prime,k}(M)$ the space of De Rham currents of degree $k$.

Following~\cite{Po85, Ru87a}, it will be simpler to consider the Laplace transform of $C_{\psi_1,\psi_2}$, i.e. for $\text{Re}(z)$ large enough, 
$$\boxed{\hat{C}_{\psi_1,\psi_2}(z):=\int_0^{+\infty}e^{-tz} C_{\psi_1,\psi_2}(t)dt.}$$
One of the consequence of the results from~\cite{dangrivieremorsesmale1} is that this function admits a meromorphic extension to $\IC$ under 
the assumptions of Theorem~\ref{t:maintheo-eigenvalues}. In~\cite{dang2016spectral, DaRi17b}, we also gave an explicit description of the poles 
and residues of this function under $\ml{C}^{\infty}$-linearization properties of the vector field $V_f$
which were for instance verified as soon as infinitely many nonresonance assumptions are satisfied. Our next theorem extends part of these 
results without this smooth linearization assumption:

\begin{theo}[Correlation spectrum]\label{t:maintheo-flow} Suppose that the assumptions of Theorem~\ref{t:maintheo-harveylawsonwitten} are satisfied. Then, 
for every $0\leq k\leq n$ and for every $(\psi_1,\psi_2)\in\Omega^k(M)\times\Omega^{n-k}(M)$, the function $\hat{C}_{\psi_1,\psi_2}(z)$
admits a meromorphic extension to $\IC$ whose poles are contained inside 
$$\ml{R}_k:=\left\{z_0:\zeta_R^{(k)}(z_0)=0\right\}.$$
 Moreover, for every $z_0\in\IR$, there exists a continuous linear 
 map
 $$\pi_{z_0}^{(k)}:\Omega^k(M)\rightarrow \ml{D}^{\prime k}(M),$$
 whose rank is equal to the multiplicity\footnote{ When $z_0\notin\ml{R}_k$, one has $\pi_{z_0}^{(k)}=0$.} of $z_0$ as a 
 zero of $\zeta_R^{(k)}(z)$ and such that, for every 
 $(\psi_1,\psi_2)\in\Omega^k(M)\times\Omega^{n-k}(M)$, the residue of $\hat{C}_{\psi_1,\psi_2}$ at $z_0$ satisfies 
 $$\boxed{\operatorname{Res}_{z_0}\left(\hat{C}_{\psi_1,\psi_2} \right)=\int_M\pi_{z_0}^{(k)}(\psi_1)\wedge\psi_2.}$$
\end{theo}

The poles of this meromorphic function are referred to as the Pollicott--Ruelle resonances or the correlation spectrum of the gradient flow. They describe 
in some sense the structure of the long time 
dynamics of the gradient flow. As in~\cite{dang2016spectral, DaRi17b}, we recover the poles of the meromorphic extension and the rank of 
their residues. Yet, compared to these references, the proof given in the present work does not say anything on the order of the poles and 
on the currents generating the range of the residues. Recall that the strategy in~\cite{dang2016spectral, DaRi17b} was to construct \emph{explicitely} 
some natural families of currents associated with each poles. Here, we shall follow slightly more standard arguments 
from~\cite{BaTs08, GiLiPo13, DyZw13} which consists into relating $\hat{C}_{\psi_1,\psi_2}(z)$ with some flat determinant 
associated with the Lie derivative of $V_f$. 

Among other things, this result shows that the correlation spectrum depends only on the Lyapunov exponents of the flow. In other 
words, the global correlation spectrum of a gradient flow depends only on the $0$-jet of the metric at the critical points. Thus, 
this result gives in some sense some insights on Bowen's first problem in~\cite{Bow} from the perspective of the global dynamic of the flow instead of 
the local one. Note that, if we were interested in the local dynamics near critical points, this could be recovered from the results 
of Baladi and Tsujii in~\cite{BaTs08}. Still regarding Bowen's question, we will also verify that the range of the residues are 
generated by families of currents carried by the unstable manifolds of the gradient flows -- see section~\ref{s:Pollicott-Ruelle}.

Besides this support property, we do not say 
much things on the structure of these residues except in the case $z_0=0$.  In that case, we can say a little bit more. 
For that purpose, given any critical point $a$ of $f$ of index $k$, we introduce its unstable manifold
$$W^u(a):=\left\{x\in M:\lim_{t\rightarrow -\infty}\varphi^t_f(x)=a\right\}.$$
Recall from the works of Smale~\cite{Sm60} that this defines a smooth embedded submanifold of $M$ whose dimension is equal to $n-k$ and 
whose closure is the union of unstable manifolds. Then, one has~:

\begin{theo}[Vacuum states]\label{t:maintheo-current} Suppose that the assumptions of Theorem~\ref{t:maintheo-harveylawsonwitten} are satisfied 
and fix $0\leq k\leq n$. Then, for every $a\in\operatorname{Crit}(f)$ of index $k$, there exists $(U_a,S_a)$ in $\ml{D}^{\prime k}(M)\times \ml{D}^{\prime n-k}(M)$ 
such that the support of $U_a$ is equal to $\overline{W^u(a)}$ and such that
 $$\ml{L}_{V_f}(U_a)=0.$$
Moreover, for every 
$$0<\Lambda<\operatorname{min}\left\{|\chi_j(a)|:1\leq j\leq n,\ a\in\operatorname{Crit}(f)\right\},$$
one has, for every $(\psi_1,\psi_2)\in\Omega^k(M)\times\Omega^{n-k}(M)$,
 $$\boxed{\int_M\varphi_f^{-t*}(\psi_1)\wedge\psi_2=\sum_{a:\operatorname{dim} W^u(a)=n-k}\int_M \psi_1\wedge S_a\int_M U_a\wedge\psi_2 
 +\ml{O}_{\psi_1,\psi_2}(e^{-\Lambda t}).}$$
\end{theo}

In particular, the pole at $z_0=0$ is simple. In the case of a locally flat metric, this result was proved by 
Harvey and Lawson~\cite{HaLa01}, except for the size of the remainder 
term. In~\cite{dang2016spectral}, we showed how to prove this Theorem when the flow satisfies some more general (smooth) linearization properties than the ones 
appearing in~\cite{Lau92, HaLa01}. Here, we will extend the argument from~\cite{dang2016spectral} 
to show that Theorem~\ref{t:maintheo-current} remains true under the rather weak assumptions of 
Theorem~\ref{t:maintheo-harveylawsonwitten}. As we shall see in Lemma~\ref{r:current}, the currents $U_a$ coincides with 
the current of integration $[W^u(a)]$ when restricted to the open set $M-\partial W^u(a)$ with 
$\partial W^u(a)=\overline{W^u(a)}-W^u(a)$. Hence, the previous Theorem shows how one can extend 
$[W^u(a)]$ into a globally defined current which still satisfies the transport equation $\ml{L}_{V_f}(U_a)=0.$ This 
extension was done by Laudenbach in the case of 
locally flat metrics in~\cite{Lau92} by analyzing carefully the structure of the boundary $\partial W^u(a)$. Here, we 
make this extension for more general metrics via a spectral method and the analysis of the structure of the boundary 
is in some sense hidden in the construction of our spectral framework~\cite{dang2016spectral, dangrivieremorsesmale1}. We emphasize that 
Laudenbach's construction shows that these extensions are currents of \emph{finite mass} while our method does not say a priori 
anything on that aspect.

Finally, following~\cite{DaRi17c}, we note that 
$((U_a)_{a\in\text{Crit}(f)},d)$ generate a finite dimensional complex which is nothing else but the Thom-Smale-Witten 
complex~\cite[Eq.~(2.2)]{Wi82}. In section~\ref{s:proofs}, we will explain how to prove some topological statement which complements what was proved in~\cite{DaRi17c}, namely~:
\\
\fbox{
\begin{minipage}{0.94\textwidth}
\begin{theo}[Witten's instanton formula]\label{t:maintheo-instanton} Suppose that the assumptions of Theorem~\ref{t:maintheo-harveylawsonwitten} are 
satisfied. Then, for every pair of critical points $(a,b)$ with $\operatorname{ind}(b)=\operatorname{ind}(a)+1$, there 
exists\footnote{An explicit expression is given in paragraph~\ref{ss:proof-ruelle}.} $n_{ab}\in\IZ$ such that
\begin{eqnarray}\label{instantonformula-intro}
\forall a\in\operatorname{Crit}(f),\quad dU_a=\sum_{b:\operatorname{ind}(b)=\operatorname{ind}(a)+1}n_{ab}U_b
\end{eqnarray}
where $n_{ab}$ counts algebraically
the number of instantons connecting
$a$ and $b$. 

In particular, the
complex $((U_a)_{a\in \text{Crit}(f)},d)$ can be defined over $\mathbb{Z}$ and realizes in the space of currents
the Morse homology over $\mathbb{Z}$.
\end{theo}
\end{minipage}
}
\\
\\
In the case of locally flat metrics, the relation between the formula for the boundary of unstable currents and Witten's instanton formula follows 
for instance from the works of Laudenbach in~\cite{Lau92}.
In~\cite{dang2016spectral,DaRi17c}, we were able to prove
that the complex $((U_a)_{a\in\text{Crit}(f)},d)$ forms a subcomplex of the De Rham complex of currents
which was quasi--isomorphic to the De Rham complex
$(\Omega^\bullet(M),d)$ but we worked in 
the (co)homology with coefficients in $\mathbb{R}$.
The instanton formula~\eqref{instantonformula-intro} allows us to actually consider
$((U_a)_{a\in\text{Crit}(f)},d)$ as a 
$\mathbb{Z}$--module and 
directly relate it to the famous Morse complex 
defined over $\mathbb{Z}$ appearing in the litterature 
whose
integral homology groups contain more 
information than working with real coefficients~\cite[p.~620]{GMS}.

\subsection{Asymptotics of Witten eigenmodes}

So far, we mostly discussed the asymptotic properties of the eigenvalues of the Witten Laplacian. Yet, as we already pointed, 
our results are more precise. We shall now describe the asymptotic 
behaviour of its eigenmodes as $\hbar$ goes to $0$ in terms of the residues of the correlation function. More precisely, one has~:
\\
\\
\fbox{
\begin{minipage}{0.94\textwidth}
\begin{theo}[Convergence of Witten spectral projectors]\label{t:maintheo-proj} Suppose that 
the assumptions of Theorem~\ref{t:maintheo-harveylawsonwitten} are satisfied. Let $0\leq k\leq n$ and let 
$z_0\in \IR$. Then, there exists $\eps_0>0$ small enough such that, for every $(\psi_1,\psi_2)\in\Omega^k(M)\times\Omega^{n-k}(M)$,
$$\forall0<\eps\leq\eps_0,\quad \lim_{\hbar\rightarrow 0^+}\int_M\mathbf{1}_{[z_0-\eps,z_0+\eps]}\left(-W_{f,\hbar}^{(k)}\right)
 \left(e^{-\frac{f}{\hbar}}\psi_1\right)\wedge\left(e^{\frac{f}{\hbar}}\psi_2\right)=\int_M\pi_{z_0}^{(k)}(\psi_1)\wedge\psi_2.$$
\end{theo}
\end{minipage}
}
\\
\\
This Theorem tells us that, up to renormalization by $e^{\frac{f}{\hbar}}$, the spectral projectors of the Witten Laplacian converges to 
the residues of the dynamical correlation function. Note that, in the case $z_0=0$, one has in fact, from 
Theorems~\ref{t:maintheo-flow} and~\ref{t:maintheo-current},
$$\boxed{\lim_{\hbar\rightarrow 0^+}\int_M\mathbf{1}_{[0,\eps_0]}\left(W_{f,\hbar}^{(k)}\right)
 \left(e^{-\frac{f}{\hbar}}\psi_1\right)\wedge\left(e^{\frac{f}{\hbar}}\psi_2\right)=
 \sum_{a:\operatorname{dim} W^u(a)=n-k}\left(\int_M \psi_1\wedge S_a\right)\left(\int_M U_a\wedge\psi_2\right).}$$
In that manner, we recover that the bottom of the spectrum of 
the Witten Laplacian acting on $k$ forms is asymptotically of dimension equal to the number 
of critical points of index $k$. Still considering the bottom of the spectrum of 
the Witten Laplacian, we can define an analogue of Theorem~\ref{t:maintheo-instanton} at 
the semiclassical level. For that purpose, we need to introduce analogues of the Helffer-Sj\"ostrand WKB states 
for the Witten Laplacian~\cite{HeSj85, HeNi05}. We fix $\eps_0>0$ small enough so that the range of $\mathbf{1}_{[0,\eps_0)}(W_{f,\hbar}^{(k)})$ 
in every degree $k$ is equal to the number of critical points 
of index $k$. Then, for $\hbar>0$ small enough, we define the following \textbf{WKB states}:
\begin{equation}\label{e:WKBquasimodes}
\boxed{U_a(\hbar):=\mathbf{1}_{[0,\eps_0)}(W_{f,\hbar}^{(k)})\left(e^{\frac{f(a)-f}{\hbar}} U_a\right)\in\Omega^{k}(M),}
\end{equation}
where $k$ is the index of the point $a$. We will show in Proposition~\ref{p:wavefrontset} that, for every critical point $a$,
the sequence $(e^{\frac{f-f(a)}{\hbar}}U_a(\hbar))_{\hbar\rightarrow 0^+}$ converges 
to $U_a$ in $\ml{D}'(M)$. As a corollary of Theorem~\ref{t:maintheo-instanton}, these WKB 
states verify the following exact tunneling formulas~:
\\
\\
\fbox{
\begin{minipage}{0.94\textwidth}
\begin{theo}[Witten--Helffer--Sj\"ostrand tunneling formula]\label{t:maintheo-tunneling}
Suppose that the assumptions of Theorem~\ref{t:maintheo-harveylawsonwitten} are satisfied. Then, for every critical point $a$ of $f$ and for 
every $\hbar>0$ small enough,
\begin{eqnarray*}
d_{f,\hbar}U_a(\hbar)&=&\sum_{b:\operatorname{ind}(b)=\operatorname{ind}(a)+1}n_{ab}e^{\frac{f(a)-f(b)}{\hbar}}U_{b}(\hbar),
\end{eqnarray*}
where $n_{ab}$ is the same integer as in Theorem~\ref{t:maintheo-instanton}.
\end{theo}
\end{minipage}
}
\\
\\
The formula we obtain may seem slightly different from the one appearing in~\cite[Eq.~(3.27)]{HeSj85}. This is mostly due to the choice of normalization, and we 
will compare more precisely our quasimodes with the ones of Helffer-Sj\"ostrand in Section~\ref{s:helffer-sjostrand}.

\subsection{A conjecture by Fukaya}

Before stating Fukaya's conjecture on Witten Laplacians, we start with a brief overview of the context in which they appear. These problems 
are related to symplectic topology and Morse theory, and it goes without saying that the reader
is strongly advised to consult the original papers of Fukaya for further details~\cite{Fu93,Fu97,Fu05}. 
In symplectic topology, one would like to attach invariants
to symplectic manifolds in particular to Lagrangian
submanifolds since they play a central role in symplectic geometry. 
Motivated by Arnold's conjectures on
Lagrangian intersections, 
Floer constructed an infinite dimensional generalization of Morse homology 
named Lagrangian Floer homology which is the homology of some chain complex 
$(CF(L_0,L_1),\partial)$ associated to pairs of Lagrangians
$(L_0,L_1)$ generated by the intersection points of $L_0$ and $L_1$~\cite[Def.~1.4, Th.~1.5]{Auroux}. Then, 
for several Lagrangians satisfying precise geometric assumptions, 
it is possible to define some product
operations on the corresponding Floer complexes~\cite[Sect.~2]{Auroux} 
and the collection of all these 
operations and the relations among them form a so called \emph{$A_\infty$ structure} 
first described by Fukaya. 
The important result is that the $A_\infty$ structure, 
up to some natural equivalence relation, 
does not depend on the various choices that were made to define it
in the same way as
the Hodge--De Rham cohomology theory of a compact Riemannian manifold
does not depend on the choice of metric $g$.

Let us briefly motivate these notions of $A_\infty$ structures 
by discussing a simple example.
On a given smooth compact manifold $M$, consider the 
De Rham complex $(\Omega^\bullet(M),d)$ with the
corresponding De Rham cohomology $H^\bullet(M)=\text{Ker}(d)/\text{Ran}(d)$. 
From classical results of 
differential topology if
$N$ is another smooth manifold diffeomorphic to
$M$, then we have 
a quasi-isomorphism
between $(\Omega^\bullet(M),d)$ and
$(\Omega^\bullet(N),d)$ which implies that
the corresponding cohomologies are isomorphic $H^\bullet(M)\simeq H^\bullet(N)$.
This means that the space of cocycles is an invariant of our space.
However, there are manifolds which have the
same cohomology groups, hence the same homology groups 
by Poincar\'e duality, and which are not homeomorphic
hence (co)homology is not enough to specify the topology 
of a given manifold.
To get more invariants, we should be able to give some informations on relations among (co)cycles.
Recall that
$(\Omega^\bullet(M),d,\wedge)$ is a differential graded algebra
where the algebra structure comes from the wedge product $\wedge$, and the fact 
that $\wedge$ satisfies the Leibniz rule
w.r.t. the differential $d$ readily implies that
$\wedge : \Omega^\bullet(M)\times \Omega^\bullet(M)\mapsto \Omega^\bullet(M)$ induces
a bilinear map on cohomology $\mathfrak{m}_2:H^\bullet(M)\times H^\bullet(M)\mapsto H^\bullet(M)$ called the cup--product. 
By Poincar\'e duality, this operation on cohomology geometrically encodes
intersection theoretic informations among cycles and gives
more informations than the usual (co)homology groups.
Algebras of $A_\infty$ type are far reaching generalizations of differential graded algebras
where the wedge product is replaced by a sequence
of $k$--multilinear products for all $k\geqslant 2$ with 
relations among them generalizing the Leibniz rule~\cite{Val}.

In perfect analogy with symplectic topology, Fukaya introduced $A_{\infty}$ structures in Morse theory~\cite[Chapter 1]{Fu93}. 
In that case, the role of Lagrangian pairs $(L_0,L_1)$ is played by a pair of smooth functions $(f_0,f_1)$ such that $f_0-f_1$ is Morse. Note 
that it is not a priori possible to endow the Morse complex with the wedge product $\wedge$ of currents since currents carried by 
the same unstable manifold cannot be intersected because of the lack of transversality. The idea is to perturb the Morse functions to create 
transversality. Thus, we should deal with several pairs of smooth functions. In that context, Fukaya formulated 
conjectures~\cite[Sect.~4.2]{Fu05} related to the $A_{\infty}$ structure associated with the Witten Laplacian. He predicted 
that the WKB states of Helffer and Sj\"ostrand should verify more 
general asymptotic formulas than the tunneling formulas associated with the action of the twisted coboundary operator $d_{f,\hbar}$~\cite[Conj.~4.1 and 4.2]{Fu05}.
Indeed, as in the above discussion, after twisting the De Rham coboundary operator $d$, 
the next natural idea is to find some twisted version of Cartan's exterior
product $\wedge$ and see if one can find some analogue
of the tunneling formulas for twisted products.
At the semiclassical limit $\hbar\rightarrow 0^+$, Fukaya conjectured that one should recover 
the Morse theoretical analogue of the wedge product modulo some exponential corrections related to 
\emph{disc instantons}~\cite{Get12,Get13}. Hence, as for the coboundary operator, the cup product in 
Morse cohomology would appear in the asymptotics of the Helffer-Sj\"ostrand WKB states. This conjecture was recently solved by Chan--Leung--Ma in~\cite{ChLeMa14} via WKB 
approximation methods.

As a last application of our analysis, we would like to show that our families of WKB states $(U(\hbar))_{\hbar\rightarrow 0^+}$ also verifies 
Fukaya's asymptotic formula in the case of products of order $2$~\cite[Conj.~4.1]{Fu05}. This approach could probably be adapted to treat the case of higher 
order products. Yet, this would be at the expense of a more delicate combinatorial work that would be beyond the scope of the present article and we shall discuss 
this elsewhere. Let us now describe precisely the framework of Fukaya's conjecture for products of order $2$ which corresponds to the classical wedge product $\wedge$ -- see also paragraph~\ref{ss:fukaya} 
for more details. 
Consider three smooth real valued functions $(f_1,f_2,f_3)$ on $M$, and define their differences~:
$$f_{12}=f_2-f_1,f_{23}=f_3-f_2,f_{31}=f_1-f_3.$$
We assume the functions $(f_{12},f_{23},f_{31})$ to be Morse. To every such pair $(ij)$, we associate a Riemannian metric
$g_{ij}$, and we make the assumption that the corresponding gradient vector fields $V_{f_{ij}}$ satisfy the Morse-Smale property\footnote{This means that the $V_{f_{ij}}$ verify the Smale 
transversality assumptions.} and that they are $\ml{C}^1$-linearizable. 
In particular, they are amenable to the above spectral analysis, and, for any critical point $a_{ij}$ of $f_{ij}$ and for every $0<\hbar\leq 1$, we can associate 
a WKB state $U_{a_{ij}}(\hbar)$. From elliptic regularity, these are smooth differential forms on $M$ and Fukaya predicted that the integral
$$\int_M U_{a_{12}}(\hbar)\wedge U_{a_{23}}(\hbar)\wedge U_{a_{31}}(\hbar)$$
has a nice asymptotic formula whenever the intersection $W^u(a_{12})\cap W^u(a_{23})\cap W^u(a_{31})$ consists of finitely many points. Note that we implicitely supposed that
\begin{equation}
\dim W^u(a_{12})+\dim W^u(a_{23})+\dim W^u(a_{31})=2n.
\end{equation}
As was already mentioned, such an asymptotic should hold when appropriate transversality assumptions are verified by $(f_{12},f_{23},f_{31})$. 
More precisely, Fukaya's conjecture concerns triple of Morse functions such that, for every
 $$(a_{12},a_{23},a_{31})\in \text{Crit}(f_{12}) \times \text{Crit}(f_{23})\times 
\text{Crit}(f_{31}),$$
one has, for every $x\in W^u(a_{12})\cap W^u(a_{23})\cap W^u(a_{31})$,
$$T_xM=(T_x W^u(a_{12})\cap T_x W^u(a_{23}))+T_x W^u(a_{31}),$$
and similarly for any permutation of $(12, 23, 31)$. In that case, we say that the vector fields $(V_{f_{12}},V_{f_{23}},V_{f_{31}})$ 
satisfy the \textbf{generalized Morse-Smale property}. Our last result 
shows that the WKB states we have constructed verify Fukaya's conjecture~:
\\
\\
\fbox{
\begin{minipage}{0.94\textwidth}
\begin{theo}[Fukaya's instanton formula]\label{t:fukaya} Using the above notations, let $(V_{f_{12}},V_{f_{23}},V_{f_{31}})$ be a family of Morse-Smale 
gradient vector fields which are 
$\ml{C}^{1}$-linearizable, and which verify the generalized Morse-Smale property. Then, for every 
$$(a_{12},a_{23},a_{31})\in \operatorname{Crit}(f_{12}) \times \operatorname{Crit}(f_{23})\times 
\operatorname{Crit}(f_{31}),$$
 such that $\dim W^u(a_{12})+\dim W^u(a_{23})+\dim W^u(a_{31})=2n,$
 $$U_{a_{12}}\wedge U_{a_{23}}\wedge U_{a_{31}}$$
 defines an element of $\ml{D}^{\prime n}(M)$ satisfying
 $\int_MU_{a_{12}}\wedge U_{a_{23}}\wedge U_{a_{31}}\in\IZ$,
 and
 $$\lim_{\hbar\rightarrow 0^+} e^{-\frac{f_{12}(a_{12})+f_{23}(a_{23})+f_{31}(a_{31})}{\hbar}}
 \int_M U_{a_{12}}(\hbar)\wedge U_{a_{23}}(\hbar)\wedge U_{a_{31}}(\hbar)=
\int_MU_{a_{12}}\wedge U_{a_{23}}\wedge U_{a_{31}}.$$
\end{theo}
\end{minipage}
}
\\
\\
The integers $\int_MU_{a_{12}}\wedge U_{a_{23}}\wedge U_{a_{31}}$ defined by triple intersections of unstable currents
have a deep geometrical meaning. On the one hand, they 
actually \emph{count} the number of Y shaped gradient flow trees~\cite[p.~8]{Fu93} as described in subsection~\ref{r:geometry-fukaya}. On the other hand,
they give a representation of the
cup--product in Morse cohomology at the cochain level as explained in subsection~\ref{ss:fukaya}.
We emphasize that there are two parts in this Theorem. The first one says that the product of the three currents $U_{a_{ij}}$ 
is well defined. This will follow from the fact that we can control the wavefront sets of these currents. The second part is, up to some normalization factors, 
the asymptotic formula conjectured by Fukaya for the WKB states of Helffer--Sj\"ostrand~\cite{HeSj85}. Here, 
our states are constructed in a slightly different manner. Yet, they belong to the same 
eigenspaces as the ones from~\cite{HeSj85} -- see Section~\ref{s:helffer-sjostrand} for a comparison. Finally, we note that, going through the details of the proof, we would get that the rate 
of convergence is in fact of order $\mathcal{O}(\hbar)$. However, for simplicity of exposition, we do not keep track of this aspect in our argument.

\subsection{Lagrangian intersections}
\label{ss:lagintersec} 
We would like to recall 
the nice symplectic interpretation of 
the exponential prefactors appearing in Theorems~\ref{t:maintheo-tunneling} 
and~\ref{t:fukaya}. 
Let us start with the case of Theorem~\ref{t:maintheo-tunneling} where we 
only consider a pair $(f,0)$ of functions where $f-0=f$ is Morse. We can consider the pair of 
exact Lagrangian submanifolds
$\Lambda_f:=\left\{(x;d_xf):x\in M\right\}\subset T^*M$ and $\underline{0}\subset T^*M$\footnote{$\Lambda_f$ is the graph of $df$ whereas the zero section is the graph of $0$.}.
Given $(a,b)$ in $\text{Crit}(f)^2$, we can define a disc $D$ whose boundary $\partial D\subset \Lambda_f\cup \underline{0}$ 
is a $2$-gon made of
the
union of two smooth curves $e_1$ and $e_2$ joining $a$ and $b$, $e_1$ and $e_2$
are respectively contained in the Lagrangian submanifolds $\Lambda_f$ and $\underline{0}$.
 Denote by $\theta$ the Liouville one form and 
by $\omega=d\theta$ the canonical symplectic form on $T^*M$. Then, by the Stokes formula
$$\int_D\omega=\int_{\partial D}\theta=\int_{e_1}df=f(a)-f(b),$$
where we choose $e_1$ to be oriented from $b$ to $a$. Hence, the exponents in the asymptotic formula of Theorem~\ref{t:maintheo-tunneling} can 
be interpreted as the symplectic area of the disc $D$ defined by $\Lambda_f$ and the zero section of 
$T^*M$\footnote{Hence the name \emph{disc instantons}.}. 
A similar 
interpretation holds in the case of Theorem~\ref{t:fukaya} and the picture goes as follows. We consider a triangle ($3$-gon) 
$T$ inside $T^*M$ with vertices $(v_{12},v_{23},v_{31})\in (T^*M)^3$ whose projection on $M$ are equal to 
$(a_{12},a_{23},a_{31})$. The edges $(e_1,e_2,e_3)$ are contained in the three Lagrangian submanifolds $\Lambda_{f_1}$, $\Lambda_{f_2}$ and 
$\Lambda_{f_3}$. To go from $v_{23}$ to $v_{12}$, we follow some smooth curve $e_2$ in 
$\Lambda_{f_2}$, from $v_{31}$ to $v_{23}$ follow some line $e_3$ in 
$\Lambda_{f_3}$ and from $v_{12}$ to $v_{31}$, we follow some line $e_1$ in 
$\Lambda_{f_1}$. These three lines define the triangle $T$ and we can compute
$$\int_{T}\theta=\sum_{j=1}^3\int_{e_j}df_j=-f_1(a_{12})+f_1(a_{31})-f_2(a_{23})+f_2(a_{12})-f_3(a_{31})+f_3(a_{23}),$$
which is (up to the sign) the term appearing in the exponential factor of Theorem~\ref{t:fukaya}. Note that
the triangle $T$ does not
necessarily bound a disk.

\subsection{Convergence of the Witten Laplacian to the gradient vector field}\label{ss:conjugation}

The \textbf{key observation} to prove our different results is the following exact relation~\cite[equation (3.6)]{frenkel2011instantons}: 
\begin{equation}\label{e:conjugation}
\boxed{e^{\frac{f}{\hbar}}W_{f,\hbar}e^{-\frac{f}{\hbar}}=\ml{L}_{V_f}+\frac{\hbar\Delta_g}{2},}
\end{equation}
where $\ml{L}_{V_f}$ is the Lie derivative along the gradient vector field and $\Delta_g=dd^*+d^*d\geq 0$ is the Laplace Beltrami operator. Indeed, one has~\cite[equations (3.4), (3.5)]{frenkel2011instantons}~:
$$ e^{\frac{f}{\hbar}}W_{f,\hbar}e^{-\frac{f}{\hbar}}= \frac{\hbar}{2} e^{\frac{f}{\hbar}}\left(d_{f,\hbar}+d_{f,\hbar}^*\right)^2 e^{-\frac{f}{\hbar}} = 
 \frac{\hbar}{2}\left(d+d_{2f,\hbar}^*\right)^2 =  \frac{\hbar}{2}\left(d d_{2f,\hbar}^*+d_{2f,\hbar}^* d\right),$$
which yields~\eqref{e:conjugation} thanks to the Cartan formula. Hence, the rough idea 
is to prove that the spectrum of the Witten Laplacian converges to the spectrum of the Lie derivative, provided that it makes sense. This kind 
of strategy was used by Frenkel, Losev and Nekrasov~\cite{frenkel2011instantons} to compute the spectrum of $\ml{L}_{V_f}$ in the case of the 
height function on the canonical $2$-sphere. However, their strategy is completely different from ours. Frenkel, 
Losev and Nekrasov computed explicitely the spectrum of the Witten Laplacian and show how to take the limit as 
$\hbar\rightarrow 0^+$. Here, we will instead compute the spectrum of the limit operator explicitely and show
without explicit computations why the spectrum of the Witten 
Laplacian should converge to the limit spectrum. 
In particular, our proof makes no explicit use of the classical results of Helffer 
and Sj\"ostrand on the Witten Laplacian~\cite{HeSj85}.

Our first step will be to define an appropriate functional framework where one can study the spectrum of $\ml{L}_{V_f}+\frac{\hbar\Delta_g}{2}$ for 
$0\leq \hbar\leq \hbar_0$. Recall that, following the microlocal strategy of Faure and Sj\"ostrand for the study of the analytical properties 
of hyperbolic dynamical systems~\cite{FS}, we constructed in~\cite{dang2016spectral} some families of anisotropic Sobolev spaces 
$\ml{H}^{m_{\Lambda}}(M)$ indexed by a parameter $\Lambda>0$ and such that~:
$$-\ml{L}_{V_f}:\ml{H}^{m_{\Lambda}}(M)\rightarrow \ml{H}^{m_{\Lambda}}(M)$$
has discrete spectrum on the half plane $\{\text{Re}(z)>-\Lambda\}$. This spectrum is intrinsic and it turns out to be the correlation 
spectrum appearing in Theorem~\ref{t:maintheo-flow}. For an Anosov vector field $V$, Dyatlov and Zworski proved that the correlation spectrum is in 
fact the limit of the spectrum of an operator of the form $\ml{L}_{V}+\frac{\hbar\Delta_g}{2}$~\cite{DZsto} -- see also~\cite{BlKeLi02, FRS08, Zw15, Dr16} 
for related questions. We will thus show how to adapt the strategy of Dyatlov and Zworski to our framework. 
It means that we will prove that the family of operators 
$$\left(\widehat{H}_\hbar:=-\ml{L}_{V_f}-\frac{\hbar\Delta_g}{2}\right)_{\hbar\in [0,+\infty)}$$ 
has nice spectral properties on the anisotropic Sobolev spaces
$\ml{H}^{m_\Lambda}(M)$ constructed in~\cite{dang2016spectral}. This will be the object of 
section~\ref{s:anisotropic}. Once these properties will be established, we will verify in which sense the spectrum of $\widehat{H}_\hbar$
converges to the spectrum of $\widehat{H}_0$ in the semiclassical limit $\hbar\rightarrow 0^+$ -- see section~\ref{s:convergence} for 
details. In~\cite{dang2016spectral}, we computed explicitely the spectrum of $\widehat{H}_0$ on these anisotropic Sobolev spaces. Under some 
(generic) smooth linearization properties, we obtained an explicit description of the eigenvalues and a rather explicit description of the 
generalized eigenmodes. Here, we generalize the results of~\cite{dang2016spectral} by relaxing these smoothness assumptions and 
by computing the spectrum under the more general assumptions of 
Theorem~\ref{t:maintheo-harveylawsonwitten}. For that purpose, we will make
crucial use of some earlier results of Baladi and Tsujii on hyperbolic 
diffeomorphisms~\cite{BaTs08} in order to compute the eigenvalues. Compared with~\cite{dang2016spectral, DaRi17b}, we will however get a 
somewhat less precise description of the corresponding eigenmodes. This will be achieved in section~\ref{s:Pollicott-Ruelle}. Then, in 
section~\ref{s:proofs}, we combine these results to prove Theorems~\ref{t:maintheo-harveylawsonwitten} to~\ref{t:maintheo-tunneling}. 
In section~\ref{s:fukaya}, we describe the wavefront set of the generalized eigenmodes and we show how to use this information to prove Theorem~\ref{t:fukaya}. 
Finally, in section~\ref{s:helffer-sjostrand}, we briefly compare our quasimodes with the ones appearing in~\cite{HeSj85}.

The article ends with three appendices. Appendix~\ref{a:order-function} contains a brief reminder on Morse-Smale gradient flows 
and on the dynamical construction from~\cite{dang2016spectral}. Appendix~\ref{a:holomorphic} shows how to prove the holomorphic extension of the 
dynamical Ruelle determinant in our framework. Appendix~\ref{a:lemma} contains the proof of a technical lemma needed for our analysis of 
wavefront sets.

\subsection{Conventions} In all this article, $\varphi_f^t$ is a Morse-Smale gradient flow which is $\ml{C}^1$-linearizable acting on a smooth, compact, 
oriented and boundaryless manifold of dimension $n\geq 1$.

\subsection*{Acknowledgements} When we gave talks on our previous work~\cite{dang2016spectral}, a recurrent question was about the 
relation between our results and the Witten Laplacian. We thank the many people who asked us this question as it motivated us to explore 
seriously this problem. We also warmly thank Alexis Drouot, Colin Guillarmou, Fr\'ed\'eric H\'erau and Benoit Merlet for useful discussions 
related to this work. The second author 
is partially supported by the Agence Nationale de la Recherche through the Labex CEMPI (ANR-11-LABX-0007-01) and the 
ANR project GERASIC (ANR-13-BS01-0007-01).



\section{Anisotropic Sobolev spaces and Pollicott-Ruelle spectrum}\label{s:anisotropic}

In~\cite{dang2016spectral, dangrivieremorsesmale1}, we have shown how one can build a proper spectral theory for the operator $-\ml{L}_{V_f}$. 
In other words, we constructed some anisotropic Sobolev spaces of currents on which we could prove that the spectrum of $-\ml{L}_{V_f}$ is discrete in a certain 
half-plane $\text{Re}(z)>-C_0$. The corresponding discrete eigenvalues are intrinsic and are the so-called Pollicott-Ruelle resonances. 
Our construction was based on a microlocal approach that was initiated by Faure and Sj\"ostrand in the framework of Anosov flows~\cite{FS} 
and further developped by Dyatlov and Zworski in~\cite{DyZw13}. As was already explained in paragraph~\ref{ss:conjugation}, we will try to relate the 
spectrum of the Witten Laplacian to the spectrum of $-\ml{L}_{V_f}$ by the use of the relation~\eqref{e:conjugation}. Hence, our first 
step will be to show that our construction from~\cite{dang2016spectral} can be adapted to fit (in a uniform manner) with the operator~: 
$$\boxed{\widehat{H}_\hbar:=-\ml{L}_{V_f}-\frac{\hbar\Delta_g}{2}.}$$ 
In the case of Anosov flows, this perturbation argument was introduced by Dyatlov and Zworski in~\cite{DZsto}. As their spectral construction is slightly 
different from the one of Faure and Sj\"ostrand in~\cite{FS} and as our proof of the meromorphic extension of $\hat{C}_{\psi_1,\psi_2}$ in~\cite{dang2016spectral} 
is closer to~\cite{FS} than to~\cite{DZsto}, we need to slightly revisit some of the arguments given in~\cite{FS,dang2016spectral} to fit the framework 
of~\cite{DZsto}. This is the purpose of this section where we will recall the definition of anisotropic Sobolev spaces and of the corresponding 
Pollicott-Ruelle resonances.

\subsection{Anisotropic Sobolev spaces}

\subsubsection{The order function.}
 We start with the definition of an anisotropic Sobolev space in the terminology of~\cite{FRS08, FS}. 
First of all, such spaces require the existence of an order function $m_{N_0,N_1}(x;\xi)$ in $\ml{C}^{\infty}(T^*M)$ with bounded derivatives which 
is adapted to the dynamics of $\varphi_f^t$. The important 
properties of this function are recalled in appendix~\ref{a:order-function}, more specifically Lemma~\ref{l:escape-function}. 
Let us just mention that $N_0$ and $N_1$ are two large positive parameters which indicate the regularity of the distribution along the 
unstable and stable manifolds of the flow. Once this function is constructed, we set
$$A_{N_0,N_1}(x;\xi):=\exp G_{N_0,N_1}^0(x;\xi),$$
where $G_{N_0,N_1}^0(x;\xi):=m_{N_0,N_1}(x;\xi)\ln(1+\|\xi\|_x^2)$ belongs to the class of symbols $S^{\eps}(T^*M)$ for every $\eps>0$. We shall 
denote this property by $G_{N_0,N_1}^0\in S^{+0}(T^*M)$. We emphasize that the construction below will require to deal with symbols of variable 
order $m_{N_0,N_1}$ whose pseudodifferential calculus was described in~\cite[Appendix]{FRS08}.

\subsubsection{Anisotropic Sobolev currents.}

Let us now define the spaces we shall work with. Let $0\leq k\leq n$. We consider the vector bundle $\Lambda^k(T^*M)\mapsto M$ of exterior $k$ forms. 
We define $\mathbf{A}_{N_0,N_1}^{(k)}(x;\xi):=A_{N_0,N_1}(x;\xi)\textbf{Id}$ which belongs to $\Gamma(T^*M,\text{End}(\Lambda^k(T^*M)))$
which is the product of the weight $A_{N_0,N_1}\in C^\infty(T^*M)$ with
the canonical identity section $\textbf{Id}$ of the endomorphism bundle $\text{End}(\Lambda^k(T^*M))\mapsto M$. 
We fix the canonical inner product $\la,\ra_{g^*}^{(k)}$ on $\Lambda^k(T^*M)$ induced by the metric $g$ on $M$. 
This allows to define the Hilbert space $L^2(M,\Lambda^k(T^*M))$ and to introduce an 
anisotropic Sobolev space of currents by setting
$$\ml{H}^{m_{N_0,N_1}}_k(M)=\Op(\mathbf{A}_{N_0,N_1}^{(k)})^{-1}L^2(M,\Lambda^k(T^*M)),$$
where $\Op(\mathbf{A}_{N_0,N_1}^{(k)})$ is a formally selfadjoint pseudodifferential operator with principal symbol $\mathbf{A}_m^{(k)}$. We refer 
to~\cite[App.~C.1]{DyZw13} for a brief reminder of pseudodifferential operators 
with values in vector bundles -- see also~\cite{BDIP02}. In particular, adapting the proof of~\cite[Cor.~4]{FRS08} to the vector bundle valued framework, 
one can verify that $\mathbf{A}_{N_0,N_1}^{(k)}$ is an elliptic symbol, and 
thus $\Op(\mathbf{A}_{N_0,N_1}^{(k)})$ can be chosen to be invertible. Mimicking the proofs of~\cite{FRS08}, we can deduce some properties of 
these spaces of currents. First of all, they are endowed with a Hilbert 
structure inherited from the $L^2$-structure on $M$. The space $$\ml{H}^{m_{N_0,N_1}}_k(M)^{\prime}=\Op(\mathbf{A}_{N_0,N_1}^{(k)})L^2(M,\Lambda^k(T^*M))$$
is the topological dual of $\ml{H}^{m_{N_0,N_1}}_k(M)$ which is in fact reflexive. We also note that the space $\ml{H}_k^{m_{N_0,N_1}}(M)$ can be 
identified with $\ml{H}_0^{m_{N_0,N_1}}(M)\otimes_{\ml{C}^{\infty}(M)}\Omega^k(M)$. Finally, one has
$$\Omega^k(M) \subset \ml{H}_k^{m_{N_0,N_1}}(M) \subset \mathcal{D}^{\prime,k}(M),$$ 
where the injections are continuous.

\subsubsection{Hodge star and duality for anisotropic Sobolev currents.}

Recall now that the Hodge star operator~\cite[Part I.4]{BDIP02} is the unique isomorphism 
$\star_k :\Lambda^k(T^*M)\rightarrow \Lambda^{n-k}(T^*M)$ such that, for every $\psi_1$ in $\Omega^k(M)$ and $\psi_2$ in $\Omega^{n-k}(M)$, 
$$\int_M\psi_1\wedge\psi_2=\int_M\la \psi_1,\star_k^{-1}\psi_2\ra_{g^*(x)}^{(k)}\omega_g(x),$$
where $\la .,.\ra_{g^*(x)}^{(k)}$ is the induced Riemannian metric on $\Lambda^k(T^*M)$ and where $\omega_g$ is the Riemannian volume form. In particular, 
$\star_k$ induces an isomorphism from $\ml{H}^{m_{N_0,N_1}}_k(M)^{\prime}$ to $\ml{H}^{-m_{N_0,N_1}}_{n-k}(M)$, 
whose Hilbert structure is given by the scalar product
$$(\psi_1,\psi_2)\in\ml{H}^{-m_{N_0,N_1}}_{n-k}(M)^2\mapsto \la \star_k^{-1}\psi_1,\star_k^{-1}\psi_2\ra_{\ml{H}_k^m(M)'}.$$
Thus, the topological dual of $\ml{H}_k^{m_{N_0,N_1}}(M)$ can be identified with $\ml{H}^{-m_{N_0,N_1}}_{n-k}(M)$, where, for every $\psi_1$ in $\Omega^k(M)$ 
and $\psi_2$ in $\Omega^{n-k}(M)$, one has the following duality relation:
$$\la\psi_1,\psi_2\ra_{\ml{H}_k^m\times\ml{H}_{n-k}^{-m}}=\int_{M}\psi_1\wedge\psi_2
=\la \Op(\mathbf{A}_{N_0,N_1}^{(k)})\psi_1,\Op(\mathbf{A}_{N_0,N_1}^{(k)})^{-1}\star_k^{-1}\overline{\psi_2}\ra_{L^2}=\la\psi_1,\star_k^{-1}\psi_2
\ra_{\ml{H}_k^m\times(\ml{H}_{k}^m)^{\prime}}.$$


\subsection{Pollicott-Ruelle resonances} Now that we have defined the appropriate spaces, 
we have to explain the spectral properties of 
$i\widehat{H}_\hbar:=-i\ml{L}_{V_f}-\frac{i\hbar\Delta_g}{2}$ 
acting on $\ml{H}_k^{m_{N_0,N_1}}(M)$. Following Faure and Sj\"ostrand in~\cite{FS}, we
introduce the following conjugation of the operator $i\widehat{H}_\hbar$~:
\begin{equation}
\widehat{P}_\hbar=\widehat{A}_N (i\widehat{H}_\hbar)\widehat{A}_N^{-1},
\end{equation}
where we use the notation $\widehat{A}_N$ instead of $\Op(\mathbf{A}^{(k)}_{N_0,N_1})$ for simplicity. 
For similar reasons, we also omit the dependence in $k$.


In any case, the spectral properties of $\widehat{H}_{\hbar}$ acting on the anisotropic Sobolev space $\ml{H}_k^{m_{N_0,N_1}}(M)$ are the same as those of the operator 
$\widehat{P}_{\hbar}$ acting on the simpler Hilbert space $L^2(M,\Lambda^k(T^*M))$. We will now apply the strategy of Faure and Sj\"ostrand 
in order to derive some spectral properties of the above operators. 
Along the way, 
we keep track of the dependence in $\hbar$ 
which is needed to apply the arguments from Dyatlov and Zworski~\cite{DZsto} 
on the convergence of the spectrum. 
In all this section, we follow closely the proofs 
from~\cite[Sect.~3]{FS} and we emphasize the differences.

\subsubsection{The conjugation argument} The first step in Faure-Sj\"ostrand's proof consists in computing the symbol of the operator 
$\widehat{P}_{\hbar}$. Starting from this operator, we separate it in two terms
\begin{eqnarray*}
{P}_\hbar
=\widehat{Q}_{1}+\hbar\widehat{Q}_{2}=-
\underset{\text{hyperbolic part}}{\underbrace{\widehat{A}_N i\left(\mathcal{L}_{V_f}\right)\widehat{A}_N^{-1}}}-
\underset{\text{elliptic perturbation}}{\underbrace{\hbar
\widehat{A}_N i\left(\frac{\Delta_g}{2}\right)\widehat{A}_N^{-1}}}
\end{eqnarray*}
which we will treat separately for the sake of simplicity. The key ingredient of~\cite{FS} is the following lemma~\cite[Lemma~3.2]{FS}:
\begin{lemm}\label{l:pseudo} The operator $\widehat{Q}_{1}+i\ml{L}_{V_f}$
is a pseudodifferential operator in $\Psi^{+0}(M,\Lambda^k(T^*M))$ whose symbol in any given system of coordinates is of the form
$$i(X_{H_f}.G_N^0)(x;\xi)\mathbf{Id}+\ml{O}(S^0)+\ml{O}_m(S^{-1+0}),$$
where $X_{H_f}$ is the Hamiltonian vector field
generating the characteristic flow of $\mathcal{L}_{V_f}$ in $T^*M$ whose definition is recalled in Appendix~\ref{a:order-function}. The operator $\widehat{Q}_{2}$ is a pseudodifferential operator in $\Psi^{2}(M,\Lambda^k(T^*M))$ whose symbol in any given system of coordinates is of the form
$$-i\frac{\|\xi\|_x^2}{2}\mathbf{Id}+\ml{O}_m(S^{1+0}).$$
\end{lemm}

In this Lemma, the notation $\ml{O}(.)$ means that the remainder is independent of the order function $m_{N_0,N_1}$, while the 
notation $\ml{O}_m(.)$ means that it depends on $m_{N_0,N_1}$. As all the principal symbols are proportional to $\mathbf{Id}_{\Lambda^k(T^*M)}$, 
the proof of~\cite{FS} can be adapted almost verbatim to encompass the case of 
a general vector bundle and of the term corresponding to the Laplace-Beltrami operator. Hence, we shall omit it and refer to this reference 
for a detailed proof. In particular, this Lemma says 
that $\widehat{Q}_1$ is an element in $\Psi^1(M,\Lambda^k(T^*M))$. We can consider that it acts on the 
domain $\Omega^k(M)$ which is dense in $L^2(M,\Lambda^k(T^*M))$. In particular, according to~\cite[Lemma~A.1]{FS}, 
it has a unique closed extension as 
an unbounded operator on $L^2(M,\Lambda^k(T^*M)).$ For $\widehat{Q}_2$, 
this property comes from the fact that the symbol is elliptic~\cite[Chapter 13 p.~125]{Wong}. The same properties also hold for the adjoint operator.

\subsubsection{The anti-adjoint part of the operator and its symbol}\label{sss:meromorphic}

We now verify that this operator has a discrete spectrum in a certain half-plane in $\IC$. Following~\cite{FS}, this will be done by arguments 
from analytic Fredholm theory. Recall that the strategy consists in studying the properties of the anti-adjoint part of the operator
\begin{equation}\label{e:antiadjoint}\widehat{P}_{\text{Im}}(\hbar):=\frac{i}{2}\left(\widehat{P}_{\hbar}^*-\widehat{P}_{\hbar}\right)=\frac{i}{2}\left(\widehat{Q}_1^*-\widehat{Q}_1\right)
+\frac{i\hbar}{2}\left(\widehat{Q}_2^*-\widehat{Q}_2\right),\end{equation}
whose symbol is, according to Lemma~\ref{l:pseudo}, given in any given system of coordinates by
$$P_{\text{Im}}(x;\xi)=X_{H_f}.G_N^0(x;\xi)\mathbf{Id}+\ml{O}(S^0)+\ml{O}_m(S^{-1+0})-\hbar\frac{\|\xi\|^2_x}{2}+\ml{O}_m(S^{1+0}).$$
We already note that, according to Lemma~\ref{l:escape-function}, there exists some constant $C>0$ independent of $m_{N_0,N_1}$ such that
\begin{equation}\label{e:upperbound-squareroot}X_{H_f}.G_N^0(x;\xi)\mathbf{Id}+\ml{O}(S^0)+\ml{O}_m(S^{-1+0})\leq (-C_N +C)\mathbf{Id}+\ml{O}_m(S^{-1+0}),\end{equation}
where $C_N$ is the constant defined by Lemma~\ref{l:escape-function}.

\subsubsection{Squareroot Lemma}
At this stage of the proof, we will need to pay some attention to the dependence of our arguments 
in $\hbar>0$ in order to be able to apply the arguments from~\cite{DZsto}. 
To handle this additional difficulty compared to~\cite{FS}, 
we need the following 
Lemma which yields a decomposition of some 
operator with positive symbol
as some operator square plus a small error which can be made
of trace class~:
\begin{lemm}[Squareroot Lemma]\label{l:squareroot} Let $L\geq 0$. Then, there exists $\widehat{a_L}\in \Psi^{+0}$ such that
\begin{eqnarray}
\frac{i}{2}\left(\widehat{Q}_{1}^*-\widehat{Q}_{1}\right)-\frac{C-C_N}{3}=-\widehat{a_L}^*\widehat{a_L}\mod(\Psi^{-L-1+0}). 
\end{eqnarray}
 Similarly, there exists $\widehat{b_L}\in \Psi^{1}$ such that
\begin{eqnarray}
\frac{i}{2}\left(\widehat{Q}_{2}^*-\widehat{Q}_{2}\right)=-\widehat{b_L}^*\widehat{b_L}\mod(\Psi^{-L-1+0}).
\end{eqnarray}
\end{lemm}
Note that the operators $\widehat{a_L}$ and $\widehat{b_L}$ also depend on $m_{N_0,N_1}$ even if we did not mention this dependence explicitely.
\begin{proof} We only deal with the case of $\widehat{Q}_1$ as the other case can be treated similarly. Recall that
$$\widehat{Q}_{1}=-\widehat{A}_N i\left(\mathcal{L}_{V_f}\right)\widehat{A}_N^{-1}.$$ 
The symbol of this operator was given in Lemma~\ref{l:pseudo}. Observe that the term $-X_{H_f}G_N^0$ appearing there belongs to\footnote{That
we write shortly as $S^{+0}$ in the sequel.} $S^{+0}(T^*M,\Lambda^k(T^*M))$ 
since it contains a $\log$ term of the form $\left(-X_{H_f}.m\right) \log\left(1+\vert\xi\vert^2 \right)$. Therefore the symbol $-\tilde{p}_m(x;\xi)$ of 
$$\widehat{P}_m:=\frac{i}{2}\left(\widehat{P}_{1}^*-\widehat{P}_{1}\right)-\frac{C-C_N}{3}$$ belongs to $S^{+0}$ and it
satisfies the following~:
\begin{eqnarray*}
-\tilde{p}_m(x;\xi)=\underset{\in S^{+0}}{\underbrace{-X_{H_f}.G_N^0\mathbf{Id}}}+\frac{C-C_N}{3}\mathbf{Id}+p_0(x;\xi)+p_{-1,m}(x;\xi)
\end{eqnarray*}
where $p_0\in S^0$ (independent of $m$) and $p_{-1,m}\in S^{-1+0}$ (depending on $m$). Moreover, from~\eqref{e:upperbound-squareroot}, one finds that,
for $\vert \xi\vert \geqslant R$ with $R\geq 1$ large enough,
$$ -\tilde{p}_m(x;\xi) \geqslant  \frac{C_N-C}{3}\mathbf{Id}.$$
This implies that we can find a smooth symbol $a_0(x;\xi)\in S^{+0}$ such that
$-\tilde{p}_m=a_0^2 $ for $\vert \xi\vert \geqslant R$ with $R\geq 1$ large enough.
By construction, $a_0$ belongs to $S^{+0}$, it depends on $m$ and $a_0^2=-\tilde{p}_m \text{ mod}(S^{-\infty})$.
Then, the theorem for composition of pseudodifferential operators yields
$$ \Op(a_0)^*\Op(a_0)=\Op(a_0^2)\mod(\Psi^{-1+0})=-\widehat{P}_m \mod(\Psi^{-1+0}).$$
Let us now show by induction that there exists a sequence of symbols
$(a_i)_{i=0}^\infty$,
$a_i\in S^{-i+0}$ such that for every $l\geq 0$,
$$ \left( \Op(a_0)^*+\dots+\Op(a_l)^* \right)\left( \Op(a_0)+\dots+ \Op(a_l) \right)=-\widehat{P}_m \mod(\Psi^{-(l+1)+0}). $$
Assume we could construct the first $l$ terms $(a_0,\dots,a_l)$ satisfying this property.
This implies that
$$ \left( \Op(a_0)^*+\dots+\Op(a_l)^* \right)\left(\Op(a_0)+\dots+ \Op(a_l) \right)=-\widehat{P}_m+ 
\Op(r_{l+1}) \mod(\Psi^{-(l+2)+0}),$$
where $r_{l+1}\in S^{-(l+1)+0}$ is selfadjoint. We are now looking for $\Op(a_{l+1})\in\Psi^{-(l+1)+0}$ such that
$$ \left( \Op(a_0)^*+\dots+\Op(a_{l+1})^* \right)\left( \Op(a_0)+\dots+ \Op(a_{l+1}) \right)=-\widehat{P}_m \mod(\Psi^{-(l+2)+0}).$$
Equivalently, we have to ensure that
$$ \Op(a_0)^* \Op(a_{l+1})+ \Op(a_{l+1})^* \Op(a_0)=\Op(r_{l+1})\mod(\Psi^{-(l+2)+0}).$$
Now let $a_0^{-1}$ denote any selfadjoint symbol such that $a_0^{-1}a_0=a_0a_0^{-1}=\mathbf{Id}$ for $\vert\xi\vert$ large enough.
We set
$\Op(a_{l+1})=-\frac{1}{2}\Op(a_0^{-1})\Op(r_{l+1})\in\Psi^{-(l+1)+0}$ which 
implies the previous equality from the composition formula for pseudodifferential operators. It concludes the proof of the Lemma.
\end{proof}

We can now follow the proof of~\cite{FS}. First of all, 
arguing as in~\cite[Lemma~3.3]{FS}, we can show 
that $\widehat{P}_{\hbar}$ has empty 
spectrum for $\text{Im}(z)>C_0$, where $C_0$ is 
some positive constant that may depend on $m$ but which can be made uniform in terms of $\hbar\in [0,1)$. In other words, the resolvent
$$\left(\widehat{P}_{\hbar}-z\right)^{-1}: L^2(M,\Lambda^k(T^*M))\rightarrow L^2(M,\Lambda^k(T^*M))$$
defines a bounded operator for  $\text{Im}(z)>C_0$. Now, we will show how to extend it meromorphically to some half-plane
$$\left\{z:\text{Im}(z)\geq\frac{C-C_N}{3}\right\}.$$

\subsubsection{From resolvent to the semigroup}\label{r:resolvent} Before doing that, we already note that the proof of~\cite[Lemma~3.3]{FS} implicitely shows that, 
for every $z$ in $\mathbb{C}$ satisfying $\text{Im} (z)>C_0$, one has
\begin{equation}\label{e:norm-resolvent}
 \left\|\left(\widehat{P}_0-z\right)^{-1}\right\|_{L^2(M,\Lambda^k(T^*M))\rightarrow L^2(M,\Lambda^k(T^*M))}\leq\frac{1}{\text{Im}(z)-C_0},
\end{equation}
which will allow to relate the spectrum of the generator to the spectrum of the corresponding semigroup $\varphi_f^{-t*}$. In particular, 
combining this observation with~\cite[Cor.~3.6,~p.76]{EnNa00}, we know that, for $t\geq 0$
$$\varphi_f^{-t*}:\ml{H}_k^{m_{N_0,N_1}}(M)\rightarrow\ml{H}_k^{m_{N_0,N_1}}(M)$$
generates a strongly continuous semigroup whose norm verifies
$$\forall t\geq 0,\quad\|\varphi_f^{-t*}\|_{\ml{H}_k^{m_{N_0,N_1}}(M)\rightarrow \ml{H}_k^{m_{N_0,N_1}}(M)}\leq e^{tC_0}.$$

\subsubsection{Resolvent construction and meromorphic continuation}

We fix some large integer $L>\dim(M)/2$ to ensure that the operator $(1+\Delta_g)^{-L}$ is \textbf{trace class}. 
For every $u$ in $\ml{C}^{\infty}(M)$ and for every $0\leq \hbar\leq 1$,
we start from the identity~\eqref{e:antiadjoint} which yields:
$$ \text{Im}\left\la \left(\widehat{P}_{\hbar}+\frac{C_N-C}{3}\right) u,u\right\ra=\left\langle \frac{1}{2}\left(\widehat{Q}^*_1-\widehat{Q}_1\right)u,u \right\rangle+
\hbar\left\langle  \frac{1}{2}\left(\widehat{Q}^*_2-\widehat{Q}_2\right)u,u \right\rangle .$$
From the squareroot Lemma~\ref{l:squareroot} applied to the operators appearing on the right hand side with $2(L+1)$ instead of $L$, we 
now find that, 
\begin{equation}\label{e:garding1}\text{Im}\left\la \left(\widehat{P}_{\hbar}+\frac{C_N-C}{3}\right) u,u\right\ra=-\left\|\widehat{a_{2(L+1)}}u\right\|^2-\hbar\left\|\widehat{b_{2(L+1)}}u\right\|^2+
\left\la \widehat{R_m}u,u\right\ra,\end{equation}
where the remainder $\widehat{R_m}$ belongs to $\Psi^{-2(L+1)}(M)$, $\widehat{a_{2(L+1)}}\in\Psi^{+0}(M)$ and $\widehat{b_{2(L+1)}}\in \Psi^1(M)$.
We made appear a negative term $-\left\|\widehat{a_{2(L+1)}}u\right\|^2-\hbar\left\|\widehat{b_{2(L+1)}}u\right\|^2$
and our goal is to give an upper bound to control the error term $\left\la \widehat{R_m}u,u\right\ra $. For that purpose, we write
\begin{equation}\label{e:garding2}\left\la \widehat{R_m}u,u\right\ra=\left\la\underset{\in \Psi^{0}(M)}{\underbrace{(1+\Delta_g)^{\frac{L}{2}}\widehat{R_m}(1+\Delta_g)^{\frac{L}{2}}}}
(1+\Delta_g)^{-\frac{L}{2}}u,(1+\Delta_g)^{-\frac{L}{2}}u\right\ra\leq  R\left\la(1+\Delta_g)^{-L}u,u\right\ra.\end{equation}
for some large enough parameter $R>0$ depending on $m_{N_0,N_1}$ and $L$. Observe that the inequality follows from the Cauchy-Schwarz inequality and the $L^2$ 
boundedness of pseudodifferential operators in $\Psi^0(M)$.

Let us now set~:
$$ \widehat{\chi}_R:=-iR (1+\Delta_g)^{-L}.$$
 Combining~\eqref{e:garding1} and~\eqref{e:garding2} yields
 $\text{Im}\left\la \left(\widehat{P}_{\hbar}+\frac{C_N-C}{3}\right) u,u\right\ra\leq 
 \la\widehat{R_m}u,u\ra \leq R\left\la(1+\Delta_g)^{-L}u,u\right\ra$ thus~:
\begin{equation}\label{e:bound-imaginary-part}\text{Im}\left\la \left(\widehat{P}_{\hbar}+\frac{C_N-C}{3}\right) u,u\right\ra\leq-\text{Im}\la \widehat{\chi}_Ru,u\ra.\end{equation}
We can now argue like in~\cite[Lemma~3.3]{FS} to conclude that $\left(\widehat{P}_{\hbar}+\widehat{\chi}_R-z\right)$ is invertible 
for $\text{Im}(z)>\frac{C-C_N}{3}$. In fact, set $\eps=\text{Im}(z)-\frac{C-C_N}{3}.$ From~\eqref{e:bound-imaginary-part}, we find that
$$\text{Im}\left\la \left(\widehat{P}_{\hbar}+\widehat{\chi}_R-z\right) u,u\right\ra\leq -\eps\|u\|^2.$$
Applying the Cauchy-Schwarz inequality, we find that
$$\left\| \left(\widehat{P}_{\hbar}+\widehat{\chi}_R-z\right) u\right\|\left\|u\right\|\geq\left|
\text{Im}\left\la \left(\widehat{P}_{\hbar}+\widehat{\chi}_R-z\right) u,u\right\ra\right|\geq\eps\|u\|^2.$$
This implies that $\left(\widehat{P}_{\hbar}+\widehat{\chi}_R-z\right)$ is injective. We can 
argue similarly for the adjoint operator and obtain that
 $$\left\| \left(\widehat{P}_{\hbar}^*+\widehat{\chi}_R-\overline{z}\right) u\right\|\left\|u\right\|\geq\eps\|u\|^2,$$
from which we can infer that $\left(\widehat{P}_{\hbar}+\widehat{\chi}_R-z\right)$ is surjective~\cite[Th.~II.19]{Br}. Hence, we can conclude that 
$$\left(\widehat{P}_{\hbar}+\widehat{\chi}_R-z\right)^{-1}: L^2(M,\Lambda^k(T^*M))\rightarrow L^2(M,\Lambda^k(T^*M))$$
defines a bounded operator for  $\text{Im}(z)>\frac{C-C_N}{3}$ and that its operator norm satisfies the estimate
$$\left\|\left(\widehat{P}_{\hbar}+\widehat{\chi}_R-z\right)^{-1}\right\|_{L^2\rightarrow L^2}\leq\frac{1}{\text{Im}(z)-\frac{C-C_N}{3}}.$$
The conclusion then follows from analytic Fredholm theory combined with the following expression:
\begin{equation}\label{e:keyidentity1}
\text{Im}(z)>\frac{C-C_N}{3}\Longrightarrow \widehat{P}_\hbar -z=\left(\text{Id}- \widehat{\chi}_R\left( P_\hbar +\widehat{\chi}_{R}-z\right)^{-1}\right)
\left(\widehat{P}_\hbar +\widehat{\chi}_{R}-z\right).
\end{equation}
Note that $\widehat{\chi}_R\in \Psi^{-2L}(M)$ is by definition a trace class operator for $L$ large enough (at least $>\dim(M)/2$~\cite[Prop.~B.20]{Zw17}). It implies that the operator
\begin{eqnarray*}\label{e:smoothingop}
\widehat{\chi}_R\left( \widehat{P}_\hbar +\widehat{\chi}_{R}-z\right)^{-1}
\end{eqnarray*}
is trace-class for every $\hbar\geqslant 0$ as composition of a trace class operator and a bounded one. Moreover, it depends holomorphically on $z$ in the domain $\{\text{Im}(z)>-\frac{C_N-C}{3}\}$ implying that $\widehat{P}_\hbar -z$ is a holomorphic family of Fredholm operators for $z$ in the same domain. Finally, we can apply arguments from analytic Fredholm theory
to $\widehat{P}_\hbar -z$
~\cite[Th.~D.4 p.~418]{Zworski} which yields the analytic continuation
of $(P_\hbar-z)^{-1}$ as a
meromorphic family of Fredholm operators for $z\in \{\text{Im}(z)>\frac{C-C_N}{3}\}$  and, arguing as in~\cite[Lemma~3.5]{FS}, we 
can conclude that $\widehat{P}_\hbar$ has discrete spectrum with finite multiplicity on $\text{Im}(z)>\frac{C-C_N}{3}.$


\subsubsection{Definition of the resonances}

Since $\widehat{P}_\hbar$ is conjugated to $\widehat{H}_\hbar$, the above discussion
implies that $\widehat{H}_{\hbar}$ has a discrete spectrum with finite multiplicity on $\text{Im}(z)>\frac{C-C_N}{3}$ as 
an operator acting on $\ml{H}^{m_{N_0,N_1}}_k(M)$. It can be shown~\cite[Sect.~4]{FS} that \textbf{this discrete spectrum is independent of 
the choice of the order function} $m_{N_0,N_1}$, meaning 
that both the eigenvalues and the generalized eigenmodes (on $\ml{H}^{m_{N_0,N_1}}_k(M)$) are intrinsic. 
In particular, the argument from this reference also shows that, for $\hbar>0$, this discrete spectrum is the same as the one obtained for 
$\widehat{H}_{\hbar}$ viewed as an operator acting on 
$L^2(M,\Lambda^k(T^*M))$. In that case, we recall that the resonances are, for $\hbar>0$, the Witten eigenvalues, i.e. of the form (up to 
multiplication by $-i$)
$$0\geq -\lambda_1^{(k)}(\hbar)\geq -\lambda_2^{(k)}(\hbar)\geq\ldots\geq -\lambda_j^{(k)}(\hbar)\rightarrow-\infty\ \text{as}\ j\rightarrow+\infty.$$
Our next step will be to show that this Witten spectrum indeed converges to the Pollicott-Ruelle spectrum.

\subsubsection{Convention} In the following, we shall denote this intrinsic discrete spectrum by $\ml{R}_k(\hbar)$. 
They correspond to the eigenvalues of $\widehat{H}_{\hbar}$ acting on an appropriate Sobolev space of currents
of degree $k$. 
When $\hbar>0$, these are the Witten eigenvalues (up to a factor $-1$) while, for $\hbar=0$, they 
represent the correlation spectrum of the gradient flow, which is often referred as the Pollicott-Ruelle spectrum.

\subsubsection{Boundedness on standard Sobolev spaces} Denote by $H^s(M,\Lambda^k(T^*M))$ the 
standard Sobolev space of index $s>0$, i.e.
$$H^s(M,\Lambda^k(T^*M)):=(1+\Delta_g^{(k)})^{-\frac{s}{2}}L^2(M,\Lambda^k(T^*M)).$$
The above construction shows that
$$\left(\widehat{P}_{\hbar}+\widehat{\chi}_R-z\right)^{-1}: L^2(M,\Lambda^k(T^*M))\rightarrow L^2(M,\Lambda^k(T^*M))$$
defines a bounded operator for  $\text{Im}(z)>\frac{C-C_N}{3}$ which depends holomorphically on $z$. For the sequel, we will in fact 
need something slightly stronger:
\begin{lemm}\label{l:resolventbounded} Let $s_0>0$ and let $N_0,N_1$ such that $N_0,N_1>4(\|f\|_{\ml{C}^0}+s_0)$. Then, there exists $R>0$ 
such that, for every $z$ satisfying $\operatorname{Im}(z)>-\frac{C_N-C}{3}$ and for every $s\in[-s_0,s_0]$, the resolvent~:
\begin{equation}
\left(P_\hbar +\widehat{\chi}_{R}-z\right)^{-1}
\end{equation}
exists as a holomorphic function of $z\in \{\operatorname{Im}(z)>-\frac{(C_N-C)}{3}\}$ valued in bounded operator from 
$H^{2s}(M,\Lambda^k(T^*M))\mapsto H^{2s}(M,\Lambda^k(T^*M))$. Moreover, one has, for every $0\leq \hbar\leq 1$,
$$\left\|\left(\widehat{P}_\hbar +\widehat{\chi}_{R}-z\right)^{-1}\right\|_{H^{2s}\rightarrow H^{2s}}\leq\frac{1}{\operatorname{Im}(z)+\frac{\left(C_N-C \right)}{3}}.$$
\end{lemm}
\begin{proof} This Lemma is just a generalization of the above argument. Fix $s_0$ in $\IR_+$ and $s\in[-s_0,s_0]$. 
We set $N_0,N_1$ to be large enough, precisely $N_0,N_1>4(\|f\|_{\ml{C}^0}+|s_0|)$ as in the statement of Lemma~\ref{l:escape-function}. 
As above, we note that this problem is equivalent to considering the spectrum of the operator
$$(1+\Delta_g)^{s}\left(\widehat{P}_{\hbar}+\widehat{\chi}_R\right)(1+\Delta_g)^{-s},$$
acting on $L^2(M)$. One more time, we removed the dependence in $k$ for simplicity. As in Lemma~\ref{l:pseudo}, 
the principal symbol of this operator remains $\xi(V_f(x))-i\hbar\|\xi\|^2$ from the composition rules for pseudodifferential operators. 
For $\hbar>0$, this defines an elliptic operator. Thus, it 
has again a unique closed extension as an unbounded operator on $L^2$. When $\hbar=0$, we can make use of~\cite[Lemma~A.1]{FS} to get 
the same conclusion. Now, we have to show that  
$$(1+\Delta_g)^{s}\left(\widehat{P}_{\hbar}+\widehat{\chi}_R-z\right)(1+\Delta_g)^{-s}$$
is invertible in some half-plane for $\text{Im}(z)$ greater than some constant which depends on the constant 
$C_N$ from Lemma~\ref{l:escape-function}. In the case $s=0$, we explained how to make use of Lemma~\ref{l:squareroot} 
to prove such a result. Here, we can argue similarly. In fact, one has
$$(1+\Delta_g)^{s}\left(\widehat{P}_{\hbar}+\widehat{\chi}_R-z\right)(1+\Delta_g)^{-s}=
(1+\Delta_g)^{s}\widehat{P}_{\hbar}(1+\Delta_g)^{-s}
+\widehat{\chi}_R-z.$$
Lemma~\ref{l:pseudo} still holds for the operator $(1+\Delta_g)^{s}\widehat{P}_{\hbar}(1+\Delta_g)^{-s}$ except that we should replace 
$G_{N}^0$ by $G_N^s(x;\xi)=(m_{N_0,N_1}(x;\xi)+s)\log(1+\|\xi\|^2)$ in the above proof\footnote{Note that they satisfy the 
same properties from Lemma~\ref{l:escape-function}.}. Then, we find by using the squareroot Lemma~\ref{l:squareroot} (extended to this case) that
$$\text{Im}\left\la \left((1+\Delta_g)^{s}\widehat{P}_{\hbar}(1+\Delta_g)^{-s}+\frac{C_N-C}{3}\right) u,u\right\ra\leq R \la (1+\Delta_g)^{-L}u,u\ra,$$
provided that $R>0$ is large enough. Note that $R$ can be made uniform for $s\in[-s_0,s_0]$. We conclude one more time by arguing as 
in~\cite[Lemma~3.3]{FS} and paragraph~\ref{sss:meromorphic}. Precisely, for $\text{Im}(z)+\frac{\left(C_N-C \right)}{3}>0$, 
$$\left(\widehat{P}_\hbar +\widehat{\chi}_{R}-z\right)^{-1}:H^{2s}(M,\Lambda^k(T^*M))\rightarrow H^{2s}(M,\Lambda^k(T^*M))$$
defines a bounded operator on the standard Sobolev space which depends holomorphically on $z$. Moreover, its norm is bounded as follows
$$\left\|\left(\widehat{P}_\hbar +\widehat{\chi}_{R}-z\right)^{-1}\right\|_{H^{2s}\rightarrow H^{2s}}\leq\frac{1}{\text{Im}(z)+\frac{\left(C_N-C \right)}{3}}.$$
\end{proof}

\subsection{Pollicott-Ruelle resonances as 
zeros of a Fredholm determinant}\label{ss:determinant} 
From expression~\eqref{e:keyidentity1}, 
we know that, for $\text{Im}(z)>\frac{C-C_N}{3}$, 
$z$ belongs to the spectrum of 
$\widehat{P}_{\hbar}$ if and only if 
the operator $\left(\text{Id}+ \widehat{\chi}_R\left( P_\hbar -\widehat{\chi}_R-z\right)^{-1}\right)$ 
is not invertible. As we have shown that 
$$\widehat{\chi}_R\left( \widehat{P}_\hbar +\widehat{\chi}_{R}-z\right)^{-1}$$
is a trace class operator on $L^2(M,\Lambda^k(T^*M))$, 
this is equivalent to saying that $z$ is a zero of the 
Fredholm determinant~\cite[Prop.~B.25]{DZresonances} 
$$D_{m_{N_0,N_1}}(\hbar,z):=\text{det}_{L^2}\left(\text{Id}- \widehat{\chi}_R\left( \widehat{P}_\hbar +\widehat{\chi}_{R}-z\right)^{-1}\right).$$
Moreover, the multiplicity of $z$ as an eigenvalue of 
$\widehat{P}_{\hbar}$ coincides with the multiplicity of $z$ as a zero of 
$D_m(\hbar,z)$~\cite[Prop.~B.29]{DZresonances}.

\section{From the Witten spectrum to the Pollicott-Ruelle spectrum}\label{s:convergence}

Now that we have recalled the precise notion of resonance spectrum for the limit operator $-\ml{L}_{V_f}$, 
we would like to explain how the 
Witten spectrum converges to the resonance spectrum of the Lie derivative. 
This will be achieved by an argument due to Dyatlov 
and Zworski~\cite{DZsto} in the context of \emph{Anosov} flows -- see also~\cite{Zw15}. 
In this section, we briefly recall their proof adapted to our framework.

\begin{rema}
 In~\cite{DZsto}, Dyatlov and Zworski prove something slightly stronger as they obtain smoothness in $\hbar$. Here, we are aiming at something 
 simpler and we shall not prove smoothness which would require some more work that would be beyond the scope of the present article 
 -- see~\cite{DZsto} for details in the Anosov case.
\end{rema}

\subsection{Convergence of the eigenvalues} We fix $N_0$, $N_1$, $s_0>2$ and $R$ as in the statement of Lemma~\ref{l:resolventbounded}. 
Using the conventions of this paragraph, we start by studying the regularity of the operator 
$$\hbar\in[0,1]\mapsto  K_m(\hbar):=\widehat{\chi}_R\left( \widehat{P}_\hbar +\widehat{\chi}_{R}-z\right)^{-1}.$$
Recall that $K_m(\hbar)$ is an holomorphic map on $\{\text{Im}(z)>(C-C_N)/3\}$ with values on the space of trace class operators on $L^2$. 
For $\hbar,\hbar'\in[0,1]$, we now write
\begin{equation}\label{e:lipschitz}\frac{\left( 
\widehat{P}_\hbar +\widehat{\chi}_{R}-z\right)^{-1}-\left( 
\widehat{P}_{\hbar'} +\widehat{\chi}_{R}-z\right)^{-1}}{\hbar-\hbar'}=-\left( 
\widehat{P}_\hbar +\widehat{\chi}_{R}-z\right)^{-1}\widehat{Q}_2\left( 
\widehat{P}_{\hbar'} +\widehat{\chi}_{R}-z\right)^{-1},\end{equation}
where we recall that
$$\widehat{Q}_2=-i\widehat{A}_N \left(\frac{\Delta_g}{2}\right)\widehat{A}_N^{-1}.$$
From Lemma~\ref{l:resolventbounded} with $s_0>2$, we find that~\eqref{e:lipschitz} is bounded for 
$\text{Im}(z)>\frac{C-C_N}{3}$ and uniformly for $\hbar\in[0,1]$ as an operator from $L^2$ to $H^{-2}$. 
Hence, we have verified that $$\hbar\mapsto\left(\widehat{P}_\hbar +\widehat{\chi}_{R}-z\right)^{-1}$$ 
defines a \emph{Lipschitz (thus continuous)} map in $\hbar$ with values in the set $\text{Hol}(\{\text{Im}(z)>\frac{C-C_N}{3} \},\mathcal{B}(L^2,H^{-2}))$ of holomorphic functions in $z$ valued in the Banach space $\mathcal{B}(L^2,H^{-2})$ of bounded operators from 
$L^{2}$ to $H^{-2}$. Recall now that
$$\widehat{\chi}_R=-iR(1+\Delta_g)^{-L}$$
is trace class from $H^{-2}$ to $L^2$ for $L$ large enough (precisely\footnote{This follows from the Weyl's law.} $L>\text{dim}(M)/2+1$). 
Denote by $\mathcal{L}^1(H^{-2}(M),L^2(M))\subset \mathcal{B}(H^{-2}(M),L^2(M))$
the set of trace class operators acting on these spaces~\cite[Sect.~B.4]{DZresonances}. By continuity of
the composition map 
$ (A,B)\in  \mathcal{L}^1(H^{-2},L^2)\times  \mathcal{B}(L^2,H^{-2}) \mapsto AB\in  \mathcal{L}^1(L^2,L^2) $~\cite[Eq.~(B.4.6)]{DZresonances}, the operator
$$K_m(\hbar)=\underset{\text{trace class}}{\underbrace{\widehat{\chi}_R}}
\underset{\text{Lipschitz in }\mathcal{B}(L^2,H^{-2})}{\underbrace{\left( \widehat{P}_\hbar +\widehat{\chi}_{R}-z\right)^{-1}}}$$ is the composition of a Lipschitz operator valued in $\text{Hol}(\{\text{Im}(z)>\frac{C-C_N}{3} \},\mathcal{B}(L^2,H^{-2}))$ with the fixed trace class operator $\widehat{\chi}_R\in \mathcal{L}^1$. $K_m$ must therefore be 
a Lipschitz map in $\hbar\in [0,1]$ valued in $\text{Hol}(\{\text{Im}(z)>\frac{C-C_N}{3} \},\mathcal{L}^1(L^2,L^2))$. 
We have
thus shown the following Lemma:
\begin{lemm}\label{l:continuity} Let $N_0,N_1>4(\|f\|_{\ml{C}^0}+2)$ and let $R>0$ be as in the statement of Lemma~\ref{l:resolventbounded} 
with $s_0=2$. Then, the map
 $$\hbar\mapsto  K_m(\hbar)$$
 is Lipschitz (hence continuous) from $[0,1]$ to the space of holomorphic functions on $\{\operatorname{Im}(z)>(C-C_N)/3\}$ with values on the space of trace class 
 operators on $L^2.$
\end{lemm}

\begin{rema} Note that, for the sake of simplicity, we omitted the dependence in the degree $k$ in that statement. 
\end{rema}

Let us now draw some consequences of this Lemma. From~\cite[Sect.~B.5, p.~426]{DZresonances}, the determinant map
$$D_{m_{N_0,N_1}}(\hbar,\dot):z\in\left\{\text{Im}(z)>\frac{C-C_N}{3}\right\}\mapsto\text{det}_{L^2}\left(\text{Id}-\widehat{\chi}_R
\left( \widehat{P}_\hbar +\widehat{\chi}_{R}-z\right)^{-1}\right)$$
is holomorphic. Moreover, one knows from~\cite[Prop.~B.26]{DZresonances} that
$$\left|D_{m_{N_0,N_1}}(\hbar,z)-D_{m_{N_0,N_1}}(\hbar',z)\right|\leq\|K_m(\hbar,z)-K_m(\hbar',z)\|_{\text{Tr}}
e^{1+\|K_m(\hbar,z)\|_{\text{Tr}}+\|K_m(\hbar',z)\|_{\text{Tr}}},$$
which, combined with Lemma~\ref{l:continuity}, implies that $\hbar\mapsto D_{m_{N_0,N_1}}(\hbar,.)$ is a continuous map from $[0,1]$ to the space of holomorphic functions on 
$\left\{\text{Im}(z)>\frac{C-C_N}{3}\right\}$.

Fix now $z_0$ an eigenvalue of $\widehat{P}_{0}$ lying on the half-plane $\left\{\text{Im}(z)>\frac{C-C_N}{3}\right\}$ and having 
algebraic multiplicity 
$m_{z_0}$. This corresponds to a zero of multiplicity $m_{z_0}$ of the determinant map $D_{m_{N_0,N_1}}(0,.)$ evaluated at 
$\hbar=0$. As the spectrum of $\widehat{P}_0$ is discrete with finite multiplicity on this half plane, we can find a small enough $r_0>0$ such that the 
closed disk centered at $z_0$ of radius $r_0$ contains only the eigenvalue $z_0$. The map $\hbar\mapsto D_{m_{N_0,N_1}}(\hbar,.)\in \text{Hol}(\{\text{Im}(z)>\frac{C-C_N}{3} \})$ 
being continuous, we know that, for every $0<r_1\leq r_0$, for $\hbar\geq 0$ small enough (which depends on $z_0$ and on $r_1$) and for $|z-z_0|=r_1$,
$$|D_{m_{N_0,N_1}}(\hbar,z)-D_{m_{N_0,N_1}}(0,z)|<\min_{z':|z'-\lambda_0|=r_0}|D_{m_{N_0,N_1}}(0,z')|\leq |D_{m_{N_0,N_1}}(0,z)|.$$
Hence, from the Rouch\'e Theorem and for $\hbar\geq 0$ small enough, the number of zeros counted with multiplicity of $D_{m_{N_0,N_1}}(\hbar)$ lying on the disk 
$\{z:|z-z_0|\leq r_1\}$ equals $m_{z_0}$. As, for $\hbar>0$, the Witten eigenvalues lie 
on the real axis, we have shown the following Theorem~:
\begin{theo}\label{t:convergence-eigenvalue} Let $0\leq k\leq n$. Then, the set of Pollicott-Ruelle resonances 
$\ml{R}_k=\ml{R}_k(0)$ of $-\mathcal{L}_{V_f}^{(k)}$
is contained inside $(-\infty,0]$. 
Moreover, given any $z_0$ in $(-\infty,0]$, there exists $r_0>0$ such that, for every $0<r_1\leq r_0$, for $\hbar>0$ 
small enough (depending on $z_0$ and $r_1$), the number of elements (counted with algebraic multiplicity) inside 
$$\ml{R}_k(\hbar)\cap\{z:|z-z_0|\leq r_1\}$$ 
 is constant and equal to the algebraic multiplicity of $z_0$ as an eigenvalue of $-\ml{L}_{V_f}^{(k)}$.
\end{theo}
 This Theorem shows that the Witten eigenvalues converge to the Pollicott-Ruelle resonances of $-\ml{L}_{V_f}$. Yet, for the moment, it does not prove 
Theorem~\ref{t:maintheo-eigenvalues} as we still have to relate this limit spectrum to the zeros of the Ruelle dynamical determinant.
Note that this Theorem combined with the results from section~\ref{s:anisotropic} implies
the following~:
\begin{lemm}[Spectrum of the semigroup]\label{l:spectrum-propagator} 
For every $0\leq k\leq n$, the operator
$$-\ml{L}_{V_f}^{(k)}:\ml{H}^{m_{N_0,N_1}}_k(M)\rightarrow\ml{H}^{m_{N_0,N_1}}_k(M)$$
 has only finitely many eigenvalues on the half-plane\footnote{Recall from paragraph~\ref{sss:meromorphic} that the spectrum is discrete on this half-plane.} 
 $\{\operatorname{Im}(z)>(C-C_N)/3\}$. Following~\cite[Sect.~IV.2]{EnNa00} -- see also~\cite[Par.~5.3]{dang2016spectral}, we find that, 
 for every $t_0>0$, the spectrum of 
 $$\varphi^{-t_0*}:\ml{H}^{m_{N_0,N_1}}_k(M)\rightarrow\ml{H}^{m_{N_0,N_1}}_k(M)$$
 is discrete with finite multiplicities inside
 $\left\{z:|z|>e^{t_0\frac{C-C_N}{3}}\right\},$
 and, when counted with algebraic multiplicities, it coincides with the set
 $\left\{e^{t_0z_0}:z_0\in\ml{R}_k\right\}.$
\end{lemm}
We also record the following useful statement on the 
asymptotics of the correlation function of $\varphi_f^{-t*}$:
\begin{prop}\label{p:correlation} Let $0\leq k\leq n$ then for any $z_0\in \mathcal{R}_k$, there is an integer
$d_{z_0}^{(k)}\geq 1$ and a continuous linear map
with finite rank,
$$\pi_{z_0}^{(k)}:\Omega^k(M)\rightarrow\ml{D}^{\prime k}(M),$$
s.t. for any $\Lambda>0$, there  exist $N_0,N_1$ large enough 
so that
for every $(\psi_1,\psi_2)\in\Omega^k(M)\times\Omega^{n-k}(M)$ and for every $t\geq 0$,
$$\int_M\varphi^{-t*}(\psi_1)\wedge\psi_2=
\sum_{z_0\in\ml{R}_k:z_0>-\Lambda}e^{tz_0}\sum_{l=0}^{d_{z_0}^{(k)}-1}\frac{t^l}{l!}
\int_{M}\left(\ml{L}_{V_f}^{(k)}+z_0\right)^{l}
\left(\pi_{z_0}^{(k)}(\psi_1)\right)\wedge\psi_2$$
$$+\ml{O}\left(e^{-\Lambda t}\|\psi_1\|_{\ml{H}^{m_{N_0,N_1}}_k}\|\psi_2\|_{\ml{H}^{-m_{N_0,N_1}}_{n-k}}\right).$$
In fact, the result also holds for any $\psi_1$ in $\ml{H}^{m_{N_0,N_1}}_k(M)$. 
\end{prop}
This Proposition follows from Lemma~\ref{l:spectrum-propagator}. We emphasize that $\pi_{z_0}^{(k)}$ corresponds to the spectral projector 
associated with the eigenvalue $z_0$ and it is intrinsic thanks to~\cite[Th.~1.5]{FS}, i.e. independent of the choice of the order function 
$m_{N_0,N_1}$. Once we will have determined the set $\ml{R}_k$, we will be able to deduce Theorem~\ref{t:maintheo-flow} by integrating over $[0,+\infty)$ 
the equality from Proposition~\ref{p:correlation} against $e^{-zt}$.

\subsection{Convergence of the spectral projectors}

Now that we have stated this Proposition of independent interest, we come back to the convergence of the spectral projectors of the Witten Laplacian 
to the operators $\pi_{z_0}^{(k)}$: 
\begin{theo}\label{t:conv-proj} Let $0\leq k\leq n$ and $z_0$ be an element\footnote{For $z_0\notin\ml{R}_k$, one has $\pi_{z_0}^{(k)}=0$.} 
in $\IR$. Then, there exists $r_0>0$ such that, for every 
 $(\psi_1,\psi_2)\in\Omega^k(M)\times\Omega^{n-k}(M)$,
 $$\forall 0<r_1\leq r_0,\quad\lim_{\hbar\rightarrow 0^+}\int_M\mathbf{1}_{[z_0-r_1,z_0+r_1]}\left(-W_{f,\hbar}^{(k)}\right)
 \left(e^{-\frac{f}{\hbar}}\psi_1\right)\wedge\left(e^{\frac{f}{\hbar}}\psi_2\right)=\int_{M}\pi_{z_0}^{(k)}(\psi_1)\wedge\psi_2.$$
\end{theo}
Knowing this Theorem, the proof of Theorem~\ref{t:maintheo-proj} now relies on the fact that we 
will be able to determine exactly the Pollicott-Ruelle spectrum of the vector field $V_f$. This will be the object of the next section and, before 
that, we prove Theorem~\ref{t:conv-proj}.
\begin{proof} Using Theorem~\ref{t:convergence-eigenvalue}, it is enough to show the existence of $r_0$ and to prove convergence for $r_1=r_0$. 
As before, it is also enough to prove this result for the conjugated operators 
$$\widehat{P}_{\hbar}=-\widehat{A}_N i\left(\mathcal{L}_{V_f}\right)\widehat{A}_N^{-1}-\hbar
\widehat{A}_N i\left(\frac{\Delta_g}{2}\right)\widehat{A}_N^{-1}$$ acting on the standard Hilbert space 
$L^2(M,\Lambda^k(T^*M))$. Fix $z_0$ in $\IR$ and $N_0,N_1$ large enough to ensure that $z_0>\frac{C-C_N}{3}$. The spectral 
projector\footnote{Note that this is eventually $0$.} associated with $z_0$ can be 
written~\cite[Th.~C.6]{DZresonances}~:
$$\Pi_{z_0}^{(k)}:=\frac{1}{2i\pi}\int_{\ml{C}(z_0,r_0)}\left(z-\widehat{P}_0 \right)^{-1} dz$$
where $\ml{C}(z_0,r_0)$ is a small circle of radius $r_0$ centered at $z_0$ such that $z_0$ is 
\emph{the only 
eigenvalue} of $\widehat{P}_0$ inside the 
closed disk surrounded by $\ml{C}(z_0,r_0)$. 
When $z_0$ is not an eigenvalue, we choose 
the disk small enough to ensure that there is no eigenvalues inside it. 
If we denote by $m_{z_0}$ the algebraic 
multiplicity of $z_0$ (which is eventually $0$), then, for $\hbar$ small enough, 
the spectral projector associated 
to $\widehat{P}_{\hbar}$, 
$$\Pi_{z_0}^{(k)}(\hbar):=\frac{1}{2i\pi}\int_{\ml{C}(z_0,r_0)}\left(z-\widehat{P}_\hbar \right)^{-1} dz,$$
has rank $m_{z_0}$ from Theorem~\ref{t:convergence-eigenvalue}. We can now argue as in~\cite[Prop.~5.3]{DZsto} and show that, for every $\psi_1$ 
in $\Omega^k(M)$ and every $\psi_2$ in $\Omega^{n-k}(M)$,
\begin{equation}\label{e:conv-proj}\lim_{\hbar\rightarrow 0^+}\int_M\Pi_{z_0}^{(k)}(\hbar)(\psi_1)\wedge\psi_2
 =\int_M\Pi_{z_0}^{(k)}(\psi_1)\wedge\psi_2.
\end{equation}
Once this equality will be proved, we will be able to conclude by recalling that the generalized eigenmodes are independent of the choice 
of the order function $m_{N_0,N_1}$ used to define $\widehat{A}_N$ and by observing that
$$\pi_{z_0}^{(k)}=-i\widehat{A}_N^{-1}\Pi_{z_0}^{(k)}\widehat{A}_N,$$
 and
 $$e^{\frac{f}{\hbar}}\mathbf{1}_{[z_0-r_0,z_0+r_0]}\left(-W_{f,\hbar}^{(k)}\right)e^{-\frac{f}{\hbar}}=
 -i\widehat{A}_N^{-1}\Pi_{z_0}^{(k)}(\hbar)\widehat{A}_N.$$
 Hence, it remains to prove~\eqref{e:conv-proj}. For that purpose, we use the conventions of Lemma~\ref{l:resolventbounded} and write
$$ (\widehat{P}_\hbar-z)^{-1}= (\widehat{P}_\hbar+\widehat{\chi}_R-z)^{-1}+(\widehat{P}_\hbar-z)^{-1}\widehat{\chi}_R
(\widehat{P}_\hbar+\widehat{\chi}_R-z)^{-1}.$$
By construction of the compact operator $\widehat{\chi}_R$, 
the family $ (\widehat{P}_\hbar+\widehat{\chi}_R-z)^{-1}$ is holomorphic 
and \emph{has no poles}
in some neighborhood of $z_0$ as $z_0>\frac{C-C_N}{3}$. Therefore, 
only the term $(\widehat{P}_\hbar-z)^{-1}\widehat{\chi}_R(\widehat{P}_\hbar+\widehat{\chi}_R-z)^{-1}$
contributes to the contour integral defining the spectral projector $\Pi_{\lambda_0}^{(k)}(\hbar)$:
\begin{eqnarray*}
\Pi_{z_0}^{(k)}(\hbar)=\frac{-1}{2i\pi}\int_{\ml{C}(z_0,r_0)}(\widehat{P}_\hbar-z)^{-1}\widehat{\chi}_R(\widehat{P}_\hbar+\widehat{\chi}_R-z)^{-1}dz.
\end{eqnarray*}
From Theorem~\ref{t:convergence-eigenvalue}, we know that, for $|z-z_0|=r_0$ and for $\hbar$ small enough, 
the operator $(\widehat{P}_\hbar-z)^{-1}$ is uniformly bounded as 
an operator in $\mathcal{B}\left(L^2(M),L^2(M)\right)$. Moreover, we have seen that the map 
$$\hbar\in[0,1]\mapsto\left(z\mapsto\widehat{\chi}_R(\widehat{P}_\hbar+\widehat{\chi}_R-z)^{-1}\right)$$
is continuous (in fact Lipschitz) with values in the set
$\text{Hol}\left(\{\text{Im}(z)>\frac{C-C_N}{3}\},\mathcal{L}^1 \right)$ 
of holomorphic functions with values in trace-class operators on $L^2$. 
This implies that, for every 
$\psi_1$ in $L^2(M,\Lambda^k(T^*M))$, 
$$\Pi_{z_0}^{(k)}(\hbar)(\psi_1)=\frac{-1}{2i\pi}\int_{\ml{C}(z_0,r_0)}(\widehat{P}_\hbar-z)^{-1}\widehat{\chi}_R(\widehat{P}_0+\widehat{\chi}_R-z)^{-1}(\psi_1)dz+o(1),$$
as $\hbar\rightarrow 0^+$. Then, we write
$$(\widehat{P}_\hbar-z)^{-1}\widehat{\chi}_R=(\widehat{P}_0-z)^{-1}\widehat{\chi}_R
+\hbar
\underbrace{(\widehat{P}_\hbar-z)^{-1}\widehat{Q}_2(\widehat{P}_0-z)^{-1}\widehat{\chi}_R}.$$
The term underbraced in factor of $\hbar$ being uniformly bounded as an operator from $L^2$ to $L^2$, we finally find that, for every 
$\psi_1$ in $L^2(M,\Lambda^k(T^*M))$, 
$$\lim_{\hbar\rightarrow 0^+}\left\|\left(\Pi_{z_0}^{(k)}(\hbar)-\Pi_{z_0}^{(k)}\right)(\psi_1)\right\|_{L^2}=0,$$
which concludes the proof of~\eqref{e:conv-proj}.
\end{proof}

\section{Computation of the Pollicott-Ruelle resonances}\label{s:Pollicott-Ruelle}

In~\cite{dang2016spectral}, we gave a full description of the Pollicott-Ruelle spectrum of a Morse-Smale gradient flow under certain 
nonresonance assumptions. Our proof was based on an explicit construction of the generalized eigenmodes and we shall now give a 
slightly different proof based on the works of Baladi and Tsujii on Axiom A diffeomorphisms~\cite{BaTs08, Ba16}. This new proof will only make use of the assumptions 
that the gradient flow is $\ml{C}^1$-linearizable. Yet, in some sense, it will be less self-contained as we shall use the results 
of~\cite{BaTs08} as a ``black-box'' while, in the proof of~\cite{dang2016spectral}, we determined the spectrum by hands even if it was under 
more restrictive assumptions. Another 
advantage of the proof from~\cite{dang2016spectral} was that it gave
an explicit local form of the eigenmodes and
some criteria under which we do not have Jordan blocks 
-- see also~\cite{DaRi17b} for slightly more precise results.
The key idea compared with~\cite{dang2016spectral, DaRi17b} is to use the localized results of Baladi--Tsujii to guess the global
resonance spectrum from the one near each
critical point. To go from local to global, we will use the geometry of the stratification
by unstable manifolds to \emph{glue together}, in some sense, these 
\emph{local spectras} and make them into a global spectrum.

Before starting our proof, let us recall the following classical result of Smale which will be useful to organize our induction 
arguments~\cite{Sm60} -- see~\cite{dangrivieremorsesmale1} for a brief reminder of Smale's works:
\begin{theo}[Smale partial order relation]\label{t:smale} 
Suppose that $\varphi_f^t$ is a Morse-Smale gradient flow. 
Then, for every $a$ in $\operatorname{Crit}(f)$, 
the closure of the unstable manifold $W^u(a)$ 
is the union of certain unstable manifolds $W^u(b)$
for some critical points in $\operatorname{Crit}(f)$. 
Moreover, we say that $b\preceq a$ (resp $b\prec a$), 
if $W^u(b)$ is contained in the closure of $W^u(a)$ 
(resp $W^u(b)\subset \overline{W^u(a)},
W^u(b)\neq W^u(a)$). 
Then, $\preceq$ is a 
\textbf{partial order relation} on 
$\operatorname{Crit}(f)$.
Finally if $b\prec a$, 
then $\operatorname{dim}W^u(b)< \operatorname{dim}W^u(a).$
\end{theo}

In this section, we use the results of Baladi and Tsujii~\cite{BaTs08}. For that purpose, we treat near every critical point the time one map
$\varphi_0:=\varphi^{-1}_f$ of the flow $\varphi^t_f$ as a hyperbolic diffeomorphism with only one fixed point. Recall that 
$\varphi_f^t$ is a Morse-Smale gradient flow which is $\ml{C}^1$-linearizable, hence amenable to the analysis of the previous sections.

\subsection{Local spectra from the work of Baladi--Tsujii.} We start by 
recalling the results of~\cite{BaTs08}. Fix $0\leq k\leq n$, 
the degree of the differential forms we are going to consider and a critical point $a$ of $f$. 
Note that the reference~\cite{BaTs08} mostly deals with
$0$-forms i.e. functions on $M$, which corresponds to $k=0$. 
General results for transfer operators
acting on vector bundles are given in~\cite[section 2]{BaTs08} and~\cite[section 6.4]{Ba16}.
In this paragraph, we consider the transfer operator acting on sections of the bundle
$\Lambda^kT^*M\mapsto M$ of $k$-forms on $M$ by pull--back~:
$ u\in\Gamma(M,\Lambda^kT^*M) \mapsto \varphi_0^*u\in\Gamma(M,\Lambda^kT^*M)$.
For any open subset $U\subset M$, we will denote by $\Omega_c^\bullet(U)$
the differential forms with compact support in $U$. 
Then, one can find a small enough open neighborhood 
$V_a$ of $a$ in $M$ such that, for every $(\psi_1,\psi_2)\in\Omega_c^k(V_a)\times\Omega^{n-k}_c(V_a)$, the map
$$\hat{c}_{\psi_1,\psi_2,a}:z\mapsto \sum_{l=1}^{+\infty}e^{-lz}\int_M\varphi_0^{l*}(\psi_1)\wedge\psi_2$$
has a meromorphic extension to $\IC$. This is a straightforward consequence 
of~\cite[Theorem 2.1]{BaTs08} and~\cite[Theorem 6.12 p.~178]{Ba16} once we note that
smooth differential forms are contained in the Banach spaces of distributional sections of $\Lambda^kT^*M$ used in these references. 
The result of~\cite{BaTs08} is in fact much more general as it holds 
for any Axiom A diffeomorphism provided that the observables are 
supported in the neighborhood of a basic set (which is here reduced to the critical point $a$). Note that this result could also be deduced from the analysis 
in~\cite{GoLi08}. Moreover in~\cite[Theorem 2.2]{BaTs08} (see also \cite[Theorem 6.13 p.~179]{Ba16}), 
Baladi and Tsujii proved the stronger result that the poles of $\hat{c}_{\psi_1,\psi_2,a}$ where $\psi_1,\psi_2$ run over 
$\Omega_c^k(V_a)\times\Omega^{n-k}_c(V_a)$ are exactly equal (with multiplicities) to
the \textbf{zeros} of some \emph{dynamical Ruelle determinant} $\zeta_{R,a}^{(k)}$~:
$\left\{ z_0 : \zeta_{R,a}^{(k)}(z_0)=0\right\}$, where~\cite[p.~179]{Ba16}~:
$$\boxed{\zeta_{R,a}^{(k)}(z):=\exp\left(-\sum_{l=1}^{+\infty}\frac{e^{-lz}}{l}
\frac{\text{Tr}\left(\Lambda^k\left(d\varphi_0^l(a)\right)\right)}{\left|\text{det}\left(\text{Id}-d\varphi_0^l(a)\right)\right|}\right).}$$
Actually, it has been proved in the litterature~\cite{Li05, GiLiPo13, DyZw13, dyatlov2016pollicott} 
for various classes of dynamical systems that the poles
of dynamical correlations correspond 
to the zeros of the dynamical Ruelle determinant. 
Moreover, for any such pole $z_0$, one can find a continuous linear map
$$\pi_{a,z_0}^{(k)}:\Omega_c^k(V_a)\rightarrow\ml{D}^{\prime k}(V_a),$$
which is of finite rank equal to the multiplicity of $z_0$ as a zero of $\zeta_{R,a}^{(k)}$ 
and such that the residue of $\hat{c}_{\psi_1,\psi_2,a}(z)$ at $z=z_0$ is equal to
$$\int_M\pi_{a,z_0}^{(k)}(\psi_1)\wedge\psi_2.$$
Again, $\pi_{a,z_0}^{(k)}$ corresponds to the spectral projector of $\varphi_0^{*}$ acting 
on a certain anisotropic Banach space of currents in $\ml{D}^{\prime k}(V_a)$.
Now the key observation is to note
that the spectral projector $\pi_{a,z_0}^{(k)}:\Omega_c^k(V_a)\mapsto \ml{D}^{\prime k}(V_a)$, whose existence follows from 
the work~\cite{BaTs08},
is just the localized version
of the global spectral projector $\pi^{(k)}_{z_0}$
whose existence follows from
proposition~\ref{p:correlation}. Indeed, for $\psi_1\in\Omega_c^k(V_a)$, we find that
\begin{equation}\label{e:baladi-tsujii}\forall \psi_1\in \Omega^k_c(V_a),\quad \pi_{z_0}^{(k)}(\psi_1)=\pi_{a,z_0}^{(k)}(\psi_1),\end{equation}
where \emph{equality holds in the sense of currents in} $\ml{D}^{\prime k}(V_a)$.
The above means that every element of $\{z_0 : \zeta^{(k)}_{R,a}(z_0)=0 \}$
contributes to the set $\mathcal{R}_k$ of Pollicott--Ruelle resonances
of the transfer operator
acting on $k$-forms. In the next paragraph, we shall prove
that $\mathcal{R}_k$ exactly equals the union over $\text{Crit}(f)$ of local spectras
$$\mathcal{R}_k=\bigcup_{a\in \text{Crit}(f)}\{z_0 : \zeta^{(k)}_{R,a}(z_0)=0 \}$$
where the zeros are counted with multiplicity.

\subsection{Gluing local spectras}

 The main purpose of this section is to prove the following statement:
\begin{prop}\label{p:multiplicity} Let $0\leq k\leq n$ and let $z_0\in\IR$. Then, one has
$$\operatorname{Rk}\left(\pi_{z_0}^{(k)}\right)=\sum_{a\in\operatorname{Crit}(f)}\operatorname{Rk}\left(\pi_{a,z_0}^{(k)}\right).$$
\end{prop}
In particular, as was already explained, one can deduce from~\cite{BaTs08, Ba16} that $\operatorname{Rk}\left(\pi_{z_0}^{(k)}\right)$ 
is equal to the multiplicity of $z_0$ as a zero of $\zeta_R^{(k)}$. Note that this may be equal to $0$ if $z_0$ does not belong to
the set $\mathcal{R}_k$ of resonances.

\subsubsection{Construction of a ``good'' basis of Pollicott-Ruelle resonant states}
\label{p:goodbasisconstruction}
Let $z_0$ be an element in $\ml{R}_k$. 
We fix $(U_j)_{j=1,\ldots m_{z_0}}$ to be a basis of the range of $\pi_{z_0}^{(k)}$.
These are generalized eigenstates of eigenvalue 
$z_0$ for $-\ml{L}_{V_f}$ 
acting on suitable anisotropic Sobolev space of
currents of degree $k$. 
We aim 
at showing that we can choose this family in such a way that $\text{supp}(U_j)\subset\overline{W^u(a)}$ 
for some critical point $a$ of $f$.
Intuitively, the reader can think that we are looking for a
``good'' basis of generalized eigencurrents
with \textbf{minimal possible support} which by some
propagation argument 
should be the closure of unstable manifolds.
For that 
purpose, we shall prove that 
$U_1\in \mathcal{D}^{\prime,k}(M)$ 
can be decomposed as a sum of currents 
inside the range of $\pi_{z_0}^{(k)}$, 
each of them being supported by some $\overline{W^u(a)}$. 
By~\cite[Lemma~7.7]{DaRi17b} which is a propagation Lemma aimed at controlling supports of generalized eigencurrents, we know that  
if the current $U_1\in \mathcal{D}^{\prime,k}(M)$ is identically 
$0$ on a certain open set $V$ then this vanishing property propagates by the flow 
and $U_1$ 
vanishes identically on $\cup_{t\in\IR}\varphi_f^t(V)$. 
We set $\text{Max}(U_1)$ to be the set 
of critical points $a$ of $f$ such that the germ
of current $U_1$ 
is not identically zero near $a$ and such that, 
for every $b\succ a$, the germ of current $U_1$ 
identically 
vanishes near $b$. In particular, 
this means that, for every $a$ in $\text{Max}(U_1)$, the germ of
current
$U_1$ is supported by $W^u(a)$ in a neighborhood of $a$ by~\cite[Lemma 7.8]{DaRi17b}
which gives a control on the support of generalized eigencurrents near maximal elements of
$\text{Crit}(f)$. 

Our strategy to convert $U_1$ into currents with minimal support
is to \emph{cut and project}. Indeed, we multiply
$U_1$ by cut--off functions near some maximal critical point $a$ then 
we project the cut--off current to get back some generalized eigencurrent
which coincides with $U_1$ near $a$ but has minimal support.  
We will implicitely use the fact that anisotropic Sobolev spaces 
of currents are $C^\infty(M)$-modules which can be seen as follows~:
$u\in \mathcal{H}_k^{m_{N_0,N_1}}\Leftrightarrow \widehat{A}_N u\in L^2(M)$. Hence,
$$\forall\psi\in C^\infty(M),  \widehat{A}_N\left(\psi u\right)=\underset{\in \Psi^0(M)}{\underbrace{\widehat{A}_N\psi \widehat{A}_N^{-1}}} \underset{\in L^2}{\underbrace{\widehat{A}_N u}}\in L^2(M) $$
where we used the composition for pseudodifferential operators~\cite[Th.~8, p.~39]{FRS08} and elements in $\Psi^0(M)$ are bounded in $L^2$.

We get back to our proof. 
For every critical point $a$, we set $\chi_a$ to be a smooth cutoff function which is identically equal to $1$ near $a$ and 
$\chi_a$ vanishes away from $a$. Then, for every $a$ in $\text{Max}(U_1)$, we define
$$\tilde{U}_1(a):=\pi_{z_0}^{(k)}\left(\chi_aU_1\right).$$
We now apply Proposition~\ref{p:correlation} 
to the test current $\chi_aU_1$ belonging to 
some anisotropic Sobolev space and to 
some test form $\psi_2$ in $\Omega^{n-k}(M)$. 
If we choose $\psi_2$ compactly supported 
in $M-\overline{W^u(a)}$, then we can verify that
$$\forall\ t\geq 0,\ 
\int_M\varphi_f^{-t*}(\chi_aU_1)\wedge\psi_2=0.$$
But from the existence of the asymptotic expansion
on the right hand side~:
$$\int_M\varphi^{-t*}(\chi_a\psi_1)\wedge\psi_2=
\sum_{z_0\in\ml{R}_k:z_0>-\Lambda}e^{tz_0}\sum_{l=0}^{d_{z_0}^{(k)}-1}\frac{t^l}{l!}
\int_{M}\left(\ml{L}_{V_f}^{(k)}+z_0\right)^{l}
\left(\pi_{z_0}^{(k)}(\chi_a\psi_1)\right)\wedge\psi_2$$
$$+\ml{O}\left(e^{-\Lambda t}\|\chi_a\psi_1\|_{\ml{H}^{m_{N_0,N_1}}_k}\|\psi_2\|_{\ml{H}^{-m_{N_0,N_1}}_{n-k}}\right),$$ 
we find that
$$\forall \psi_2 \text{ s.t. }\text{supp}(\psi_2)\cap \overline{W^u(a)}=\emptyset, 
\int_{M}\pi_{z_0}^{(k)}\left(\chi_aU_1\right)\wedge \psi_2=0.$$
This implies that $\tilde{U}_1(a)$ is supported by $\overline{W^u(a)}$. If we now choose $\psi_2$ to be compactly supported in the neighborhood of 
$a$ where $\chi_a=1$, then one has
$$\int_M\varphi_f^{-t*}(\chi_aU_1)\wedge\psi_2=\int_M\varphi_f^{-t*}\left(U_1\right)\wedge\psi_2,$$
where we used the fact that $U_1$ is supported by $\overline{W^u(a)}$. Applying the asymptotic expansion of Proposition~\ref{p:correlation} one more 
time to the left hand side of the above equality, we find that $\tilde{U}_1(a)=\pi_{z_0}^{(k)}(\chi_aU_1)$ is equal to $U_1=\pi_{z_0}^{(k)}(U_1)$ in a neighborhood of $a$. 
We define 
$$\tilde{U}_1=U_1-\sum_{a\in\text{Max}(U_1)}\tilde{U}_1(a),$$
which by construction still belongs to the range of $\pi_{z_0}^{(k)}$ and which is now identically $0$ in a neighborhood of each 
$b$ satisfying $b\succeq a$ for the chosen element $a$ in $\text{Max}(U_1)$. Then, either $\tilde{U}_1=0$ in which case $U_1=\sum_a\tilde{U}_1(a)$ is decomposed 
with this minimal support property and we are done. Otherwise, we repeat the above argument 
with $\tilde{U}_1$ instead of $U_1$ and deal with critical points which 
are smaller for Smale's partial order relation. As there is only a finite number of critical points to exhaust, 
this procedure will end after a finite number of steps 
and we will find that
$$U_1=\sum_{a\in\text{Crit}(f)}\tilde{U}_1(a),$$
where either $\tilde{U}_1(a)=0$, 
or the support of $\tilde{U}_1(a)$ is contained in $\overline{W^u(a)}$. 
For the moment this procedure gives
a sequence of integers $m_a^{(k)}(z_0)\in \mathbb{N}, a\in \text{Crit}(f)$ and some
family of currents $(U_{j,a}(z_0))_{a\in \text{Crit}(f), 1\leqslant j\leqslant m_a^{(k)}(z_0)}$
which spans the image of $\pi_{z_0}^{(k)}$ and such that each $U_{j,a}(z_0)$ is supported
in $\overline{W^u(a)}$.
Note that our family of currents may not be linearly independent and we can extract a subfamily to make it into a basis of $\text{Ran}(\pi_{z_0}^{(k)})$. 
 However, we warn the reader that the notion
of linear independence we need  
is a little bit subtle and depends on the open subset in which we consider our current. 
Indeed,
we may have some currents which are linearly independent as elements in
$\ml{D}^{\prime,k}(M)$ but become dependent when we restrict them to smaller open subsets $U\subset M$.
Also, we can always construct some currents  
which are independent as elements
of $\ml{D}^{\prime,k}(U)$ but are dependent as elements of $\ml{D}^{\prime,k}(V)$ if $(U,V)$ are distinct
open subsets of $M$.
\begin{def1}[Independent germs at some given point]
A family of currents $(u_i)_{i\in I}$ in $\ml{D}^{\prime,k}(M)$ 
are linearly independent germs
at $a\in M$, if for all open neighborhoods $V_a$ of $a$, 
$(u_i)_{i\in I}$ are linearly independent as elements of $\ml{D}^{\prime,k}(V_a)$. 
\end{def1}

We start from a critical point $a$ such 
that $(U_j(a,z_0))_{j=1,\ldots,m_a^{(k)}(z_0)}$ are not independent germs at $a$ and, 
for every $b\succ a$, $(U_j(b,z_0))_{j=1,\ldots,m_b^{(k)}(z_0)}$ are linearly independent germs at $b$.
We next define a method to \emph{localize the linear dependence} near $a$ as follows.  
\begin{def1}[Local rank of germs at some point]
Consider the family of currents $(U_j(a,z_0))_{j=1,\ldots,m_a^{(k)}(z_0)}$.
Define a sequence $B_a(n)$ of balls of radius $\frac{1}{n}$ around $a$.
Consider the sequence
$r_n=\text{Rank}(U_j(a,z_0)|_{B_a(n)})_{j=1,\ldots,m_a^{(k)}(z_0)}$
where each $U_j(a,z_0)|_{B_a(n)}\in \ml{D}^{\prime,k}(B_a(n))$ is the restriction of 
$U_j(a,z_0)\in \ml{D}^{\prime,k}(M)$
to the ball $B_a(n)$.We call
$\lim_{n\rightarrow +\infty} r_n$ the rank of the germs
$(U_j(a,z_0))_{j=1,\ldots,m_a^{(k)}(z_0)}$ at $a$.
\end{def1}

If $\lim_{n\rightarrow +\infty} r_n<m_a^{(k)}(z_0)$, then there exists an open neighborhood
$V_a$ of $a$
such that the currents
$(U_j(a,z_0)|_{V_a})_{j=1,\ldots,m_a^{(k)}(z_0)}$
are linearly dependent in $\ml{D}^{\prime,k}(V_a)$ and the open subset $V_a$ is optimal as one cannot find 
a smaller open subset around $a$ on which one could write new linear relations among $(U_j(a,z_0))_{j=1,\ldots,m_a^{(k)}(z_0)}$.
It means that one can find some $j$ 
(say $j=1$) such that, on the open set $V_a$,
$$U_1(a,z_0)=\sum_{j=2}^{m_a^{(k)}(z_0)}\alpha_jU_j(a,z_0).$$
Then, we set
$$\tilde{U}(z_0)=U_1(a,z_0)-\sum_{j=2}^{m_a^{(k)}(z_0)}\alpha_jU_j(a,z_0),$$
which is equal to $0$ near $a$. Hence, by propagation~\cite[Lemma 7.7]{DaRi17b}, 
$\tilde{U}(z_0)$ is supported inside $\overline{W^u(a)}-W^u(a)$. 
Thus, proceeding by induction on 
Smale's partial order relation, we can without 
loss of generality suppose that, for every 
critical point $a$, the currents 
$(U_j(a,z_0))_{j=1,\ldots,m_a^{(k)}(z_0)}$ are linearly independent germs at $a$ and not only as elements of $\ml{D}^{\prime,k}(M)$. 

To summarize, we have proved~:
\begin{lemm}\label{l:goodbasis} Let $0\leq k\leq n$ and $z_0$ be an element of $\ml{R}_k$. For every
$a\in \text{Crit}(f)$, there exist
an integer $m_a^{(k)}(z_0)\geq 0$ together with a corresponding basis of generalized 
eigencurrents 
$$\left\{U_j(a,z_0):\ a\in\operatorname{Crit}(f),\ 1\leq j\leq m_a^{(k)}(z_0)\right\}$$
of the range of $\pi_{z_0}^{(k)}$ satisfying the following properties
$$\forall\ a\in\operatorname{Crit}(f),\ \forall\ 1\leq j\leq m_a^{(k)}(z_0),\ \operatorname{supp}(U_j(a,z_0))\subset\overline{W^u(a)},$$
and, for all $a\in \text{Crit}(f)$,
the family $\left(U_j(a,z_0)\right)_{j=1}^{m_a^{(k)}(z_0)}$ are independent germs at $a$.
\end{lemm}
We denote by
$$\left\{S_j(a,z_0):\ a\in\operatorname{Crit}(f),\ 1\leq j\leq m_a^{(k)}(z_0)\right\},$$
the dual basis (given by the adjoint operator) of this ``good'' basis. In particular, the spectral projector $\pi_{z_0}^{(k)}$ can be written as follows:
\begin{equation}\label{e:spectral-projector}\forall\psi_1\in\Omega^k(M),\quad\pi_{z_0}^{(k)}(\psi_1)=\sum_{a\in\text{Crit}(f)}
\sum_{j=1}^{m_a^{(k)}(z_0)}\left(\int_M\psi_1\wedge S_j(a,z_0)\right)U_j(a,z_0).\end{equation}
The currents $(S_j(a,z_0))_{j,a,z_0}$ are generalized eigenmodes for the dual operator $(-\ml{L}_{V_f}^{(k)})^{\dagger}=-\ml{L}_{V_{-f}}^{(n-k)}$ acting on the 
anisotropic Sobolev space $\ml{H}^{-m}_{n-k}(M)$. Also, from the definition of the duality pairing, one has, for every critical points $(a,b)$, for every indices $(j,k)$ and for every 
$(z,z')$ in $\ml{R}_k$,
\begin{equation}
\left\langle U_k(b,z^\prime),S_j(a,z)   \right\rangle=\int_M  U_k(b,z^\prime)\wedge S_j(a,z)=\delta_{jk}\delta_{zz^\prime}\delta_{ab}.
\end{equation}

\subsubsection{Support of the dual basis}

We would like to show that the dual basis 
$$\left\{S_j(a,z_0):\ a\in\operatorname{Crit}(f),\ 1\leq j\leq m_a^{(k)}(z_0)\right\}$$ defined above 
contains currents with minimal support. In fact, we will prove that
\begin{lemm}\label{l:dualbasissupport}
For all $z_0\in \ml{R}_k$,
the above dual basis satisfies the condition~:
$$\boxed{\forall\ a\in\operatorname{Crit}(f),\ \forall\ 1\leq j\leq m_a^{(k)}(z_0),\ \operatorname{supp}(S_j(a,z_0))\subset\overline{W^s(a)}.}$$
\end{lemm}
The above bound on the support of the dual basis actually shows that~:
\begin{equation}\label{e:intersection-support}\text{supp}\left(S_j(a,z_0)\right)\cap \text{supp}\left(U_j(a,z_0)\right)=\{a\}.\end{equation}
\begin{proof} Let $0\leq k\leq n$ and let $z_0\in\ml{R}_k$. We shall prove this Lemma by induction on Smale's partial order relation $\succeq$. 
In that manner, it is sufficient to prove that, for every $a\in\text{Crit}(f)$ such that the conclusion of the Lemma holds for all\footnote{Note that $a$ 
may be a minimum and, in that case, there is no such $b$.} $b\succ a$, one has 
$$\forall\ 1\leq j\leq m_a^{(k)}(z_0),\ \operatorname{supp}(S_j(a,z_0))\subset\overline{W^s(a)}.$$
 Fix such a critical point $a$ and $\psi_1$ compactly supported in $M-\overline{W^s(a)}$. Then, we consider $V_a$ to be a small enough neighborhood 
 of $a$ which does not interesect the support of $\psi_1$ and we fix $\psi_2$ in $\Omega_c^k(V_a)$. 
 From~\cite[Remark 4.5 p.~17]{dangrivieremorsesmale1}, we know that, if $V_a$ is chosen small enough, then $\varphi_f^{-t}(V_a)$ remains inside the 
 complementary of $\text{supp}(\psi_1)$ for $t\geq 0$. In particular, for every $t\geq 0$, $\varphi_f^{-t*}(\psi_1)\wedge\psi_2=0$ for every $t\geq 0$. 
 Applying the asymptotic expansion of Proposition~\eqref{p:correlation}, we then find that
 $$\int_M\pi_{z_0}^{(k)}(\psi_1)\wedge\psi_2=0,$$
for every $\psi_2$ in $\Omega_c^k(V_a)$. From~\eqref{e:spectral-projector}, this can be rewritten as
$$\forall\psi_2\in\Omega^k_c(V_a),\quad\sum_{b\in\text{Crit}(f)}
\sum_{j=1}^{m_b^{(k)}(z_0)}\left(\int_M\psi_1\wedge S_j(b,z_0)\right)\left(\int_M U_j(b,z_0)\wedge \psi_2\right)=0.$$
As $V_a$ is a small neighborhood of $a$ and as $U_j(b,z_0)$ is carried by $\overline{W^s(b)}$, we can apply Smale's 
Theorem~\ref{t:smale} in order to verify that only the points $b$ such that $b\succeq a$ contribute to the above sum, i.e.
$$\forall\psi_2\in\Omega^k_c(V_a),\quad\sum_{b\in\text{Crit}(f):b\succeq a}
\sum_{j=1}^{m_b^{(k)}(z_0)}\left(\int_M\psi_1\wedge S_j(b,z_0)\right)\left(\int_M U_j(b,z_0)\wedge \psi_2\right)=0.$$
We can now use our assumption on $a$ and the fact that $\overline{W^s(b)}\subset\overline{W^s(a)}$ for $b\succeq a$ in order to get 
$$\forall\psi_2\in\Omega^k_c(V_a),\quad
\sum_{j=1}^{m_a^{(k)}(z_0)}\left(\int_M\psi_1\wedge S_j(a,z_0)\right)\left(\int_M U_j(a,z_0)\wedge \psi_2\right)=0.$$
As the germs of currents are independent at $a$, we can deduce that $\int_M\psi_1\wedge S_j(a,z_0)=0$ for every 
$1\leq j \leq m_a^{(k)}(z_0)$, which concludes the proof of the Lemma.
\end{proof}

\subsection{Proof of Proposition~\ref{p:multiplicity}} We can now conclude the proof of Proposition~\ref{p:multiplicity}. With the above conventions, 
it is sufficient to show that $m_a^{(k)}(z_0)=\text{Rk}(\pi_{z_0,a}^{(k)})$. Hence, we fix a critical point $a$ and thanks 
to~\eqref{e:baladi-tsujii}, we can write that, for every $\psi_1$ in $\Omega^{k}_c(V_a)$,
$\pi_{z_0,a}^{(k)}(\psi_1)|_{V_a}=\sum_{b\in \text{Crit}(f)}\sum_{j=1}^{m_b^{(k)}(z_0)}\left(\int_M\psi_1\wedge S_j(b,z_0)\right) 
U_j(b,z_0)|_{V_a} .$
Now we choose $V_a$ small enough around $a$ such that $V_a\cap \overline{W^u(b)}=\emptyset$ (resp $V_a\cap \overline{W^s(b)}=\emptyset$) 
unless $b \succeq a$ (resp unless $b\preceq a$). 
Then 
$S_j(b,z_0)\wedge \psi_1=0$ unless $b\preceq a$
because $\text{supp}(S_j(b,z_0))\subset \overline{W^s(b)}$ does not meet
$V_a$ hence $\text{supp}(\psi_1)$.
In the same manner, $U_j(b,z_0)|_{V_a}=0 $
unless $b\succeq a$
since $\text{supp}(U_j(b,z_0))\subset \overline{W^u(b)}$ does not meet
$V_a$ unless $b\succeq a$. 
Therefore, all these cancellations imply that~:
$\sum_{b\in \text{Crit}(f)}\sum_{j=1}^{m_b^{(k)}(z_0)}\left(\int_M\psi_1\wedge S_j(b,z_0)\right) 
U_j(b,z_0)|_{V_a}=\sum_{j=1}^{m_a^{(k)}(z_0)}\left(\int_M\psi_1\wedge S_j(a,z_0)\right) 
U_j(a,z_0)|_{V_a}$ yielding~:
\begin{equation}\label{e:localtoglobal}\pi_{z_0,a}^{(k)}(\psi_1)=\sum_{j=1}^{m_a^{(k)}(z_0)}\left(\int_M\psi_1\wedge S_j(a,z_0)\right) 
U_j(a,z_0)|_{V_a}.\end{equation}
Thanks to Lemma~\ref{l:goodbasis}, 
we know that the currents $U_j(a,z_0)|_{V_a}$ are linearly independent in 
$\ml{D}^{\prime k}(V_a)$. Using~\eqref{e:intersection-support} and the fact that $S_j(a,z_0)$ is the dual basis of $U_j(a,z_0)$, we can verify that the 
$S_j(a,z_0)$ are also independent germs at $a$. Hence, one can verify that the range 
of $\pi_{z_0,a}^{(k)}$ is spanned by the currents $(U_j(a,z_0)|_{V_a})_{j=1,\ldots, m_a^{(k)}(z_0)}$ which concludes the proof 
of Proposition~\ref{p:multiplicity}.

\subsection{No Jordan blocks for $z_0=0$}
\label{ss:noJordan}
 Let $0\leq k\leq n$. Thanks to Appendix~\ref{a:holomorphic}, 
we know that the multiplicity of $1$ as a zero of $\zeta_{R}^{(k)}(z)$ is equal to 
the number of critical points of index $k$. On the other hand, given a critical 
point $a$ of index $l$, if we use Baladi-Tsujii's local result, we know that the 
multiplicity of the eigenvalue 
$1$ near $a$ is equal to $1$ if $k=l$ and to $0$ otherwise. Hence, if we use~\eqref{e:localtoglobal} 
combined with Proposition~\ref{p:multiplicity}, we can then deduce that, for 
$z_0=0$, one can find a basis of generalized eigencurrents for $\text{Ker}(\mathcal{L}_{V_f}^{(k)})^N$ (for some large enough $N$):
$$\left\{U_a:\ \text{dim}\ W^s(a)=k\right\},$$
whose support is equal to $\overline{W^u(a)}$. We would now like to verify that we can indeed pick $N=1$, 
equivalently that there is no Jordan blocks in the kernel. 
Suppose by contradiction that we have a nontrivial Jordan block, i.e. there exists $(u_0,u_1)$ such that
$$\ml{L}_{V_f}^{(k)}u_0=0\quad\text{and}\quad\ml{L}_{V_f}^{(k)}u_1=u_0.$$
We fix $a$ to be a critical point of index $k$ such that $u_0$ is not equal to $0$ near $a$. 
Such a point exists as $u_0$ is a linear combination of the 
$(U_b)_{b:\text{dim} W^s(b)=k}$. Recall from Smale's Theorem that, for every $b$ in $\text{Crit}(f)$, $\overline{W^u(b)}-W^u(b)$ is the union of unstable 
manifolds whose dimension is $<\text{dim}\ W^u(b)$. Hence, as $u_1$ is also a linear combination of the $(U_b)_{b:\text{dim} W^s(b)=k}$, we necessarily have that 
$u_1$ is proportional to $U_a$ near $a$. In a neighborhood of $a$, we then have $u_0=\alpha_0 U_a$ 
(with $\alpha_0\neq 0$) and $u_1=\alpha_1U_a$. If we use 
the eigenvalue equation, we find that, in a neighborhood of $a$:
$$\alpha_0\ml{L}_{V_f}^{(k)}U_a=0\quad\text{and}\quad\alpha_1\ml{L}_{V_f}^{(k)}U_a=\alpha_0 U_a.$$
As $U_a$ is not identically $0$ near $a$, we find the expected contradiction.

We next prove the following Lemma on the local structure of eigencurrents
in $\text{Ker}(\ml{L}_{V_f})$ near critical points~:
\begin{lemm}\label{r:current} Let $y_0$ be a point inside $W^u(a)$. 
Then, one can find a local system of coordinates $(x_1,\ldots x_n)$ such that $W^u(a)$ is given locally near $y_0$ 
by $\left\{x_1=\ldots=x_r=0\right\},$ where $r$ is the index of $a$
and the current
$[W^u(a)]=\delta_0(x_1,\ldots,x_r)dx_1\wedge\ldots \wedge dx_r$ coincides with $U_a$ near $y_0$. 
Similarly, one has $S_a=[W^s(a)]$ near $a$.
\end{lemm}
\begin{proof}
Recall from~\cite{Sm60, Web06} that $W^u(a)$ is an embedded submanifold inside $M$. 
Then, there is a local system of coordinates $(x_1,\ldots x_n)$ such that $W^u(a)$ is given locally near $y_0$ 
by $\left\{x_1=\ldots=x_r=0\right\},$ where $r$ is the index of $a$.  
The current of integration on $W^u(a)$, for the choice of orientation
given by $[dx_1\wedge\ldots \wedge dx_r]$ (see \cite[appendix D]{DaRiequi} for a discussion about orientations of integration currents), 
reads in this system of coordinates~: 
$[W^u(a)]=\delta_0(x_1,\ldots,x_r)dx_1\wedge\ldots \wedge dx_r$ by~\cite[Corollary D.4]{DaRiequi}.  
Moreover, for all test form $\omega$ whose support does not meet the boundary
$\partial W^u(a)=\overline{W^u(a)}\setminus W^u(a)$, one has for all $t\in \mathbb{R}$ the identity~:
$\left\langle \varphi_f^{-t*}[W^u(a)],\omega\right\rangle =\int_{W^u(a)} \varphi_f^{t*}\omega=\int_{\varphi_f^{-t}(W^u(a))=W^u(a)} \omega=
\left\langle [W^u(a)],\omega\right\rangle $
since $\varphi^t_f: M\mapsto M$ is an orientation preserving diffeomorphism which leaves 
$W^u(a)$ invariant.
This implies that in the weak sense $\varphi_f^{-t*}[W^u(a)]=[W^u(a)], \forall t\in \mathbb{R}$ hence $\ml{L}_{V_f}([W^u(a)])=0$. 
Near $a$, $[W^u(a)]$ belongs to the anisotropic Sobolev space 
$\ml{H}^{m_{N_0,N_1}}_r(M)$ for $N_0,N_1$ large enough. 
Hence, if we fix a smooth cutoff function $\chi_a$ near $a$, we can verify, by a propagation argument 
similar to the ones used to prove Lemma~\ref{l:goodbasis}, that $U_a$ can be chosen equal to $\pi_0^{(r)}(\chi_a[W^u(a)])$, 
and one has $U_a=[W^u(a)]$ 
near $a$. Similarly, one has $S_a=[W^s(a)]$ near $a$.
\end{proof}

\section{Proofs of Theorems~\ref{t:maintheo-harveylawsonwitten} to~\ref{t:maintheo-tunneling}}\label{s:proofs}

In this section, we collect the different informations we proved so far and prove the main statements of the introduction except for Theorem~\ref{t:fukaya}
that will be proved in section~\ref{s:fukaya}.

\subsection{Proof of Theorems~\ref{t:maintheo-flow} and~\ref{t:maintheo-current}}\label{ss:proof-ruelle} 
We start by proving the results on the limit operator $-\ml{L}_{V_f}$. Recall that 
these results were already proved in~\cite{dang2016spectral, DaRi17b} under the more restrictive assumptions that the Morse-Smale gradient 
flow is $\ml{C}^{\infty}$-linearizable near its critical points. In order to prove Theorem~\ref{t:maintheo-flow}, we first make use of 
Proposition~\ref{p:correlation} and we integrate the asymptotic expansion from  
Proposition~\ref{p:correlation}
against the function $e^{-zt}$ over 
$[0,+\infty)$~: $
\hat{C}_{\psi_1,\psi_2}(z)
=
\int_0^\infty dt e^{-zt}
\int_M  \varphi^{-t*}_f\psi_1 \wedge\psi_2$.
 This shows that $\hat{C}_{\psi_1,\psi_2}$ has a meromorphic extension to $\IC$ whose poles are contained inside $\IR_-$~:
$$\hat{C}_{\psi_1,\psi_2}(z)=\sum_{z_0\in\ml{R}_k:z_0>-\Lambda}\sum_{l=0}^{d_{z_0}^{(k)}-1}(-1)^l\frac{\int_{M}\left(\ml{L}_{V_f}^{(k)}+z_0\right)^{l}
\left(\pi_{z_0}^{(k)}(\psi_1)\right)\wedge\psi_2}{(z-z_0)^{l+1}} +\text{holomorphic part}.$$
Now, it 
remains to identify these poles. For that purpose, we use the results of section~\ref{s:Pollicott-Ruelle} 
to determine the rank of each residue. More 
specifically, according to Proposition~\ref{p:multiplicity}, 
this is equal to the sum of the rank of the local spectral projectors of Baladi and Tsujii. 
Then, using that this local spectrum is determined by the zeros of the dynamical Ruelle determinant (see the remark following 
Proposition~\ref{p:multiplicity}), we can conclude the proof of Theorem~\ref{t:maintheo-flow}.

To prove Theorem~\ref{t:maintheo-current}, we just need to use Proposition~\ref{p:correlation} combined with 
the fact that we showed in subsection~\ref{ss:noJordan} the 
absence of Jordan blocks in the kernel of the operator $-\ml{L}_{V_f}^{(k)}$.

\subsection{Proof of Theorem~\ref{t:maintheo-instanton}}
Regarding the limit operator, it now remains to show 
the instanton formula of Theorem~\ref{t:maintheo-instanton}. 
For that purpose, we first discuss some orientations issues on curves connecting some pair $(a,b)$ of critical points of $f$.
Choosing some orientation of every unstable manifolds
$(W^u(a))_{a\in \text{Crit}(f)}$ defines a local germ of current
$[W^u(a)]$ near every critical point $a$ and some 
integration current in $\ml{D}^{\prime,\bullet}\left(M\setminus\partial W^u(a)\right)$. 
Both Theorem~\ref{t:maintheo-current} and Lemma~\ref{r:current} 
show us that each germ $[W^u(a)]$
extends 
into a globally well--defined current $U_a$ on $M$ 
which coincides with $[W^u(a)]$ on $M\setminus\partial W^u(a)$. 
As $M$ is oriented, the orientation of $W^u(a)$ induces a
canonical \textbf{coorientation} on $W^s(a)$
so that the intersection pairing
at the level of currents gives $\int_M \chi [W^u(a)]\wedge [W^s(a)]=\chi(a)$ for every $a\in \text{Crit}(f)$ and for all
smooth $\chi$ compactly supported near $a$. Given any two critical points $(a,b)$ verifying $\text{ind}(a)=\text{ind}(b)+1$, recall 
from~\cite[Prop.~3.6]{Web06} 
that there exists finitely many flow lines connecting $a$ and $b$. These curves are called instantons and we shall denote them by 
$\gamma_{ab}$. Such a curve is naturally oriented by the gradient vector field $V_f$, 
hence defines a current of integration of degree $n-1$, $[\gamma_{ab}]\in \mathcal{D}^{\prime, n-1}(M)$.
\begin{def1} 
We define an orientation coefficient
$\sigma(\gamma_{ab})\in \{\pm 1\}$
by the following relation~:
\begin{equation}
\boxed{[\gamma_{ab}]=\sigma(\gamma_{ab})[W^u(a)]\wedge [W^s(b)]}
\end{equation}
in the neighborhood of some $x\in \gamma_{ab}$ where $x$ differs from both $(a,b)$.
\end{def1} 
From the Smale transversality assumption -- see Appendix~\ref{a:order-function}, one has, for 
$x\in \gamma_{ab}\setminus\{a,b\}$, the intersection 
of the conormals $N^*(W^u(a))$ and $N^*(W^s(b))$ is empty. Hence, according to~\cite[p.~267]{Ho90} (see also~\cite{BrDaHe16} or section~\ref{s:fukaya}), 
it makes sense to 
consider the wedge product $[W^u(a)]\wedge [W^s(b)]$ near such a point $x$. Moreover, it defines, near $x$, the germ of integration current along 
$\gamma_{ab}$ using the next Lemma: 
\begin{lemm}\label{l:intersection-current}
Let $X,Y$ be two tranverse submanifolds of $M$ whose intersection is a submanifold denoted by $Z$.
Then choosing an orientation of $X,Y,M$ induces a canonical orientation of $Z$ such that near every point of $Z$, we have a local equation in the sense of currents
$[Z]=[X]\wedge [Y] $.
\end{lemm}
\begin{proof}
Thanks to the transversality assumption, we can use local coordinates $(x,y,h)$ where locally
$X=\{x=0\}$, $Y=\{y=0\}$
and $Z=\{x=0,y=0\}$
Hence, one has 
$$ [X]\wedge [Y]=\delta_{\{0\}}^{\mathbb{R}^p}(x)dx\wedge \delta_{\{0\}}^{\mathbb{R}^q}(y)dy=\delta^{\mathbb{R}^{p+q}}_{\{0\}}(x,y)dx\wedge dy=[Z]$$
by definition of integration currents.
\end{proof}
All together, this shows that the coefficient $\sigma(\gamma_{ab})$ is well defined.
In fact, using the flow, we see that the formula
$$
[\gamma_{ab}]=\sigma(\gamma_{ab}) [W^u(a)]\wedge [W^s(b)]$$
holds true on $M\setminus\{a,b\}$. We are now ready to prove Theorem~\ref{t:maintheo-instanton} by setting
$$
n_{ab}=(-1)^n\sum_{\gamma_{ab}}\sigma(\gamma_{ab}),$$
where the sum runs over critical points $b$ of index $\text{ind}(a)+1$. In other words, 
the integer $n_{ab}$ counts with sign the number of instantons connecting $a$ and $b$.
We first recall that, as $d$ commutes with $\ml{L}_{V_f}$ and as the currents $(U_a)_{a\in\text{Crit}(f)}$
are elements in\footnote{Recall also that this spectrum is intrinsic, i.e. independent of the 
choice of the anisotropic Sobolev space.} $\text{Ker}(\ml{L}_{V_f})$,
we already know that
$$dU_a=\sum_{b:\text{ind}(b)=\text{ind}(a)+1}n_{ab}'U_b$$
where the coefficients $n_{ab}'$ are a priori real numbers. The goal is to prove that they are indeed equal 
to the integer coefficients $n_{ab}$ we have just defined. Let $a$ be 
some critical point of $f$ of index $k$. 
Choose some arbitrary cutoff function $\chi$ such that
$\chi=1$ in a small neighborhood of $a$ and $\chi=0$ outside some slightly bigger neighborhood of $a$.
Then the following identity
holds true in the sense of currents~:
$$ d\left(\chi[W^u(a)] \right)=d(\chi U_a)=d\chi\wedge U_a + \chi \wedge dU_a=d\chi\wedge [W^u(a)],$$
where we used the fact that $[W^u(a)]=U_a$ on the support of $\chi$, Smale's Theorem~\ref{t:smale} and the fact 
that $\chi\wedge dU_a=0$ since $dU_a$ is a linear combination of the $U_b$ with $\text{ind}(b)=\text{ind}(a)+1.$
In other words, we used the fact that the current $dU_a$ is supported by $\partial W^u(a)$.

Choose now some critical point $b$ such that
$\text{ind}(b)=\text{ind}(a)+1$. Then, for a small open neighborhood $O$ of $\{a\}\cup\partial W^u(a)$, we have the following identity
in the sense of currents in $\mathcal{D}^\prime(M\setminus O)$~:
\begin{equation}
[W^u(a)]\wedge [W^s(b)]|_{M\setminus O}=\sum_{\gamma_{ab}} \sigma(\gamma_{ab})[\gamma_{ab}]|_{M\setminus O}
\end{equation}
where the sum runs over instantons $\gamma_{ab}$ connecting $a$ and $b$. Recall from above that
the wedge product makes sense thanks to Smale's transversality assumption. 
We choose $O$ in such a way that
$O$ does not meet the support of $d\chi$, then the following 
identity holds true~:
\begin{eqnarray*}
\left\langle d(\chi[W^u(a)]),[W^s(b)]  \right\rangle & = &\int_M d\chi\wedge [W^u(a)]\wedge [W^s(b)]\\
&= &(-1)^{(n-1)}\sum_{\gamma_{ab}}\sigma(\gamma_{ab})\int_M[\gamma_{ab}]\wedge d\chi\\
 & = & (-1)^{n-1}\sum_{\gamma_{ab}}  \sigma(\gamma_{ab})\int_{\gamma_{ab}}d\chi \\
&= & (-1)^{n-1}\sum_{\gamma_{ab}}  \sigma(\gamma_{ab})\underset{0-1}{\underbrace{(\chi(b)-\chi(a))}}=n_{ab}.
\end{eqnarray*}
We just proved that, for any function $\chi$ such that
$\chi=1$ near $a$ and $\chi=0$ outside some slightly bigger neighborhood of $a$, one has
$$\left\langle  d(\chi[W^u(a)]) ,[W^s(b)]\right\rangle=n_{ab} .$$
Note that this equality remains true for any $\chi$ such that
$\chi=1$ near $a$ and $\chi=0$ in some neighborhood of $\partial W^u(a)=\overline{W^u(a)}\setminus W^u(a)$. In particular, it applies to
the pull--back $\varphi^{-t*}_f(\chi)$ for all $t\geq 0$. Recall in fact that $\varphi_f^{-t*}[W^u(a)]=[W^u(a)]$ on the support 
of $\varphi_f^{-t*}(\chi)$ by Lemma~\ref{r:current}. Still from this Lemma, 
one knows that $S_b=[W^s(b)]$ on the support of 
$d(\varphi_f^{-t*}(\chi))$. Therefore, one has also
$$\forall t\geq0,\quad \left\langle d \varphi_f^{-t*}(\chi[W^u(a)]),S_b \right\rangle
=\left\langle d \varphi_f^{-t*}(\chi[W^u(a)]),[W^s(b)] \right\rangle=n_{ab}.$$
Still from Lemma~\ref{r:current} and as $\chi$ is compactly supported near $a$, we know that, for an appropriate choice 
of integers $N_0,N_1$, the current $\chi[W^u(a)]$ belongs to the anisotropic Sobolev space $\ml{H}^{m_{N_0,N_1}}_k(M)$ and the spectrum 
of $-\ml{L}_{V_f}^{(k)}$ is discrete on some half plane $\text{Re}(z)>-c_0$ with $c_0>0$. Thanks to Lemma~\ref{l:spectrum-propagator} 
and to the fact that there is no Jordan blocks, we can conclude that, in the Sobolev space $\ml{H}^{m_{N_0,N_1}}_k(M)$,
$$\varphi_f^{-t*}(\chi[W^u(a)])\rightarrow \sum_{a'\in\text{Crit}(f):\text{ind}(a')=k}\left(\int_{M}(\chi[W^u(a)])\wedge S_{a'}\right)U_{a'},
\quad\text{as}\ t\rightarrow+\infty.$$
For every smooth test $(n-k)$-form $\psi_2$ compactly supported in $M\setminus \overline{W^u(a)}$, we can verify that
$$\forall t\geq0,\quad \varphi_f^{-t*}(\chi[W^u(a)])\wedge \psi_2=0,$$
which implies that the above reduces to
$$\varphi_f^{-t*}(\chi[W^u(a)])\rightarrow
\underset{=\left\langle U_a,S_a \right\rangle=1}{ \underbrace{\left(\int_{M}(\chi[W^u(a)])\wedge S_{a}\right)}}U_{a}=U_{a},
\quad\text{as}\ t\rightarrow+\infty,$$
since $\chi(a)=1$, $\text{supp}(S_a)\cap \text{supp}(\chi[W^u(a)])=\{a\} $ by equation~\eqref{e:intersection-support} 
and $S_a=[W^s(a)]$ near $a$.
Then, it follows from continuity of
$d:\ml{H}^{m_{N_0,N_1}}_k(M)\mapsto \ml{H}^{m_{N_0,N_1}-1}_{k+1}(M)$ that
$d\varphi^{-t*}(\chi[W^u(a)]) \rightarrow dU_a $ in $\ml{H}^{m_{N_0,N_1}-1}_{k+1}(M)$. Finally, by continuity of the duality
pairing $(u,v)\in \ml{H}^{m_{N_0,N_1}-1}_{k+1}(M)\times \ml{H}^{1-m_{N_0,N_1}}_{n-(k+1)}(M)\longmapsto \langle u,v\rangle$, 
we deduce that
$$n_{ab} =\lim_{t\rightarrow +\infty} \left\langle S_b,d \varphi^{-t*}(\chi[W^u(a)]) \right\rangle=\left\langle S_b , dU_a  \right\rangle.$$
This shows that
the complex $\left(\text{Ker}(\ml{L}_{V_f}),d\right)$ generated
by the currents $(U_a)_{a\in \text{Crit}(f)}$ is well--defined as a $\mathbb{Z}$--module.
Then, we note that
tensoring the above complex with $\mathbb{R}$ yields a complex $\left(\text{Ker}(\ml{L}_{V_f}),d\right)\otimes_\mathbb{Z} \mathbb{R}$
which is quasi--isomorphic to the De Rham complex $\left(\Omega^\bullet(M),d\right)$ of smooth forms
by~\cite[Theorem 2.1]{DaRi17c} as a consequence of the \textbf{chain homotopy equation}~\cite[paragraph 4.2]{DaRi17c}:
\begin{equation}
\boxed{\exists R:\Omega^{\bullet}(M)\mapsto \ml{D}^{\prime,\bullet-1}(M),\,\ Id-\pi_0=d\circ R + R\circ d.}
\end{equation}
This ends our proof of Theorem~\ref{t:maintheo-instanton}.

\subsection{Proof of the results on the Witten Laplacian}\label{ss:tunnelling} First of all, we note that
the result from Theorem~\ref{t:maintheo-harveylawsonwitten}~:
$$\lim_{\hbar\rightarrow 0^+}\int_M\mathbf{1}_{[0,\eps]}\left(W_{f,\hbar}^{(k)}\right)
 \left(e^{-\frac{f}{\hbar}}\psi_1\right)\wedge\left(e^{\frac{f}{\hbar}}\psi_2\right)=\lim_{t\rightarrow+\infty}\int_M\varphi_f^{-t*}(\psi_1)\wedge\psi_2   $$ 
is a direct consequence of 
Theorem~\ref{t:maintheo-proj} which yields a convergence of spectral 
projectors $\lim_{\hbar\rightarrow 0^+}\int_M\mathbf{1}_{[0,\eps]}\left(W_{f,\hbar}^{(k)}\right)
 \left(e^{-\frac{f}{\hbar}}\psi_1\right)\wedge\left(e^{\frac{f}{\hbar}}\psi_2\right)=\int_M\pi_{z_0}^{(k)}(\psi_1)\wedge \psi_2 $ combined with Theorem~\ref{t:maintheo-current} where the limit term $\lim_{t\rightarrow+\infty}\int_M\varphi_f^{-t*}(\psi_1)\wedge\psi_2$  
 is identified with the term
$\int_M\pi_{z_0}^{(k)}(\psi_1)\wedge \psi_2$ coming 
from the spectral projector corresponding to the 
eigenvalue $0$ -- see section~\ref{s:Pollicott-Ruelle}. 
Concerning 
Theorem~\ref{t:maintheo-eigenvalues}, we observe that it is also a consequence of 
the weak convergence of the spectral projectors given by Theorem~\ref{t:maintheo-proj}. To see that, 
recall that the rank of the spectral projectors of $-\ml{L}_{V_f}^{(k)}$ is completely 
determined by Proposition~\ref{p:multiplicity} and that the results of Baladi and Tsujii give 
a bijection between the zeros of the local Ruelle determinant 
and the rank of the local spectral projectors. Hence, it now remains to 
observe that Theorem~\ref{t:maintheo-proj}, which claims that the 
spectral projectors of the Witten Laplacian converge to the spectral projectors 
of $-\ml{L}_{V_f}$, follows straightforwardly from the content of Theorem~\ref{t:conv-proj}.

We now prove Theorem~\ref{t:maintheo-tunneling} about the Witten--Helffer--Sj\"ostrand tunnelling formula for our WKB states which becomes a direct corollary of Theorem~\ref{t:maintheo-instanton}. 
Indeed, our WKB states were defined by using the spectral projector 
on the small eigenvalues of the Witten Laplacian, i.e.
$$U_a(\hbar)=\mathbf{1}_{[0,\eps_0]}(W_{f,\hbar}^{(k)})\left(e^{\frac{f(a)-f}{\hbar}} U_a\right),$$
where $k$ is the index of the critical point. Thanks to Theorem~\ref{t:maintheo-instanton}, we already know
\begin{equation}\label{e:tunneling-before-projection}d_{f,\hbar}\left(e^{\frac{f(a)-f}{\hbar}} U_a\right)=
\sum_{b:\text{ind}(b)=\text{ind}(a)+1}n_{a,b}e^{-\frac{f(b)-f(a)}{\hbar}} e^{\frac{f(b)-f}{\hbar}}U_b.\end{equation}
Recall now that the spectral projector has the following integral expression
$$\mathbf{1}_{[0,\eps_0]}(W_{f,\hbar}^{(k)})=\frac{1}{2i\pi}\int_{\ml{C}(0,\eps_0)}(z-W_{f,\hbar})^{-1}dz.$$
Hence, $d_{f,\hbar}$ commutes with $\mathbf{1}_{[0,\eps_0]}(W_{f,\hbar}^{(\bullet)})$. It is then sufficient to apply the spectral 
projector to both sides of~\eqref{e:tunneling-before-projection} in order to conclude.

\section{Proof of Theorem~\ref{t:fukaya}}\label{s:fukaya}

In this section, we give the proof of Theorem~\ref{t:fukaya} which states that our WKB states verify the Fukaya's instanton formula. 
Using the conventions of Theorem~\ref{t:fukaya}, we start with the following 
observation~:
\begin{eqnarray*}U_{a_{ij}}(\hbar) & = & \mathbf{1}_{[0,\eps_0]}(W_{f_{ij},\hbar})\left( e^{\frac{f_{ij}(a_{ij})-f_{ij}(x)}{\hbar}}U_{a_{ij}}\right)\\
 & = &   e^{\frac{f_{ij}(a_{ij})-f_{ij}(x)}{\hbar}}\mathbf{1}_{[0,\eps_0]}\left(\ml{L}_{V_{f_{ij}}}+\frac{\hbar\Delta_{g_{ij}}}{2}\right)\left(U_{a_{ij}}\right) ,
\end{eqnarray*}
where $\eps_0>0$ is small enough and where $ij$ belongs to $\{12, 23,31\}$. Hence, we can deduce that~:
$$U_{a_{12}}(\hbar)\wedge U_{a_{23}}(\hbar)\wedge U_{a_{31}}(\hbar)=e^{\frac{f_{12}(a_{12})+f_{23}(a_{23})+f_{31}(a_{31})}{\hbar}}
\tilde{U}_{a_{12}}(\hbar)\wedge \tilde{U}_{a_{23}}(\hbar)\wedge \tilde{U}_{a_{31}}(\hbar),$$
where, for every $ij$ and for $\hbar>0$,
$$\tilde{U}_{a_{ij}}(\hbar):=\mathbf{1}_{[0,\eps_0]}\left(\ml{L}_{V_{f_{ij}}}+\frac{\hbar\Delta_{g_{ij}}}{2}\right)\left(U_{a_{ij}}\right),$$
while $\tilde{U}_a(0):=U_a$. Hence, the proof of Theorem~\ref{t:fukaya} consists in showing that
$$\int_M\tilde{U}_{a_{12}}(\hbar)\wedge \tilde{U}_{a_{23}}(\hbar)\wedge \tilde{U}_{a_{31}}(\hbar)$$
converges as $\hbar\rightarrow 0^+$ to $\int_MU_{a_{12}}\wedge U_{a_{23}}\wedge U_{a_{31}}$, and that this limit is an integer. In particular, we will already have to 
justify that $U_{a_{12}}\wedge U_{a_{23}}\wedge U_{a_{31}}$
is well defined. The proof will be divided in two steps. First, we will show that 
$(\tilde{U}_{a_{ij}}(\hbar))_{\hbar\rightarrow 0^+}$ defines a bounded sequence in some space of currents 
$\ml{D}^{\prime}_{\Gamma_{ij}}(M)$ with prescribed wavefront sets. Then, we will apply theorems on the continuity of 
wedge products for currents with transverse 
wavefront sets.

\subsection{Wavefront set of eigencurrents} In this first paragraph, we fix $V_f$ to be a smooth Morse-Smale gradient 
vector field which is $\ml{C}^1$-linearizable. Fix $0\leq k\leq n$ and $\Lambda>0$. Then, following section~\ref{s:anisotropic}, 
choose some large enough integers $N_0,N_1$ to ensure that
for every $0\leq \hbar< \hbar_0$, the operator
$$-\mathcal{L}_{V_f}-\frac{\hbar\Delta_g}{2}:
\Omega^k(M)\subset \ml{H}_k^{m_{N_0,N_1}}(M)\mapsto \ml{H}_k^{m_{N_0,N_1}}(M) $$
has a discrete spectrum with finite multiplicity on the domain 
$\text{Re}(z)>-\Lambda$. Recall from~\cite[Th.~1.5]{FS} that the eigenmodes are intrinsic and that they do not depend on the 
choice of the order function. Recall also from section~\ref{s:anisotropic} that, up to some uniform constants, the parameter $\Lambda$ 
has to be smaller than $c_0\min\{N_0,N_1\}$ which is the quantity appearing in Lemma~\ref{l:escape-function}. Hence, if we choose 
$N_1'\geq N_1$, we do not change the spectrum on $\text{Re}(z)>-\Lambda$. In particular, any generalized eigenmode
$U\in \ml{H}_k^{m_{N_0,N_1}}(M)$ associated 
with an eigenvalue $z_0$ belongs to any anisotropic Sobolev space $\ml{H}_k^{m_{N_0,N_1'}}(M)$ with $N_1'\geq N_1$. We also 
note from the proof of Theorem~\ref{t:conv-proj} that, for every $N_1'\geq N_1$,
\begin{equation}\label{e:conv-sobolev-norm}\left\|\tilde{U}_a(\hbar)-U_a\right\|_{\ml{H}_k^{m_{N_0,N_1'}}(M)}
 \rightarrow 0\quad\text{as}\quad\hbar\rightarrow 0.
\end{equation}
\begin{rema}
 Note that the proof in section~\ref{s:convergence} shows that the convergence is of order $\ml{O}(\hbar)$ but we omit this information 
 for simplicity of exposition.
\end{rema}
We now have to recall a few facts on the topology of the space $\ml{D}^{\prime, k}_{\Gamma_-(V_f)}(M)$ of currents
whose wavefront set is contained in the closed conic set $\Gamma_-(V_f)=\cup_{a\in \text{Crit}(f)} N^*(W^u(a))\subset T^*M\setminus \underline{0}$ which is defined in Appendix~\ref{a:order-function}. 
Note that we temporarily omit the 
dependence in $V_f$ as we only deal with one Morse function for the moment. 
Recall that on some vector space $E$, given some family of seminorms
$P$, we can define a topology on $E$ which makes it a locally convex topological vector space. A basis of neighborhood of the origin
is defined by the subsets $\{x\in E \text{ s.t. }P(x)<A\}$ with $A\in \mathbb{R}_{+}^*$ and  with $P$ a seminorm. 
In the particular case of currents, we will use the
strong topology~:
\begin{def1}[Strong topology and bounded subsets]\label{r:topology}
The strong topology of $\mathcal{D}^{\prime,k}(M)$ for $M$ compact is defined by the following seminorms. 
Choose some bounded set $B$ in $\Omega^{n-k}(M)$. 
Then, we define a seminorm $P_B$ as $P_B(u)=\sup_{\varphi\in B}\vert \langle u,\varphi \rangle \vert$. A subset $B$ of currents is bounded iff it is
weakly bounded which means for every test form $\varphi\in \Omega^{n-k}(M)$, 
$\sup_{t\in B}\vert \langle t,\varphi \rangle  \vert<+\infty$~\cite[Ch.~3, p.~72]{Schwartz-66}. This is equivalent to $B$ being bounded in some 
Sobolev space $H^s(M,\Lambda^k(T^*M))$ of currents by suitable application of the uniform boundedness principle~\cite[Sect.~5, Lemma~23]{DabBr14}.  
\end{def1}
We can now define the normal topology in the space of currents essentially following~\cite[Sect.~3]{BrDaHe16}:
\begin{def1}[Normal topology on the space of currents] For every closed conic subset $\Gamma\subset T^*M\setminus \underline{0}$,
the topology of $\mathcal{D}^{\prime,k}_{\Gamma}(M)$ is 
defined as the weakest topology which makes continuous the seminorms of the strong topology of $\mathcal{D}^{\prime,k}(M)$ and
the seminorms~:
\begin{eqnarray}
\Vert u\Vert_{N,C,\chi,\alpha,U}= \Vert (1+\|\xi\|)^{N} \ml{F}(u_\alpha\chi)(\xi) \Vert_{L^\infty(C)}
\end{eqnarray}
where $\chi$ is supported on some chart $U$, where $u=\sum_{\vert\alpha\vert=k} u_\alpha dx^\alpha$ where $\alpha$ is a multi--index, where $\ml{F}$ is the Fourier transform calculated in the local chart
and $C$ is a closed cone such that
$\left(\text{supp }\chi\times C\right)\cap \Gamma=\emptyset$. A subset $B\subset \mathcal{D}^{\prime,k}_\Gamma$ is called bounded
in $\mathcal{D}^{\prime k}_{\Gamma}$ if it is bounded in $\mathcal{D}^{\prime k}$ and if all seminorms 
$\Vert .\Vert_{N,C,\chi,\alpha,U}$ are bounded on $B$.
\end{def1}

We emphasize that this definition is given purely 
in terms of local charts without loss of generality. The above topology 
is in fact \emph{intrinsic as a consequence of the continuity of the
pull--back}~\cite[Prop 5.1 p.~211]{BrDaHe16} as emphasized by H\"ormander~\cite[p.~265]{Ho90}. 
Note that it is the same to consider currents or distributions
when we define the relevant topologies since currents are just elements of the form
$\sum u_{i_1,\dots,i_k} dx^{i_1}\wedge \dots \wedge dx^{i_k}$ in local coordinates
$(x^1,\dots,x^n)$ where the coefficients $u_{i_1,\dots,i_k}$ are distributions.

 Note from~\eqref{e:conv-sobolev-norm} that $(\tilde{U}_a(\hbar))_{0\leq \hbar<1}$ is a bounded family in the anisotropic Sobolev space 
 $\ml{H}_k^{m_{N_0,N_1}}(M)$ and is thus bounded in $H^{-s}(M,\Lambda^k(T^*M))$ for $s$ large enough. In particular, from
 definition~\ref{r:topology}, 
 it is a bounded family in $\ml{D}^{\prime,k}(M)$. 
 We would now like to verify that it is a bounded family in 
 $\ml{D}^{\prime,k}_{\Gamma_-}(M)$ which converges in 
 the normal topology as $\hbar$ goes to $0$ in order to apply the results from~\cite{BrDaHe16}. 
 For that purpose, we can already observe that, for some $s$ large enough,
 $\left\|\tilde{U}_a(\hbar)-U_a\right\|_{H^{-s}(M,\Lambda^k(T^*M))}
 \rightarrow 0\quad\text{as}\quad\hbar\rightarrow 0.$
In particular, it converges for the strong topology in $\ml{D}^{\prime,k}(M)$. 
Hence, it remains to discuss the boundedness and the convergence with 
respect to the seminorms $\Vert .\Vert_{N,C,\chi,\alpha,U}$. We note that these seminorms involve the $L^{\infty}$ norm while the 
anisotropic spaces we deal with so far are built from $L^2$ norms. This problem is handled by the 
following Lemma~:
\begin{lemm}[$L^2$ vs $L^\infty$]
\label{l:L2vsinfty}
Let $N$, $\tilde{N}$ be some positive integers and let $W_0$ be a closed cone in $\IR^{n*}$. Then, 
for every closed conic neighborhood $W$ of $W_0$, one can find a constant $C=C(N,\tilde{N},W)>0$ such that, for every 
$u$ in $\ml{C}^{\infty}_c(B_{\IR^n}(0,1))$, one has
$$\sup_{\xi\in W_0}(1+\vert\xi\vert)^N\vert\widehat{u}(\xi)\vert \leqslant C\left(\Vert (1+\vert\xi\vert)^N\widehat{u}(\xi)\Vert_{L^2(W)}+\Vert u\Vert_{H^{-\tilde{N}}}\right).$$
\end{lemm}
 
We postpone the proof of this Lemma to appendix~\ref{a:lemma} and we show first how to use it in our context.  We consider the family of 
currents $(\tilde{U}_a(\hbar))_{0\leq\hbar<\hbar_0}$ in $\ml{D}^{\prime,k}(M)$ and we would like to show that it is a bounded family in 
$\ml{D}^{\prime,k}_{\Gamma_-}(M)$ and that $\tilde{U}_a(\hbar)$ converges to $U_a$ in the normal topology we have 
just defined. Recall that this family 
is bounded and that we have convergence in every anisotropic Sobolev space with $\ml{H}_k^{m_{N_0,N_1^\prime}}(M)$ with $N_1^\prime$ large enough.
Fix $(x_0;\xi_0)\notin \Gamma_-$. Fix some $N>0$. Note that, up to shrinking the neighborhood used to define the order function in 
appendix~\ref{a:order-function} and up to increasing $N_1$, we can suppose that $m_{N_0,N_1'}(x;\xi)$ 
is larger than $N/2$ for any $N_1'$ and for every $(x;\xi)$ in a small conical neighborhood $W$ of $(x_0,\xi_0)$. 

Fix now a smooth test function $\chi$ supported near $x_0$ and a closed cone $W_0$ which is 
strictly contained in the conical neighborhood $W$
we have just defined. Thanks to Lemma~\ref{l:L2vsinfty} and to the Plancherel equality, the norm we have to estimate is
\begin{eqnarray*}\|(1+\|\xi\|)^{N}\ml{F}(\chi\tilde{U}_a(\hbar))\|_{L^{\infty}(W_0)} 
& \leq &  C\left(\|\chi_1(\xi)(1+\|\xi\|)^{N}\ml{F}(\chi\tilde{U}_a(\hbar))\|_{L^{2}}+\|\tilde{U}_a(\hbar)\|_{H^{-\tilde{N}}}\right)\\
& \leq &  C\left(\|\Op(\chi_1(\xi)(1+\|\xi\|)^{N}\chi)\tilde{U}_a(\hbar)\|_{L^{2}}+\|\tilde{U}_a(\hbar)\|_{H^{-\tilde{N}}}\right),\end{eqnarray*}
where $\chi_1\in\ml{C}^{\infty}$ is identically equal to $1$ on the conical neighborhood $W$ and equal to $0$ outside a slightly bigger 
neighborhood. For $\tilde{N}$ large enough, we can already observe that the second term $\|\tilde{U}_a(\hbar)\|_{H^{-\tilde{N}}}$ 
in the upper bound is uniformly bounded as $\tilde{U}_a(\hbar)$ is uniformly bounded in some fixed anisotropic 
Sobolev space. Hence, it remains to estimate
$$\|\Op\left(\chi_1(\xi)(1+\|\xi\|)^{N}\chi\right)\Op(\mathbf{A}_{N_0,N_1'}^{(k)})^{-1}\Op(\mathbf{A}_{N_0,N_1'}^{(k)})\tilde{U}_a(\hbar))\|_{L^{2}(M)}.$$
By composition of pseudodifferential operators and as we chose $N_1'$ large enough to ensure that the order function $m_{N_0, N_1}$ is larger than 
$N/2$ on $\text{supp}(\chi_1)$, 
we can deduce that this quantity is bounded (up to some constant) by 
$\|\Op(\mathbf{A}_{N_0,N_1'}^{(k)})\tilde{U}_a(\hbar)\|_{L^{2}}$ which is exactly the norm on the anisotropic Sobolev space. To summarize, 
this argument shows the following
\begin{prop}\label{p:wavefrontset} Let $V_f$ be a Morse-Smale gradient flow which is $\ml{C}^1$-linearizable. Then, there exists $\hbar_0>0$ such that, 
for every $0\leq k\leq n$ and for every $a\in\operatorname{Crit}(f)$ of index $k$, 
the family $(\tilde{U}_a(\hbar))_{0\leq\hbar<\hbar_0}$ is bounded in 
$\ml{D}^{\prime,k}_{\Gamma_-(V_f)}(M)$. Moreover, $\tilde{U}_a(\hbar)$ 
converges to $U_a$ for the normal topology in $\ml{D}^{\prime,k}_{\Gamma_-(V_f)}(M)$ as $\hbar\rightarrow 0^+$.
\end{prop}
This proposition is the key ingredient we need in order to apply the theoretical results from~\cite{BrDaHe16}. Before doing that, we can already observe that, 
if we come back to the framework of Theorem~\ref{t:fukaya}, then the generalized Morse-Smale assumptions ensures that the 
wavefront sets of the three family of currents are transverse. In particular, we can define the wedge product even 
for\footnote{For $\hbar\neq 0$, there is no problem as the eigenmodes are smooth by elliptic regularity.} $\hbar=0$, i.e. 
 $U_{a_{12}}\wedge U_{a_{23}}\wedge U_{a_{31}}$ defines an element in $\ml{D}^{\prime,n}(M)$.

\subsection{Convergence of products}
\label{ss:productscurrents}
Given two closed conic sets $(\Gamma_1,\Gamma_2)$ which have empty intersection, the usual
wedge product of smooth forms
$$\wedge:(\varphi_1,\varphi_2)\in \Omega^k(M) \times \Omega^l(M) \longmapsto \varphi_1\wedge\varphi_2\in \Omega^{k+l}(M)$$
extends uniquely as an hypocontinuous map for the normal topology~\cite[Th.~6.1]{BrDaHe16} 
$$\wedge:(\varphi_1,\varphi_2)\in \ml{D}^{\prime k}_{\Gamma_1}(M) \times \ml{D}^{\prime l}_{\Gamma_2}(M) \longmapsto \varphi_1\wedge
\varphi_2 \in\mathcal{D}^{\prime k+l}_{s(\Gamma_1,\Gamma_2)}(M),$$
with $s(\Gamma_1,\Gamma_2)=\Gamma_1\cup \Gamma_2\cup (\Gamma_1+\Gamma_2)$.
\begin{rema} The proof in~\cite{BrDaHe16} was given for product of distributions and it extends to currents as 
$\ml{D}^{\prime,k}_{\Gamma}(M)=\ml{D}^{\prime}_{\Gamma}(M)\otimes_{\ml{C}^{\infty}(M)}\Omega^k(M).$ Recall that the fact that $\wedge$ 
is hypocontinuous means that, for every neighborhood $W\subset \mathcal{D}^{\prime,k+l}_{s(\Gamma_1,\Gamma_2)}(M)$ of zero and for every
bounded set $B_2\subset \mathcal{D}^{\prime,l}_{\Gamma_2(M)}$, there is some open neighborhood
$U_1\subset \mathcal{D}^{\prime,k}_{\Gamma_1}$ of zero
such that for $i=1$ the condition reads $\wedge(U_1\times B_2)\subset W$. The same holds true if we invert 
the roles of $1$ and $2$. We note that hypocontinuity implies boundedness, in the sense that any bounded subset of 
$\ml{D}^{\prime k}_{\Gamma_1}(M) \times \ml{D}^{\prime l}_{\Gamma_2}(M)$ is sent to a bounded subset of 
$\mathcal{D}^{\prime k+l}_{s(\Gamma_1,\Gamma_2)}(M)$.
This follows from the observation that
a set $B$ is bounded iff for every open neighborhood  $U$ of $0$, 
$B$ can be rescaled by multiplication by $\lambda>0$ such that
$\lambda B\subset U$.
\end{rema}

Let us now come back to the proof of Theorem~\ref{t:fukaya}. Combining Proposition~\ref{p:wavefrontset} with the hypocontinuity of the 
wedge product, we find that $(\tilde{U}_{a_{12}}(\hbar)\wedge \tilde{U}_{a_{23}}(\hbar))_{0\leq\hbar<\hbar_0}$ is a bounded family in 
$\mathcal{D}^{\prime,k+l}_{s(\Gamma_-(V_{f_{12}}),\Gamma_-(V_{f_{23}}))}(M)$, where $k$ is the index of $a_{12}$ and $l$ is 
the one of $a_{23}$. Moreover, as $\hbar\rightarrow 0$, one has 
$$\tilde{U}_{a_{12}}(\hbar)\wedge \tilde{U}_{a_{23}}(\hbar)\rightarrow \tilde{U}_{a_{12}}(0)\wedge \tilde{U}_{a_{23}}(0)=
U_{a_{12}}\wedge U_{a_{23}},$$
for the normal topology of $\mathcal{D}^{\prime,k+l}_{s(\Gamma_-(V_{f_{12}}),\Gamma_-(V_{f_{23}}))}(M)$. As our three vector fields verify the 
generalized Morse--Smale assumptions, we can repeat this argument with the 
spaces $\mathcal{D}^{\prime,n-(k+l)}_{\Gamma_-(V_{f_{31}})}(M)$ and $\mathcal{D}^{\prime,k+l}_{s(\Gamma_-(V_{f_{12}}),\Gamma_-(V_{f_{23}}))}(M)$. Hence, we 
get that, as $\hbar\rightarrow 0$,
$$\tilde{U}_{a_{12}}(\hbar)\wedge \tilde{U}_{a_{23}}(\hbar)\wedge\tilde{U}_{a_{31}}(\hbar)\rightarrow 
U_{a_{12}}\wedge U_{a_{23}}\wedge U_{31},$$
in $\mathcal{D}^{\prime,s}_{s(s(\Gamma_-(V_{f_{12}}),\Gamma_-(V_{f_{23}})),\Gamma_-(V_{f_{31}}))}(M)$. Finally, testing against the smooth 
form $1$ in $\Omega^0(M)$, we find that, as $\hbar\rightarrow 0$,
$$\int_M\tilde{U}_{a_{12}}(\hbar)\wedge \tilde{U}_{a_{23}}(\hbar)\wedge\tilde{U}_{a_{31}}(\hbar)\rightarrow 
\int_M U_{a_{12}}\wedge U_{a_{23}}\wedge U_{31},$$
which concludes the proof of Theorem~\ref{t:fukaya} up to the fact that we need to verify that
$\int_M U_{a_{12}}\wedge U_{a_{23}}\wedge U_{31}$
is an integer.

\subsection{End of the proof}

In order to conclude the proof of Theorem~\ref{t:fukaya}, we will show that $\overline{W^u(a_{12})}$, $\overline{W^u(a_{23})}$ and $\overline{W^u(a_{31})}$ 
intersect transversally at finitely many points belonging to the intersection $W^u(a_{12})\cap W^u(a_{23})\cap W^u(a_{31})$. Then, the fact that 
$\int_M U_{a_{12}}\wedge U_{a_{23}}\wedge U_{31}$
is an integer will follow from Lemma~\ref{l:intersection-current}. Let us start by showing that any point in the intersection 
$\overline{W^u(a_{12})}\cap \overline{W^u(a_{23})}\cap \overline{W^u(a_{31})}$ must belong to $W^u(a_{12})\cap W^u(a_{23})\cap 
 W^u(a_{31})$. We proceed by contradiction. 
Suppose that $x$ belongs to $\overline{W^u(a_{12})}\cap \overline{W^u(a_{23})}\cap \overline{W^u(a_{31})}$ but not 
 to $W^u(a_{12})\cap W^u(a_{23})\cap W^u(a_{31})$. From Smale's Theorem~\ref{t:smale}, it means that there exists critical points $b_1$, $b_2$ and 
 $b_3$ such that $x$ belongs to $W^u(b_1)\cap W^u(b_2)\cap W^u(b_3)$ with $W^u(b_i)\subset\overline{W^u(a_{ij})}$ and 
 \emph{at least 
 one} of the $i$ verifies $\text{dim}(W^u(b_i))<\text{dim}(W^u(a_{ij}))$. This implies that
 $$\text{dim}(W^u(b_1))+\text{dim}(W^u(b_2))+\text{dim}(W^u(b_3))<2n.$$
Then, using our transversality assumption, we have that 
$$\text{dim} (W^u(b_1) \cap W^u(b_2)\cap W^u(b_3))=
\text{dim} (W^u(b_1) \cap  W^u(b_2))+\text{dim} ( W^u(b_3))-n.$$
Using transversality one more time, we get that
$$\text{dim} (W^u(b_1) \cap W^u(b_2)\cap W^u(b_3))=
\text{dim} (W^u(b_1))+\text{dim}(  W^u(b_2))+\text{dim} (W^u(b_3))-2n<0,$$
which contradicts the fact that $W^u(b_1) \cap W^u(b_2)\cap W^u(b_3)$ is not empty. 
Hence, we have already shown that
$$\overline{W^u(a_{12})}\cap \overline{W^u(a_{23})}\cap \overline{W^u(a_{31})}=W^u(a_{12})\cap W^u(a_{23})\cap W^u(a_{31}),$$
and it remains to show that this intersection consists of finitely many points. For that purpose, observe that as the intersection 
$W^u(a_{12})\cap W^u(a_{23})\cap W^u(a_{31})$ is transverse, it defines a $0$-dimensional submanifold of $M$. Thus, $x$ belonging to 
$W^u(a_{12})\cap W^u(a_{23})\cap W^u(a_{31})$ is an isolated point  
inside $W^u(a_{ij})$ for the induced topology by the embedding of $W^u(a_{ij})$ in $M$, for every $ij$ in 
$\{12, 23, 31\}$. Moreover, as this intersection coincides with its closure, 
we can deduce there can only be finitely many points in it and this concludes the proof of Theorem~\ref{t:fukaya}.

\subsection{Morse gradient trees.}\label{r:geometry-fukaya} Now that we have shown that the limit in Theorem~\ref{t:fukaya} is an integer, let 
us give a geometric interpretation to this integer in terms of counting gradient flow trees. 
From the above proof, we count with orientation the number of points 
in $W^u(a_{12})\cap W^u(a_{23})\cap W^u(a_{31})$. In dynamical terms, such a point $x_0$ 
corresponds to the intersection of three flow lines starting from $a_{12}$, $a_{23}$ and $a_{31}$ and passing through $x_0$. This represents a one dimensional 
submanifold having the form of a Y shaped tree whose edges are gradient lines. Hence, the integral
$\int_M U_{a_{12}}\wedge U_{a_{23}}\wedge U_{a_{31}}$
counts the number of such Y shaped gradient tree given by a triple of Morse-Smale gradient flows.

\subsection{Cup products}\label{ss:fukaya}

These triple products can be interpreted in terms of the cup-products appearing in Morse theory~\cite{Fu93, Fu97}. Indeed, we can define a bilinear map representing the cup product as follows:
$$\mathfrak{m}_2^{(k,l)}:\text{Ker}(-\ml{L}_{V_{f_{12}}}^{(k)})\times \text{Ker}(-\ml{L}_{V_{f_{23}}}^{(l)})\rightarrow\text{Ker}(-\ml{L}_{V_{f_{13}}}^{(k+l)})\simeq \text{Ker}(-\ml{L}_{V_{f_{31}}}^{(n-k+l)})^{\dagger},$$
whose matrix coefficients in the basis $(U_{a_{12}}, U_{a_{23}}, U_{a_{13}})=(U_{a_{12}}, U_{a_{23}}, S_{a_{31}})$ are given by
$\int_M U_{a_{12}}\wedge U_{a_{23}}\wedge U_{a_{31}}\in\IZ.$
Note that, compared with the classical theory where these maps are defined in an algebraic manner~\cite{Fu97}, our formulation is of purely analytical nature. The spaces (which are 
made of currents) correspond to the kernel of appropriate operators, and the matrix elements are defined in terms of the classical wedge product on currents. Using this analytical 
formulation, we can rapidly verify that $\mathfrak{m}_2$ induces a chain map for the Morse complexes\footnote{Recall from~\cite{dang2016spectral, DaRi17c} that 
this complex is quasi-isomorphic to the De Rham complex $(\Omega(M),d)$ via the spectral projector associated with the eigenvalue $0$.} $(\text{Ker}(-\ml{L}_{V_{f}}),d)$. Indeed, we fix $(U_1, U_2)$ in 
$\text{Ker}(-\ml{L}_{V_{f_{12}}}^{(k)})\times \text{Ker}(-\ml{L}_{V_{f_{23}}}^{(l)})$, and we write, using the Stokes formula,
$$ \mathfrak{m}_2^{(k+1,l)}(dU_1,U_2)+(-1)^k\mathfrak{m}_2^{(k,l+1)}(U_1,dU_2) = (-1)^{k+l+1}\sum_{a_{13}\in\text{Crit}(f_{13})} \left(\int_M U_1\wedge U_2\wedge d U_{a_{31}}\right)U_{a_{13}}.
$$
Then, we can make use of the fact that we have a complex:
\begin{eqnarray*}
d U_{a_{31}}& = &\sum_{b_{31}:\text{ind}(b_{31}) = \text{ind}(a_{31})+1}\left(\int_{M} S_{b_{31}}\wedge d U_{a_{31}}\right) U_{b_{31}}\\
 & = & (-1)^{k+l+1}\sum_{b_{13}:\text{ind}(b_{13})+1=\text{ind}(a_{13})}\left(\int_{M} d U_{b_{13}}\wedge  S_{a_{13}}\right) U_{b_{31}}, 
\end{eqnarray*}
where we used the Stokes formula one more time to write the second equality. Intertwining the sums over $a_{13}$ and $b_{13}$ yields
$$ \mathfrak{m}_2^{(k+1,l)}(dU_1,U_2)+(-1)^k\mathfrak{m}_2^{(k,l+1)}(U_1,dU_2) = \sum_{b_{13}\in\text{Crit}(f_{13})} \left(\int_M U_1\wedge U_2\wedge U_{b_{31}}\right)dU_{b_{13}},
$$
or equivalently
$$\mathfrak{m}_2^{(k+1,l)}(dU_1,U_2)+(-1)^k\mathfrak{m}_2^{(k,l+1)}(U_1,dU_2)= d\left(\mathfrak{m}_2^{(k,l)}(U_1,U_2)\right).$$
This relation shows that $\mathfrak{m}_2$ is a 
cochain map for the Morse complexes 
$(\text{Ker}(-\ml{L}_{V_{f_{ij}}}),d)$, hence induces a 
cup-product on the Morse 
cohomologies. In other terms, this map $\mathfrak{m}_2$ is a 
(spectral) realization in terms of currents of the algebraic cup product coming from Morse theory~\cite{Fu97}.

Fukaya's conjecture states that, up to some exponential factors involving the Liouville period over certain triangles defined by Lagrangian submanifolds, this 
algebraic cup product can be recovered by computing triple products of Witten quasimodes~\cite[Conj.~4.1]{Fu05}. To summarize this section, by giving this analytical 
interpretation of the Morse cup-product, we have been able to obtain Fukaya's instanton formula by considering the limit $\hbar\rightarrow 0^+$ in appropriate 
Sobolev spaces where both $-\ml{L}_{V_f}$ and $W_{f,\hbar}$ have nice spectral properties.

\begin{rema} Note that $\mathfrak{m}_1=d$ and 
$\mathfrak{m}_2=\wedge$ are the two first operations of
the Morse $A_{\infty}$--category discovered by Fukaya~\cite{Fu93, Fu97}. 
Our analysis shows that these algebraic maps can be interpreted 
in terms of analysis 
as Witten deformations of the coboundary operator and exterior products. 
An $A_{\infty}$--category is in fact endowed with graded maps $(\mathfrak{m}_k)_{k\geq 1}$ of algebraic nature, and it is natural to think that all these algebraic 
maps can also be given analytic interpretations by considering appropriate Witten deformations which is the content of
Fukaya's general conjectures~\cite[Conj.~4.2]{Fu05}. 
However, this is at the expense of 
a more subtle combinatorial work and we shall discuss this issue elsewhere. 
\end{rema}

\section{Comparison with the Helffer-Sj\"ostrand quasimodes}\label{s:helffer-sjostrand}

In~\cite[Eq.~(1.37)]{HeSj85}, Helffer and Sj\"ostrand also constructed a natural basis for the bottom of the spectrum of the Witten Laplacian. 
For the sake of completeness\footnote{We note that, 
except for this section, our results are self-contained and they do not rely on~\cite{HeSj85}.}, we will 
compare our family of quasimodes with their one and show that they are equal at leading order. In order to apply the results of~\cite{HeSj85}, we remark that 
the dynamical assumptions~(H1) and~(H2) from this reference are automatically satisfied as soon as the gradient 
flow verifies the Smale transversality assumption. For (H1), this follows from Smale's Theorem~\ref{t:smale} while~(H2) was for instance proved 
in~\cite[Prop.~3.6]{Web06}.

We denote the Helffer-Sj\"ostrand's quasimodes by $(U_a^{HS}(\hbar))_{a\in\text{Crit}(f)}$. By construction, they belong to the same 
eigenspaces as our quasimodes $(U_a(\hbar))_{a\in\text{Crit}(f)}.$ Fix a critical point $a$ of index $k$. 
These quasimodes do not form an orthonormal family. Yet, if $V^{(k)}(\hbar)$ is the matrix whose coefficients are given by $\la U_b^{HS}(\hbar), U_{b'}^{HS}(\hbar)\ra_{L^2}$, then one 
knows from~\cite[Eq.~(1.43)]{HeSj85} that
$$V^{(k)}(\hbar)=\text{Id}+\ml{O}(e^{-C_0/\hbar}),$$ 
for some positive constant $C_0>0$ depending only on $(f,g)$. Hence, if we transform this family into an orthonormal family 
$(\tilde{U}_b^{HS}(\hbar))_{b:\text{Ind}(b)=k}$, then one has
$$\tilde{U}_b^{HS}(\hbar)=\sum_{b':\text{Ind}(b')=k}\left(\delta_{bb'}+\ml{O}(e^{-C_0/\hbar})\right)U_{b'}^{HS}(\hbar).$$
In particular, the spectral projector can be written as
$$\mathbf{1}_{[0,\eps]}\left(W_{f,\hbar}^{(k)}\right)(x,y,dx,dy)=
\sum_{b\in\text{Crit}(f):\text{Ind}(b)=k}\tilde{U}_b^{HS}(\hbar)(x,dx)\tilde{U}_b^{HS}(\hbar)(y,dy).$$
Hence, from the definition of our WKB state $U_a(\hbar)$, one has
$$U_a(\hbar)=\sum_{b\in\text{Crit}(f):\text{Ind}(b)=k}\int_M U_a\wedge 
\star_k\left(e^{-\frac{f-f(a)}{\hbar}} \tilde{U}_b^{HS}(\hbar)\right)\tilde{U}_b^{HS}(\hbar),$$
which can be expanded as follows:
$$U_a(\hbar)=\sum_{b\in\text{Crit}(f):\text{Ind}(b)=k}\sum_{b':\text{Ind}(b')=k}\left(\delta_{bb'}+\ml{O}(e^{-C_0/\hbar})\right)
\left(\int_M U_a\wedge \star_k\left(e^{-\frac{f-f(a)}{\hbar}} U_{b'}^{HS}(\hbar)\right)\right)\tilde{U}_b^{HS}(\hbar).$$
Everything now boils down to the calculation of
$$\alpha_{ab}(\hbar)=\int_M U_a\wedge \star_k\left(e^{-\frac{f-f(a)}{\hbar}} U_b^{HS}(\hbar)\right).$$
More precisely, if we are able to prove that
\begin{equation}\label{e:comparisonHS}
 \alpha_{ab}(\hbar)=\delta_{ab}\alpha_a(\hbar)(1+\ml{O}(\hbar))+\ml{O}(e^{-C_0/\hbar}),
\end{equation}
for a certain $\alpha_a(\hbar)\neq 0$ depending polynomially on $\hbar$ (that has to be determined), then, after gathering all the equalities, we will find that
\begin{equation}\label{e:DRvsHS}U_a(\hbar)=\alpha_a(\hbar)(1+\ml{O}(\hbar))U_a^{HS}(\hbar)+\sum_{b\neq a\in\text{Crit}(f):\text{Ind}(b)=\text{Ind}(a)}
\ml{O}(e^{-C_0/\hbar})U_b^{HS}(\hbar),\end{equation}
showing that our quasimodes are at leading order equal to the ones of Helffer and Sj\"ostrand (up to some normalization factor). 
Let us now prove~\eqref{e:comparisonHS} by making use of the results from~\cite{HeSj85}. First of all, we write that
$$\alpha_{ab}(\hbar)=\int_M U_a\wedge \star_k\left(e^{-\frac{f-f(a)}{\hbar}} U_b^{HS}(\hbar)\right).$$
According to~\cite[Eq.~(1.38)]{HeSj85}, we know that 
$$U_b^{HS}(\hbar)=\Psi_b(\hbar)+\ml{O}_{\Omega^k(M)}(e^{-\frac{C_0}{\hbar}}),$$
where $\Psi_b(\hbar)$ is a certain ``Gaussian state'' centered at $b$ defined by~\cite[Eq.~(1.35)]{HeSj85} and $C_0$ is some positive constant. 
Thus, as $f(x)\geq f(a)$ on the support of $U_a$, we have
$$\alpha_{ab}(\hbar)=\int_M U_a\wedge \star_k\left(e^{-\frac{f-f(a)}{\hbar}} \Psi_b(\hbar)\right)+\ml{O}(e^{-\frac{C_0}{\hbar}}).$$
We now introduce a smooth cutoff function $\chi_a$ which is equal to $1$ in a neighborhood of $a$ and we write
\begin{eqnarray*}\alpha_{ab}(\hbar) & = &\int_M U_a\wedge \star_k\left(\chi_a e^{-\frac{f-f(a)}{\hbar}} \Psi_b(\hbar)\right)\\
&+ &\int_M U_a\wedge \star_k\left((1-\chi_a) e^{-\frac{f-f(a)}{\hbar}} \Psi_b(\hbar)\right)+\ml{O}(e^{-\frac{C_0}{\hbar}}).
\end{eqnarray*}
Thanks to~\cite[Th.~1.4]{HeSj85} and to the fact that the support of $U_a$ is equal to $\overline{W^u(a)}$, we know that the second term, which corresponds to the points which are far from $a$, is also exponentially small. Hence
$$\alpha_{ab}(\hbar)=\int_M U_a\wedge \star_k\left(e^{-\frac{f-f(a)}{\hbar}}\chi_a \Psi_b(\hbar)\right)+\ml{O}(e^{-\frac{C_0}{\hbar}}).$$
Thanks to Lemma~\ref{r:current}, this can be rewritten as
$$\alpha_{ab}(\hbar)=\int_{W^u(a)} \star_k\left(e^{-\frac{f-f(a)}{\hbar}}\chi_a \Psi_b(\hbar)\right)+\ml{O}(e^{-\frac{C_0}{\hbar}}).$$
Using~\cite[Th.~1.4]{HeSj85}, we find that, for $a\neq b$, one has
$$\alpha_{ab}(\hbar)=\ml{O}(e^{-\frac{C_0}{\hbar}}).$$
It remains to treat the case $a=b$. In that case, we can use~\cite[Th.~1.4 and Th.~2.5]{HeSj85} to show 
$$\alpha_{ab}(\hbar)=\alpha_a(\pi \hbar)^{\frac{n-2k}{4}}(1+\ml{O}(\hbar)),$$
for a certain positive constant $\alpha_a\neq 0$ which depends only on the Lyapunov exponents at the critical point $a$ (and not on $\hbar$). Precisely, one has
$$\left|\alpha_a\right|=\left(\frac{\prod_{j=1}^k|\chi_j(a)|}{\prod_{j=k+1}^n|\chi_j(a)|}\right)^{\frac{1}{4}}.$$
This shows that our eigenmodes are not a priori normalized in $L^2$. To fix this, we would need to set, for every critical point $a$ of $f$,
$$\mathbf{U}_a(\hbar):=\frac{1}{|\alpha_a|(\pi\hbar)^{\frac{n-2k}{4}}}U_a(\hbar).$$
With this renormalization, the tunneling formula of Theorem~\ref{t:maintheo-tunneling} can be rewritten as
$$\hbar d_{f,\hbar}\mathbf{U}_a(\hbar)=\left(\frac{\hbar}{\pi}\right)^{\frac{1}{2}}
\sum_{b:\operatorname{ind}(b)=\operatorname{ind}(a)+1}n_{ab}\left(\frac{e^{\frac{f(a)}{\hbar}}}{|\alpha_a|}\right)\left(\frac{e^{\frac{f(b)}{\hbar}}}{|\alpha_b|}\right)^{-1}
\mathbf{U}_{b}(\hbar).$$
Under this form, we now recognize exactly the tunneling formula as it appears in~\cite[Eq.~(3.27)]{HeSj85}. Concerning the Fukaya's instanton formula, we observe 
that it can be rewritten as
$$\lim_{\hbar\rightarrow 0^+} \frac{|\alpha_{a_{12}}\alpha_{a_{23}}\alpha_{a_{31}}|(\pi\hbar)^{\frac{n}{4}}}{e^{\frac{f_{12}(a_{12})+f_{23}(a_{23})+f_{31}(a_{31})}{\hbar}}}
 \int_M \mathbf{U}_{a_{12}}(\hbar)\wedge \mathbf{U}_{a_{23}}(\hbar)\wedge \mathbf{U}_{a_{31}}(\hbar)=
\int_MU_{a_{12}}\wedge U_{a_{23}}\wedge U_{a_{31}}.$$

\appendix

\section{Order functions}\label{a:order-function}

In~\cite{dang2016spectral, dangrivieremorsesmale1}, one of the key difficulty is the construction of an order function adapted to the 
Morse-Smale dynamics induced by the flow $\varphi_f^t$. Here, we recall some of the properties proved in that reference and we also 
recall along the way some properties of Morse-Smale gradient flows. We refer to~\cite{Web06} for a detailed introduction on that topic.

\subsection{Stable and unstable manifolds}

Similarly to the unstable manifold $W^u(a)$, we can define, for every $a\in\text{Crit}(f)$,
$$W^s(a):=\left\{x\in M: \lim_{t\rightarrow+\infty}\varphi_f^t(x)=a\right\}.$$
A remarkable property of gradient flows is that, given any $x$ in $M$, there exists an unique $(a,b)$ in $\text{Crit}(f)^2$ such that $f(a)\leq f(b)$ 
and
$$x\in W^u(a)\cap W^s(b).$$
Equivalently, the unstable manifolds form a partition of $M$. It is known from the works of Smale~\cite{Sm60} that 
these submanifolds are embedded inside $M$~\cite[p.~134]{Web06} 
and that their dimension is equal to $n-r(a)$ where $r(a)$ is the Morse index of $a$. The Smale transversality assumption is the requirement that, given any $x$ 
in $M$, one has
$$T_xM=T_xW^u(a)+T_xW^s(b).$$
Equivalently, it says that the intersection of
$$\Gamma_+=\Gamma_+(V_f):=\bigcup_{a\in\text{Crit}(f)}N^*(W^s(a))\quad\text{and}\quad\Gamma_-=\Gamma_-(V_f):=\bigcup_{a\in\text{Crit}(f)}N^*(W^u(a))$$
is empty, where $N^*(\ml{W})\subset T^*M\backslash 0$ denotes the conormal of the manifold $\ml{W}$. In the proofs of section~\ref{s:anisotropic}, an important role 
is played by the Hamiltonian vector field generated by
$$H_f(x;\xi):=\xi(V_f(x)).$$
Recall that the corresponding Hamiltonian flow can be written
$$\Phi_f^t(x;\xi):=\left(\varphi_f^t(x),(d\varphi^t(x)^T)^{-1}\xi\right),$$
and that it induces a flow on the unit cotangent bundle $S^*M$ by setting
$$\tilde{\Phi}_f^t(x;\xi):=\left(\varphi_f^t(x),\frac{(d\varphi^t(x)^T)^{-1}\xi}{\left\|(d\varphi^t(x)^T)^{-1}\xi\right\|_{g^*\circ\varphi^t(x)}}\right).$$
The corresponding vector field are denoted by $X_{H_f}$ and $\tilde{X}_{H_f}$.

\subsection{Escape function}

In all this paragraph, $V_f$ satisfies the assumption of Theorem~\ref{t:maintheo-harveylawsonwitten}. We recall the following result~\cite[Lemma~2.1]{FS}:
\begin{lemm}\label{l:faure-sjostrand} Let $V^u$ and $V^s$ be small open neighborhoods of $\Gamma_+\cap S^*M$ and 
$\Gamma_-\cap S^*M$ respectively, and let $\eps>0$. Then, 
there exist $\ml{W}^u\subset V^u$ and $\ml{W}^s\subset V^s$, $\tilde{m}$ in $\ml{C}^{\infty}(S^*M,[0,1])$, $\eta>0$ such that 
$\tilde{X}_{H_f}.\tilde{m}\geq 0$ on $S^*M$, 
$\tilde{X}_{H_f}.\tilde{m}\geq\eta>0$ on $S^*M-(\ml{W}^u\cup\ml{W}^s)$, 
$\tilde{m}(x;\xi)>1-\epsilon$ for $(x;\xi)\in \ml{W}^s$ and $\tilde{m}(x;\xi)<\eps$ for $(x;\xi)\in \ml{W}^u$. 
\end{lemm}
This Lemma was proved by Faure and Sj\"ostrand in~\cite{FS} in the case of Anosov flows and its extension to gradient 
flows require some results on the Hamiltonian dynamics that were obtained in~\cite[Sect.~3]{dang2016spectral} -- see 
also~\cite[Sect.~4]{dangrivieremorsesmale1} in the more general framework of Morse-Smale flows.

As we have a function $\tilde{m}(x;\xi)$ defined on $S^*M$, we introduce a smooth function $m$ defined on $T^*M$ which satisfies
$$m(x;\xi)=N_1\tilde{m}\left(x,\frac{\xi}{\|\xi\|_x}\right)-N_0\left(1-\tilde{m}\left(x,\frac{\xi}{\|\xi\|_x}\right)\right),\ \text{for}\ \|\xi\|_x\geq 1,$$
and
$$m(x;\xi)=0,\ \text{for}\ \|\xi\|_x\leq \frac{1}{2}.$$
We set the order function of our escape function to be
$$m_{N_0,N_1}(x;\xi)=-f(x)+m(x;\xi).$$
It was shown in~\cite[Lemma~4.1]{dang2016spectral} that it satisfies the following properties (for $V^u$, $V^s$ and $\eps>0$ small enough\footnote{In particular, 
$V^u\cap V^s=\emptyset.$}):
\begin{lemm}[Escape function]\label{l:escape-function} Let $s\in\IR$ and $N_0,N_1>4(\|f\|_{\ml{C}^0}+|s|)$ be two elements in $\IR$. Then, there exist 
$c_0>0$ (depending on $(M,g)$ but not on $s$, $N_0$ and $N_1$) such that $m_{N_0,N_1}(x;\xi)+s$  
\begin{itemize}
 \item takes values in $[-2N_0,2N_1]$,
 \item is $0$ homogeneous for $\|\xi\|_x\geq 1$,
 \item is $\leq -\frac{N_0}{2}$ on a conic neighborhood of $\Gamma_-$ (for $\|\xi\|_x\geq 1$),
 \item is $\geq \frac{N_1}{2}$ on a conic neighborhood of $\Gamma_+$ (for $\|\xi\|_x\geq 1$),
\end{itemize}
and such that there exists $R_0>0$ for which the escape function
$$G_{N_0,N_1}^s(x;\xi):=(m_{N_0,N_1}(x;\xi)+s)\log(1+\|\xi\|_x^2)$$
verifies, for every $(x;\xi)$ in $T^*M$ with $\|\xi\|_x\geq R_0$,
$$X_{H_f}.(G_{N_0,N_1}^s)(x;\xi)\leq -C_N:=-c_0\min\{N_0,N_1\}.$$
\end{lemm}

\section{Holomorphic continuation of the Ruelle determinant}\label{a:holomorphic}

In this appendix, we consider a Morse-Smale gradient flow $\varphi_f^t$. We fix $0\leq k\leq n$ and $a\in\text{Crit}(f)$. 
We recall how to prove that the local Ruelle determinant
$$\zeta_{R,a}^{(k)}(z):=\exp\left(-\sum_{l=1}^{+\infty}\frac{e^{-lz}}{l}
\frac{\text{Tr}\left(\Lambda^k\left(d\varphi_f^{-l}(a)\right)\right)}{\left|\text{det}\left(\text{Id}-d\varphi_f^{-l}(a)\right)\right|}\right)$$
has an holomorphic extension to $\IC$, and we compute explicitely its zeros in terms of the Lyapunov exponents $(\chi_j(a))_{1\leq j\leq n}$. Recall that 
the dynamical Ruelle determinant from the introduction is given by
$$\zeta_{R}^{(k)}(z)=\prod_{a\in\text{Crit}(f)}\zeta_{R,a}^{(k)}(z).$$
By definition of the Lyapunov exponents, we also recall that $d\varphi_f^{-1}(a)=\exp(-L_f(a))$ where $L_{f}(a)$ is a symmetric matrix 
whose eigenvalues are given by the $(\chi_j(a))_{1\leq j\leq n}$. If $a$ is of index $r$, we used the convention:
$$\chi_1(a)\leq \ldots\leq \chi_r(a)<0<\chi_{r+1}(a)\leq\ldots\leq\chi_n(a).$$
In order to show this holomorphic continuation, we start by observing that, in terms of the Lyapunov exponents,
\begin{eqnarray*}\left|\text{det}\left(\text{Id}-d\varphi_f^{-l}(a)\right)\right|^{-1} & = & \prod_{j=1}^{r}(e^{-l\chi_j(a)}-1)^{-1}
\prod_{j=r+1}^n(1-e^{-l\chi_j(a)})^{-1}\\
  &= &e^{l\sum_{j=1}^r\chi_j(a)}\prod_{j=1}^n(1-e^{-l|\chi_j(a)|})^{-1}\\
  &= &e^{l\sum_{j=1}^r\chi_j(a)}\sum_{\alpha\in\IN^n}e^{-l\alpha.|\chi(a)|},
\end{eqnarray*}
where $\IN=\{0,1,2,\ldots\}$ and $|\chi(a)|=(|\chi_j(a)|)_{1\leq j\leq n}$. We now compute the trace
$$\text{Tr}\left(\Lambda^k\left(d\varphi_f^{-l}(a)\right)\right)=\sum_{J\subset\{1,\ldots,n\}:|J|=k}\exp\left(-l\sum_{j\in J}\chi_j(a)\right),$$
which implies that
$$\frac{\text{Tr}\left(\Lambda^k\left(d\varphi_f^{-l}(a)\right)\right)}{\left|\text{det}\left(\text{Id}-d\varphi_f^{-l}(a)\right)\right|}$$
is equal to
$$\sum_{J\subset\{1,\ldots,n\}:|J|=k}\sum_{\alpha\in\IN^n}\exp\left(-l\left(\sum_{j\in J\cap \{r+1,\ldots,n\}}|\chi_j(a)|+
\sum_{j\in J^c\cap\{1,\ldots,r\}}|\chi_j(a)|
+\alpha.|\chi(a)|\right)\right).$$
Under this form, one can verify that $\zeta_{R,a}^{(k)}(z)$ has an holomorphic extension to $\IC$ whose zeros are given by the set
$$\ml{R}_k(a):=\left\{-\sum_{j\in J\cap \{r+1,\ldots,n\}}|\chi_j(a)|-
\sum_{j\in J^c\cap\{1,\ldots,r\}}|\chi_j(a)|
-\alpha.|\chi(a)|: |J|=k\ \text{and}\ \alpha\in\IN^n\right\}.$$
Moreover, the multiplicity of $z_0$ in $\ml{R}_k(a)$ is given by the number of couples $(\alpha,J)$ such that
$$z_0=-\left(\sum_{j\in J\cap \{r+1,\ldots,n\}}|\chi_j(a)|+
\sum_{j\in J^c\cap\{1,\ldots,r\}}|\chi_j(a)|
+\alpha.|\chi(a)|\right).$$
\begin{rema} In particular, we note that $z_0=0$ is a zero of $\zeta_{R,a}^{(k)}(z)$ if and only if the index of $a$ (meaning the dimension 
of $W^s(a)$) is equal to $k$. In that case, the zero is of multiplicity $1$. This implies that the multiplicity of $0$ as a zero 
of $\zeta_{R}^{(k)}(z)$ is equal to the number of critical points of index $k$. 
\end{rema}

\section{Proof of Lemma~\ref{l:L2vsinfty}}
\label{a:lemma}

In this appendix, we give the proof of Lemma~\ref{l:L2vsinfty}. Up to minor modifications due to the fact that we are dealing with $L^2$ norms, 
we follow the lines of~\cite[p.~58]{dangthesis}. We fix $N$, $\tilde{N}$, $W_0$ and $W$ as in the statement of this Lemma.

The cone $W_0$ being given, we can choose $W$ to be a thickening of the cone $W_0$, i.e.
$$W=\left\{\eta\in\mathbb{R}^n\setminus \{0\} |\exists\xi\in V,  \left\vert\frac{\xi}{\vert \xi\vert}- \frac{\eta}{\vert \eta\vert}\right\vert \leqslant \delta \right\},$$
for some fixed positive $\delta$. This means that small angular perturbations of covectors in $W_0$ will lie in the neighborhood $W$. 
Choose some smooth compactly supported function $\varphi$ which equals
$1$ on the support of $u$ hence we have the identity
$\widehat{u}=\widehat{u\varphi}$. 
We compute the Fourier transform
of the product:
$$\vert\widehat{u\varphi}(\xi)\vert\leqslant \int_{\mathbb{R}^n} \vert \widehat{\varphi}(\xi-\eta)    \widehat{u} (\eta) \vert d\eta .$$
We reduce to the estimate $$ \int_{\mathbb{R}^n} \vert \widehat{\varphi}(\xi-\eta)    \widehat{u} (\eta) \vert d\eta = 
\underset{I_1(\xi)}{\underbrace{\int_{ \vert\frac{\xi}{\vert \xi\vert}- \frac{\eta}{\vert \eta\vert}\vert \leqslant \delta  } \vert \widehat{\varphi}(\xi-\eta)    \widehat{u} (\eta) \vert d\eta}} + \underset{I_2(\xi)}{\underbrace{\int_{ \vert\frac{\xi}{\vert \xi\vert}- \frac{\eta}{\vert \eta\vert}\vert \geqslant \delta  }\vert \widehat{\varphi}(\xi-\eta)    \widehat{u} (\eta) \vert d\eta}},$$
and we will estimate separately the two terms $I_1(\xi),I_2(\xi)$.

Start with $I_1(\xi)$, if $\xi\in W_0$ then, by definition of $W$, $\eta$ belongs to $W$.
Hence, using the Cauchy-Shwarz inequality, this yields the estimate
\begin{eqnarray*}
I_1(\xi)&=&\int_{ \vert\frac{\xi}{\vert \xi\vert}- \frac{\eta}{\vert \eta\vert}\vert \leqslant \delta  } \vert\widehat{\varphi}(\xi-\eta)    \widehat{u} (\eta) \vert d\eta  \\
&= & (1+\vert\xi \vert)^{-N}
\int_{ \vert\frac{\xi}{\vert \xi\vert}- \frac{\eta}{\vert \eta\vert}\vert \leqslant \delta  } \vert\widehat{\varphi}(\xi-\eta) (1+\vert \xi-\eta\vert)^{N}     
\widehat{u} (\eta)(1+\vert\eta \vert)^N \vert \frac{ (1+\vert\xi \vert)^N}{(1+\vert\eta \vert)^N (1+\vert \xi-\eta\vert)^{N}} d\eta \\
&\leqslant &  (1+\vert\xi \vert)^{-N} \sup_{\xi,\eta} \frac{ (1+\vert\xi \vert)^N}{(1+\vert\eta \vert)^N (1+\vert \xi-\eta\vert)^{N}} \Vert \varphi \Vert_{H^N}\Vert (1+\vert\xi \vert)^N\widehat{u}(\xi)  \Vert_{L^2(W)}\\
&\leqslant & C_{\varphi,N} (1+\vert\xi \vert)^{-N}\Vert (1+\vert\xi \vert)^N\widehat{u}(\xi)\Vert_{L^2(W)}
\end{eqnarray*}
where we used the triangle inequality $|\xi|\leq |\xi-\eta|+|\eta|$ 
in order to bound $\frac{(1+\vert\xi \vert)^N}{(1+\vert\eta \vert)^N (1+\vert \xi-\eta\vert)^{N}}$ by some 
constant $C$ uniformly in $\xi$ and in $\eta$.

To estimate the second term $I_2(\xi)$, we shall use that the integral is over $\eta$ such that 
$\vert\frac{\xi}{\vert \xi\vert}- \frac{\eta}{\vert \eta\vert}\vert \geqslant \delta$. This implies that the angle 
between $\xi$ and $\eta$ is bounded from below by some $\alpha\in(0,\pi/2)$ which depends only on the aperture 
$\delta$. We now observe that
$$a^2+b^2-2ab\cos c=(a-b\cos c)^2+b^2\sin^2c\geq b^2\sin^2 c,$$
and we apply this lower bound to $a=|\xi|$, $b=|\eta|$ and $c$ the angle between $\xi$ and $\eta$.
Thus, 
$$\forall (\xi,\eta)\in \left(V\times ^cW\right),
\vert( \sin\alpha) \eta\vert  \leqslant \vert \xi-\eta\vert, \vert(\sin\alpha) \xi\vert  \leqslant \vert \xi-\eta\vert .$$
Then, for such $\xi$ and $\eta$, there exists some constant $C$ (depending only on $N$, $\tilde{N}$ and $\delta$) such that 
$$(1+|\xi-\eta|)^{-N-\tilde{N}}\leq(1+\vert(\sin\alpha) \eta\vert )^{-\tilde{N}} (1+ \vert(\sin\alpha) \xi\vert)^{-N} \leq C(1+\vert \eta\vert )^{-\tilde{N}} (1+ \vert \xi\vert)^{-N}.$$
Thus, up to increasing the value of $C$ and by applying the Cauchy-Schwarz inequality, we find
\begin{eqnarray*}
\int_{ \vert\frac{\xi}{\vert \xi\vert}- \frac{\eta}{\vert \eta\vert}\vert \geqslant \delta  } \vert \widehat{\varphi}(\xi-\eta)    \widehat{u} (\eta) \vert d\eta
&\leqslant &C \|\varphi\|_{H^{N+\tilde{N}}} (1+ \vert\xi\vert)^{-N} \left(\int_{\mathbb{R}^n}(1+\vert\eta\vert )^{-2\tilde{N}} \vert\widehat{u} (\eta) \vert^2 d\eta\right)^{1/2}\\
\end{eqnarray*}
Gathering the two estimates yields the final result.

\end{document}